\documentclass[10pt]{article}
\usepackage[hmargin=1in,vmargin=1in]{geometry} 
\usepackage{amsmath,amsthm,textcomp,amssymb,scrextend, bm, graphicx,stmaryrd,euscript,leftindex,mathalfa,mathtools,xcolor}
\usepackage[citestyle=numeric,backend=biber,sorting=nyt,url=false]{biblatex}
\usepackage{listings}
\usepackage{tikz-cd}
\usetikzlibrary{arrows}
\DeclareFontFamily{OT1}{pzc}{} \DeclareFontShape{OT1}{pzc}{m}{it}% 
{<-> s * [1.15] pzcmi7t}{} \DeclareMathAlphabet{\mathpzc}{OT1}{pzc}{m}{it}
\usepackage[T1]{fontenc}
\DeclareMathAlphabet{\dutchcal}{U}{dutchcal}{m}{n}
\SetMathAlphabet{\dutchcal}{bold}{U}{dutchcal}{b}{n}

\usetikzlibrary{decorations.pathreplacing, calc}

\tikzset{
  altstackar/.style={decorate, decoration={show path construction,
    lineto code={
      \path (\tikzinputsegmentfirst); \pgfgetlastxy{\xstart}{\ystart}
      \path (\tikzinputsegmentlast); \pgfgetlastxy{\xend}{\yend}
      \path ($(0,0)!4pt!(\ystart-\yend,\xend-\xstart)$); \pgfgetlastxy{\xperp}{\yperp}
      \foreach \n[evaluate=\n as \k using 0.5*#1-\n+0.5] in {1,...,#1}{
        \ifodd\n{\draw[->, shorten <=4pt, shift={($\k*(\xperp,\yperp)$)}](\xstart+3,\ystart)--(\xend-3,\yend);}
        \else{\draw[<-, shorten >=4pt, shift={($\k*(\xperp,\yperp)$)}](\xstart+3,\ystart)--(\xend-3,\yend);}\fi
      }
    }
  }}, altstackar/.default={1}
}

\renewcommand{\refname}{References \vspace{-0.6\baselineskip}}

\usepackage{bookmark}
\bookmarksetup{
  numbered, 
  open,
}

\definecolor{love}{RGB}{128, 15, 37}

\usepackage{hyperref}
\hypersetup{
  colorlinks=true, %set true if you want colored links
  linktoc=all,     %set to all if you want both sections and sections linked
  linkcolor=love,  %choose some color if you want links to stand out
  citecolor = black,
  urlcolor = love,
  pdfencoding=unicode,
  hypertexnames = true,
  bookmarksdepth=subsection,
}

\usepackage[shortlabels]{enumitem}
\setlist[enumerate]{
     itemsep  = 0.2cm,
       label  = {\upshape (\arabic*)},
         ref  = \arabic*,
    leftmargin  = *
}

\newcommand{\enumref}[2]{(\hyperref[#1.#2]{\ref*{#1}.\ref*{#1.#2}})}

\newcommand\restr[2]{{% we make the whole thing an ordinary symbol
  \left.\kern-\nulldelimiterspace % automatically resize the bar with \right
  #1 % the function
  \littletaller % pretend it's a little taller at normal size
  \right|_{#2} % this is the delimiter
  }}

\newcommand{\littletaller}{\mathchoice{\vphantom{\big|}}{}{}{}}

\lstdefinestyle{myStyle}{
  language=SQL,
  basicstyle=\small\ttfamily,
  moredelim=[is][\underbar]{_}{_},
  keepspaces=true
}

\usepackage{mathrsfs}
\usepackage{fancyhdr}
\usepackage{gensymb}
\usepackage{polynom}
\usepackage[shortlabels]{enumitem}

\newcommand{\Ei}{{\mathbb{E}_{\infty}}}

\newcommand{\mS}{\mathcal S}

\newcommand{\C}{\mathcal C}

\newcommand{\scO}{\mathscr O}

\newcommand{\cB}{\mathcal{B}}

\newcommand{\cD}{\mathcal{D}}
\newcommand{\cE}{\mathcal{E}}

\newcommand{\cT}{\mathcal{T}}
\newcommand{\cU}{\mathcal{U}} 		% Notice this is different

\newcommand{\cX}{\mathcal{X}}
\newcommand{\cY}{\mathcal{Y}}

\newcommand{\ZZ}{\mathbb{Z}}

\newcommand{\sX}{\mathsf{X}}
\newcommand{\sY}{\mathsf{Y}}

\newcommand{\n}{\mathfrak{n}}

\def\ZZ{{\mathbf Z}}

\def\AA{{\mathbb A}}

\usepackage{citationscheme}

\usepackage{amsthm}

\usepackage[noabbrev,nameinlink,capitalize]{cleveref} 
\renewcommand{\qed}{\hfill$\square$}

\renewenvironment{proof}{\begin{addmargin}[1em]{0em}\begin{newproof}}{\end{newproof}\end{addmargin}\qed}

\newtheoremstyle{witharg}
  {} % ABOVESPACE
  {} % BELOWSPACE
  {\itshape} % BODYFONT
  {0pt} % INDENT (empty value is the same as 0pt)
  {\bfseries} % HEADFONT
  {} % HEADPUNCT
  {5pt plus 1pt minus 1pt} % HEADSPACE
  {\thmname{#1} \thmnumber{#2}\thmnote{ (#3)}\emph{.}}
\newcommand{\theoremarg}

\newtheoremstyle{witharg2}
  {} % ABOVESPACE
  {} % BELOWSPACE
  {} % BODYFONT
  {0pt} % INDENT (empty value is the same as 0pt)
  {\bfseries} % HEADFONT
  {} % HEADPUNCT
  {5pt plus 1pt minus 1pt} % HEADSPACE
  {\thmname{#1} \thmnumber{#2}\thmnote{ (#3)}\emph{.}}

\theoremstyle{witharg}

\newtheorem{theoreminner}{Theorem}[subsection]

\NewDocumentEnvironment{theorem}{mo}
 {%
  \renewcommand{\theoremarg}{#1}%
  \IfNoValueTF{#2}{\theoreminner}{\theoreminner[#2]}\ignorespaces
 }
 {\endtheoreminner}

\newtheorem{corollaryinner}[theoreminner]{Corollary}

\NewDocumentEnvironment{corollary}{mo}
 {%
  \renewcommand{\theoremarg}{#1}%
  \IfNoValueTF{#2}{\corollaryinner}{\corollaryinner[#2]}\ignorespaces
 }
 {\endcorollaryinner}
  
\theoremstyle{witharg}
\newtheorem{lemmainner}[theoreminner]{Lemma}

\NewDocumentEnvironment{lemma}{mo}
 {%
  \renewcommand{\theoremarg}{#1}%
  \IfNoValueTF{#2}{\lemmainner}{\lemmainner[#2]}\ignorespaces
 }
 {\endlemmainner}

\newtheorem{propositioninner}[theoreminner]{Proposition}

\NewDocumentEnvironment{proposition}{mo}
 {%
  \renewcommand{\theoremarg}{#1}%
  \IfNoValueTF{#2}{\propositioninner}{\propositioninner[#2]}\ignorespaces
 }
 {\endpropositioninner}

\newtheorem{motivatingquestioninner}[theoreminner]{Motivating Question}
\NewDocumentEnvironment{motivatingquestion}{mo}
 {%
  \renewcommand{\theoremarg}{#1}%
  \IfNoValueTF{#2}{\motivatingquestioninner}{\variantinner[#2]}\ignorespaces
 }
 {\endmotivatingquestioninner}

\theoremstyle{witharg2}

\newtheorem{remarkinner}[theoreminner]{Remark}

\NewDocumentEnvironment{remark}{mo}
 {%
  \renewcommand{\theoremarg}{#1}%
  \IfNoValueTF{#2}{\remarkinner}{\remarkinner[#2]}\ignorespaces
 }
 {\endremarkinner}

\newtheorem{constructioninner}[theoreminner]{Construction}

\NewDocumentEnvironment{construction}{mo}
 {%
  \renewcommand{\theoremarg}{#1}%
  \IfNoValueTF{#2}{\constructioninner}{\constructioninner[#2]}\ignorespaces
 }
 {\endconstructioninner}

\newtheorem{definitioninner}[theoreminner]{Definition}

\NewDocumentEnvironment{definition}{mo}
 {%
  \renewcommand{\theoremarg}{#1}%
  \IfNoValueTF{#2}{\definitioninner}{\definitioninner[#2]}\ignorespaces
 }
 {\enddefinitioninner}

 \newtheorem{roughdefinitioninner}[theoreminner]{Rough Definition}

\NewDocumentEnvironment{roughdefinition}{mo}
 {%
  \renewcommand{\theoremarg}{#1}%
  \IfNoValueTF{#2}{\roughdefinitioninner}{\definitioninner[#2]}\ignorespaces
 }
 {\endroughdefinitioninner}

\newtheorem{notationinner}[theoreminner]{Notation}

\NewDocumentEnvironment{notation}{mo}
 {%
  \renewcommand{\theoremarg}{#1}%
  \IfNoValueTF{#2}{\notationinner}{\notationinner[#2]}\ignorespaces
 }
 {\endnotationinner}

\newtheorem{rpropositioninner}[theoreminner]{Rough Proposition}

\NewDocumentEnvironment{rproposition}{mo}
 {%
  \renewcommand{\theoremarg}{#1}%
  \IfNoValueTF{#2}{\rpropositioninner}{\rpropositioninner[#2]}\ignorespaces
 }
 {\endrpropositioninner}

 \newtheorem{exampleinner}[theoreminner]{Example}

\NewDocumentEnvironment{example}{mo}
 {%
  \renewcommand{\theoremarg}{#1}%
  \IfNoValueTF{#2}{\exampleinner}{\exampleinner[#2]}\ignorespaces
 }
 {\endexampleinner}

 \newtheorem{variantinner}[theoreminner]{Variant}

\NewDocumentEnvironment{variant}{mo}
 {%
  \renewcommand{\theoremarg}{#1}%
  \IfNoValueTF{#2}{\variantinner}{\variantinner[#2]}\ignorespaces
 }
 {\endvariantinner}

\newtheorem{roughpropositioninner}[theoreminner]{Rough Proposition}

\NewDocumentEnvironment{roughproposition}{mo}
 {%
  \renewcommand{\theoremarg}{#1}%
  \IfNoValueTF{#2}{\roughpropositioninner}{\variantinner[#2]}\ignorespaces
 }
 {\endroughpropositioninner}

\newtheorem{potentialobstacleinner}[theoreminner]{Potential Obstacle}
\NewDocumentEnvironment{potentialobstacle}{mo}
 {%
  \renewcommand{\theoremarg}{#1}%
  \IfNoValueTF{#2}{\potentialobstacleinner}{\variantinner[#2]}\ignorespaces
 }
 {\endpotentialobstacleinner}

\newtheorem{philosophyinner}[theoreminner]{Philosophy}
\NewDocumentEnvironment{philosophy}{mo}
 {%
  \renewcommand{\theoremarg}{#1}%
  \IfNoValueTF{#2}{\philosophyinner}{\variantinner[#2]}\ignorespaces
 }
 {\endphilosophyinner}

\newtheorem{desirablepropertiesinner}[theoreminner]{Desirable Properties}
\NewDocumentEnvironment{desirableproperties}{mo}
 {%
  \renewcommand{\theoremarg}{#1}%
  \IfNoValueTF{#2}{\desirablepropertiesinner}{\variantinner[#2]}\ignorespaces
 }
 {\enddesirablepropertiesinner}

\DeclareCiteCommand{\citeq}
  {\usebibmacro{prenote}}
  {\usebibmacro{citeindex}%
   \usebibmacro{cite}}
  {\multicitedelim}
  {\usebibmacro{postnote}}

\DeclareFieldFormat{labelnumberwidth}{#1}
\DeclareFieldFormat{shorthandwidth}{#1}

\title{Geometric Categories and Sheaves on Topoi}
\author{Connor Bass}
\date{May 1, 2026}

\newcommand{\sss}{{\fontfamily{lmr}\selectfont\ss}}

\begin{document}

\maketitle

\begin{abstract}

    We introduce the notion of a geometric $(\infty,1)$-category, the protopyical example of which is an $(\infty,1)$-topos. We study (hyper)sheaves on geometric $(\infty,1)$-categories, proving that these are characterized by a form of \v{C}ech (hyper)descent. As an application we study (hyper)sheaves on $(n,1)$-topoi for all $n\in \ZZ_{\geq 1}\cup \{\infty\}$, and prove that the effective epimorphism topology on an $(n,1)$-topos $\cX$ may be identified as the canonical topology on $\cX$. Moreover, we show that for finite $n\in \ZZ_{\geq 1}$ the study of sheaves on an $(n,1)$-topos $\cX$ is equivalent to the study of $(n-1)$-truncated sheaves on certain $(\infty,1)$-topoi. We then globalize our study to consider sheaves on $\infty\cT op$. In the appendix, we study the behavior of modules under a reflective monoidal $(\infty,1)$-functor $L^\otimes:\C^{\otimes}\rightarrow \cD^{\otimes}$, and study (hyper)sheafification under a change of universe.
    
\end{abstract}

\vspace{-10pt}

\tableofcontents

\section{Introduction}

\subsection{Motivation and Approach}\label{Motivation}

In Higher Topos Theory (\citeq{highertopostheory}) 6.3.5, Lurie introduces the $(\infty,1)$-category $\textnormal{Shv}_{\widehat{\mS}}(\cX)$ of sheaves (of possibly large spaces) on an $(\infty,1)$-topos $\cX$, defined as the full subcategory of $\textnormal{Fun}(\cX^{\textnormal{opp}},\widehat{\mS})$ spanned by those maps $\cX^{\textnormal{opp}}\rightarrow \widehat{\mS}$ which preserve small limits. This is a very well-behaved notion of descent on $\cX$, leading to many well-developed higher geometric theories (foundational references for this include \citeq{lurie2009derivedalgebraicgeometryv} and \citeq{lurie2018sag}). Given the central importance of the $(\infty,1)$-topos (in a larger universe) $\textnormal{Shv}_{\widehat{\mS}}(\cX)$ to higher geometry, it is desirable to better understand the geometry of $\textnormal{Shv}_{\widehat{\mS}}(\cX)$ itself. In particular, consider the following folkloric result (which we prove in \ref{Truncated Objects}):

\begin{proposition}{}\label{n-truncated preserves small limits: Intro}

    (Compare to \ref{n-truncated preserves small limits: Appendix}.) Let $n\in \ZZ_{\geq 1}$, let $\cY$ be an $(n,1)$-topos, and let $\cD$ be an $(n,1)$-category. Then, a map $F:\cY^{\textnormal{opp}}\rightarrow \cD$ is a sheaf for the effective epimorphism topology on $\cY$ if and only if $F$ preserves small limits.

\end{proposition}

\noindent This suggests a meaningful relationship between $\textnormal{Shv}_{\widehat{\mS}}(\cX)$ and sheaves $\widehat{\textnormal{Shv}}_{\C an(\cX)}(\cX)$ for the effective epimorophism topology $\C an(\cX)$ on $\cX$. The na\"{i}ve guess of an identification $\textnormal{Shv}_{\widehat{\mS}}(\cX)=\widehat{\textnormal{Shv}}_{\C an(\cX)}(\cX)$ is decidedly \emph{false} (see \ref{counterexample of sheaf for subcanonical topology not preserving limits} for an explicit family of counterexamples), the discrepancy being that in the $n=\infty$ case, the map $L_\cX:\textnormal{Fun}(\cX^{\textnormal{opp}},\widehat{\mS})\rightarrow \textnormal{Shv}_{\widehat{\mS}}(\cX)$ left adjoint to the inclusion need not be a topological localization (even in our larger universe). Instead, we are appealing to the following more subtle notion of descent in higher category theory: 

\begin{proposition}{}[HTT 6.3.5.28, HTT 6.5.2.19]\label{Sheaves on Topoi From HTT}

    There exists a unique Grothendieck topology $J$ on $\cX$ such that $L_\cX$ factors as $\textnormal{Fun}(\cX^{\textnormal{opp}},\widehat{\mS})\xrightarrow{T}\widehat{\textnormal{Shv}}_J(\cX)\xrightarrow{C}\textnormal{Shv}_{\widehat{\mS}}(\cX)$, where $T$ is the $J$-sheafification functor, and $C$ is a cotopological localization (in our larger universe).

\end{proposition}  

Our first main result is that comparing $\textnormal{Shv}_{\widehat{\mS}}(\cX)$ to $\widehat{\textnormal{Shv}}_{\C an(\cX)}(\cX)$ is not all that unreasonable, in that the topology $J$ on $\cX$ identified in \ref{Sheaves on Topoi From HTT} may always be identified as $\C an (\cX)$. We prove this as a part of the following theorem, which identifies several of the basic geometric structures of $\textnormal{Shv}_{\widehat{\mS}}(\cX)$:

\begin{theorem}{}\label{Sheaves on Topoi Result Intro}

     (Compare to \ref{Sheaves on Topoi Result}.) Let $\cX$ be an $(\infty,1)$-topos, and write $\C an(\cX)$ for the effective epimorphism topology on $\cX$. Then:

     \begin{itemize}

         \item [$(1)$] The inclusion $\textnormal{Shv}_{\widehat{\mS}}(\cX)\hookrightarrow \textnormal{Fun}(\cX^{\textnormal{opp}},\widehat{\mS})$ factors through $\widehat{\textnormal{Shv}}_{\C an(\cX)}(\cX)\subset \textnormal{Fun}(\cX^{\textnormal{opp}},\widehat{\mS})$. Moreover, the left adjoint $C:\widehat{\textnormal{Shv}}_{\C an(\cX)}(\cX)\rightarrow \textnormal{Shv}_{\widehat{\mS}}(\cX)$ induced by HTT 6.3.5.28 is a cotopological localization (in our larger universe).

         \item [$(2)$] For all $n\in \ZZ_{\geq 0}$, pulling-back along the localization $\cX\rightarrow \tau_{\leq n}\cX$ yields a categorical equivalence $\widehat{\textnormal{Shv}}_{\tau_{\leq n}\widehat{\mS}}(\tau_{\leq n}\cX)\xrightarrow{\sim} \tau_{\leq n-1}\widehat{\textnormal{Shv}}_{\C an(\cX)}(\cX)$, where $\widehat{\textnormal{Shv}}_{\tau_{\leq n}\widehat{\mS}}(\tau_{\leq n}\cX)$ is the full subcategory of $\textnormal{Fun}(\tau_{\leq n}\cX^{\textnormal{opp}},\tau_{\leq n}\widehat{\mS})$ spanned by those maps $F:\tau_{\leq n}\cX^{\textnormal{opp}}\rightarrow \tau_{\leq n}\widehat{\mS}$ which preserve small limits.

         \item [$(3)$] Pulling-back along $\cX\rightarrow \cX^{\textnormal{hyp}}$ yields a categorical equivalence $\textnormal{Shv}_{\widehat{\mS}}(\cX^{\textnormal{hyp}})\xrightarrow{\sim} \widehat{\textnormal{Shv}}_{\C an(\cX)}(\cX)^{\textnormal{hyp}}$.
         
     \end{itemize}
    
\end{theorem}

 \noindent We note that in light of \ref{Sheaves on Topoi Result Intro} $(1)$, the respective classifications of the truncated and hypercomplete objects of $\widehat{\textnormal{Shv}}_{\C an(\cX)}(\cX)$ given by \ref{Sheaves on Topoi Result Intro} $(2)$ and \ref{Sheaves on Topoi Result Intro} $(3)$ also classify the analogous data for $\textnormal{Shv}_{\widehat{\mS}}(\cX)$. Theorem \ref{Sheaves on Topoi Result Intro} has several immediate consequences, including:  

\begin{corollary}{}\label{hypersheaves on hypercomplete topoi}

    Let $\cX$ be a hypercomplete $(\infty,1)$-topos, and let $\cD$ be an $(\infty,1)$-category. Then, a map $F:\cX^{\textnormal{opp}}\rightarrow \cD$ is a hypersheaf for the effective epimorphism topology on $\cD$ precisely when $F$ preserves small limits. 
    
\end{corollary}

\begin{corollary}{}\label{The Canonical Topology on a Topos Intro}

    (Compare to \ref{The Canonical Topology on a Topos}.) Let $n\in \ZZ_{\geq 1}\cup \{\infty\}$, and let $\cX$ be an $(n,1)$-topos. Then, there exists a unique finest Grothendieck topology $J$ on $\cX$ satisfying the property that every representable functor $\cX^{\textnormal{opp}}\rightarrow \mS$ is a $J$-sheaf. Moreover, this is the effective epimorphism topology on $\cX$. 
    
\end{corollary}

\noindent Corollary \ref{hypersheaves on hypercomplete topoi} may be seen as a higher analogue of \ref{Sheaves on Topoi From HTT}, and \ref{The Canonical Topology on a Topos Intro} is a verification of the folkloric claim that the effective epimorphism topology on $\cX$ is in fact the canonical topology on $\cX$. Unlike in the classical case for $(1,1)$-topoi, claim $(1)$ of \ref{Sheaves on Topoi Result Intro} is not inherently obvious. Instead, our proof of \ref{Sheaves on Topoi Result Intro} (1) is carried out by first introducing and developing a more general theory of \emph{geometric $(\infty,1)$-categories}, which will occupy section \ref{Geometric Categories and Geometric Sites}: 

\begin{definition}{}\label{kappa geometric intro definition}

    Fix a regular cardinal $\kappa$, and let $\C$ be an $(\infty,1)$-category with $\kappa$-small coproducts and pullbacks. We call $\C$ a \emph{$\kappa$-geometric} $(\infty,1)$-category if any of the following equivalent conditions hold:

    \begin{itemize}

        \item [$(1)$] Said $\kappa$-small coproducts in $\C$ are disjoint and universal. 

        \item [$(2)$] Changing universe if necessary so that $\C$ is small, the reflective localization functor $L_\kappa:\textnormal{Fun}(\C^{\textnormal{opp}},\mS)\rightarrow \textnormal{Fun}^{\times,\kappa}(\C^{\textnormal{opp}},\mS)$ is left exact, where $\textnormal{Fun}^{\times,\kappa}(\C^{\textnormal{opp}},\mS)$ denotes the full subcategory of $\textnormal{Fun}(\C^{\textnormal{opp}},\mS)$ spanned by maps $F:\C^{\textnormal{opp}}\rightarrow \mS$ which preserve $\kappa$-small products. 

        \item [$(3)$] There exists a subcanonical Grothendieck topology $J$ on $\C$ such that every $J$-sheaf $F:\C^{\textnormal{opp}}\rightarrow \mS$ preserves $\kappa$-small products. 

    \end{itemize}

    \noindent When $\kappa=\Omega$ is the strongly inaccessible cardinal bounding our implicitly chosen universe, we will refer an $\Omega$-geometric $(\infty,1)$-category $\C$ as \emph{geometric}, and when $\kappa=\aleph_0$ is the smallest infinite cardinal, we will refer to an $\aleph_0$-geometric $(\infty,1)$-category $\C$ as \emph{finitary}. 
    
\end{definition}

\noindent We prove the equivalence of conditions $(1)$, $(2)$ and $(3)$ of \ref{kappa geometric intro definition} (see \ref{Converse to Left Exact Localization for kappa Topology} and \ref{kappa cats. through sheaves}). There are many important examples of geometric $(\infty,1)$-categories:

\begin{example}{}\label{Geom Cats. Remark Intro}

   The following are several examples of $\kappa$-geometric $(\infty,1)$-categories: 

    \begin{itemize}
        \item [$(1)$] Any $(n,1)$-topos $\cX$ is geometric, for all $n\in \ZZ_{\geq 1}\cup \{\infty\}$. 
    
        \item [$(2)$] Both $\textnormal{CAlg}^{\textnormal{opp}}$ and $\textnormal{CAlg}^{\textnormal{cn},\textnormal{opp}}$ are finitary $(\infty,1)$-categories.

        \item [$(3)$] The $(\infty,1)$-category of $(\infty,1)$-topoi, $\infty\cT op$, is geometric. 

        \item [$(4)$] The $(\infty,1)$-category of spectral Deligne Mumford stacks $\textnormal{SpDM}$ is geometric. 

        \item [$(5)$] The $(1,1)$-category of schemes and the $(1,1)$-category of topological spaces are both geometric.

    \end{itemize}

\end{example}

\noindent Each of the $(\infty,1)$-categories in \ref{Geom Cats. Remark Intro} behaves \emph{geometrically} in some inexact sense, and \ref{kappa geometric intro definition} is meant to make this sentiment precise. To this point, each satisfies the hypotheses of the following, which may be generalized in the case that $\C$ is an $(n,1)$-topos (see \ref{Generalization For Groups Internal to N Topoi}):

\begin{proposition}{}\label{Intro of Internal Vs External Group Objects Theorem}

    (Compare to \ref{Internal Vs External Group Objects Theorem}.) Fix a (possibly large) regular cardinal $\kappa$, and let $\C$ be a $\kappa$-geometric $(\infty,1)$-category with final object $\bm{1}_\C\in \C$. Let $G$ be a $\kappa$-small discrete group, and write $G_\C$ for the group object in $\C$ given by the canonical group object structure on $\coprod_{g\in G}\bm{1}_\C\in \C$. Then, there exists a categorical equivalence $\Phi:\textnormal{Fun}(\textnormal{B}G,\C)\xrightarrow{\sim} \textnormal{LMon}_{G_\C}(\C)$ satisfying the property that the composition $\textnormal{Fun}(\textnormal{B}G,\C)\xrightarrow{\Phi} \textnormal{LMon}_{G_\C}(\C)\xrightarrow{F}\C$ is naturally isomorphic to $\textnormal{ev}_*:\textnormal{Fun}(\textnormal{B}G,\C)\rightarrow \C$, where $F:\textnormal{LMon}_{G_\C}(\C)\rightarrow \C$ is the canonical forgetful functor and $*\in \textnormal{B}G$ is the lone vertex.

\end{proposition}

\noindent Our main result for geometric $(\infty,1)$-categories is the following generalization of Lurie's (hyper)sheaf classification result for finitary $(\infty,1)$-categories given by \citeq{lurie2018sag} A.3.3.1 and \citeq{lurie2018sag} A.5.7.2:

\begin{theorem}{}\label{Cech Descent Intro}

  (Compare to \ref{Cech Descent Theorem}.) Fix a (possibly large) regular cardinal $\kappa$, and let $\C$ be a $\kappa$-geometric $(\infty,1)$-category. Let $\tau$ be a Grothendieck pretopology on $\C$ satisfying the following conditions: 

    \begin{itemize}
    
        \item [\textnormal{(i)}] For all $\kappa$-small collections of objects $\{U_i\}_{i\in I}$ in $\C$, the family $\{U_i\rightarrow \coprod_{j\in I}U_j\}_{i\in I}$ is a $\tau$-covering. 

        \item [\textnormal{(ii)}] All $\tau$-coverings $\{V_k\rightarrow X\}_{k\in K}$ satisfy the property that $K$ is $\kappa$-small.

    \end{itemize}

    \noindent Then, for all $(\infty,1)$-categories $\cD$, a map $F:\C^{\textnormal{opp}}\rightarrow \cD$ is a $\tau$-sheaf precisely when $F$ satisfies the following conditions: 

    \begin{itemize}

        \item [$(1)$] $F$ preserves $\kappa$-small products.

        \item [$(2)$] For all $\tau$-coverings $\{V_k\rightarrow X\}_{k\in K}$, the composition $\Delta_{\textnormal{inj}}^\triangleleft\xrightarrow{U}\C^{\textnormal{opp}}\xrightarrow{F}\cD$ is a limit diagram, where $U$ is the (opposite of the) \v{C}ech nerve of $\coprod_{k\in K}V_k\rightarrow X$. 

    \end{itemize}

    \noindent Moreover, $F$ is a $\tau$-hypersheaf precisely when $F$ satisfies (1) along with the following condition:

    \begin{itemize}
    
        \item [$(2^\prime)$] For all $\tau$-hypercoverings $X_\bullet:(\Delta^{\textnormal{opp}}_{\textnormal{inj}})^\triangleright\rightarrow \C$ (as defined in \ref{Hypercoverings Definition}), the composition $\Delta_{\textnormal{inj}}^\triangleleft\xrightarrow{X_\bullet^{\textnormal{opp}}}\C^{\textnormal{opp}}\xrightarrow{F}\cD$ is a limit diagram.

    \end{itemize}
    
\end{theorem}

\noindent After invoking \ref{Cech Descent Intro} to prove \ref{Sheaves on Topoi Result Intro}, we dedicate \ref{Global Sheaves Section} to proving the following global analogue of \ref{Sheaves on Topoi Result Intro}:

\begin{proposition}{}\label{Global Sheaves on Topoi Prop Intro}

   (Compare to \ref{Global Sheaves on Topoi Prop}.) Let $F:\infty\cT op^{\textnormal{opp}}\rightarrow \widehat{\mS}$ be a sheaf (see  \ref{Relative Global Sheaves on Topoi Defn}), and write $\widehat{\textnormal{Shv}}(\infty\cT op_{/F})$ for the full subcategory of $\textnormal{Fun}((\infty\cT op_{/F})^{\textnormal{opp}},\widehat{\mS})$ spanned by the sheaves $\infty\cT op^{\textnormal{opp}}\rightarrow \widehat{\mS}$. Writing $\widehat{\textnormal{Shv}}_{\textnormal{\'{e}t}_{/F}}(\infty\cT op_{/F})$ for the full subcategory spanned by the sheaves for the \'{e}tale topology on $\infty\cT op_{/F}$ (in the sense of \ref{Etale Slices of Sheaf Topoi} (1)), we then have the following:

    \begin{itemize}
    
        \item [$(1)$] The inclusion $\widehat{\textnormal{Shv}}(\infty\cT op_{/F})\hookrightarrow \textnormal{Fun}((\infty\cT op_{/F})^{\textnormal{opp}},\widehat{\mS})$ factors as the composition $\widehat{\textnormal{Shv}}(\infty\cT op_{/F})\hookrightarrow \widehat{\textnormal{Shv}}_{\textnormal{\'{e}t}_{/F}}(\infty\cT op_{/F})\hookrightarrow \textnormal{Fun}((\infty\cT op_{/F})^{\textnormal{opp}},\widehat{\mS})$. 

        \item [$(2)$] The inclusion $\widehat{\textnormal{Shv}}(\infty\cT op_{/F})\hookrightarrow \widehat{\textnormal{Shv}}_{\textnormal{\'{e}t}_{/F}}(\infty\cT op_{/F})$ induced by $(1)$ admits a cotopologoical (in our larger universe) left adjoint $C:\widehat{\textnormal{Shv}}_{\textnormal{\'{e}t}_{/F}}(\infty\cT op_{/F})\rightarrow \widehat{\textnormal{Shv}}(\infty\cT op_{/F})$. 
        
    \end{itemize}

    \noindent Moreover, the inclusion $\widehat{\textnormal{Shv}}(\infty\cT op_{/F})\hookrightarrow \textnormal{Fun}((\infty\cT op_{/F})^{\textnormal{opp}},\widehat{\mS})$ identifies $\widehat{\textnormal{Shv}}(\infty\cT op_{/F})^{\textnormal{hyp}}$ with the full subcategory of $\textnormal{Fun}((\infty\cT op_{/F})^{\textnormal{opp}},\widehat{\mS})$ spanned by those maps $G:(\infty\cT op_{/F})^{\textnormal{opp}}\rightarrow \widehat{\mS}$ which satisfy the following two properties:

    \begin{itemize}
        \item [$(\textnormal{i})$] $G$ preserves small products. 

        \item [$(\textnormal{ii})$] For all augmented semisimplicial objects $X_\bullet:(\Delta_{\textnormal{inj}}^{\textnormal{opp}})^\triangleright\rightarrow \infty\cT op_{/F}$ satisfying the property that for each $n\in \ZZ_{\geq 0}$ the $n$th covering map $X_n\rightarrow M_n(X)$ is both \'{e}tale and an effective epimorphism of $(\infty,1)$-topoi, the composition $\Delta_{\textnormal{inj}}^\triangleleft\xrightarrow{X^{\textnormal{opp}}_\bullet} (\infty\cT op_{/F})^{\textnormal{opp}}\xrightarrow{G} \widehat{\mS}$ is a limit diagram.

    \end{itemize}
    
\end{proposition}

As a concluding remark of this introduction, a significant portion of the theory of this paper relies on the following folkloric result, which we prove in \ref{Slices of Sheaf Topoi}:

\begin{proposition}{}\label{Etale Slices of Sheaf Topoi Intro}

    (Compare to \ref{Etale Slices of Sheaf Topoi}.) Let $\C$ be a small $(\infty,1)$-category equipped with a topology $J$. Let $F:\C^{\textnormal{opp}}\rightarrow \mS$ be a map, and write $p:\C_{/F}\rightarrow \C$ for the right fibration $\C_{/F}:=\C\times_{\textnormal{Fun}(\C^{\textnormal{opp}},\mS)}\textnormal{Fun}(\C^{\textnormal{opp}},\mS)_{/F}\rightarrow \C$ classifying $F$. Write $p^*:\textnormal{Fun}(\C^{\textnormal{opp}},\mS)\rightarrow \textnormal{Fun}((\C_{/F})^{\textnormal{opp}},\mS)$ for the map given by pulling-back along $p^{\textnormal{opp}}:(\C_{/F})^{\textnormal{opp}}\rightarrow \C^{\textnormal{opp}}$, and write $p_!:\textnormal{Fun}((\C_{/F})^{\textnormal{opp}},\mS)\rightarrow \textnormal{Fun}(\C^{\textnormal{opp}},\mS)$ and $p_*:\textnormal{Fun}((\C_{/F})^{\textnormal{opp}},\mS)\rightarrow \textnormal{Fun}(\C^{\textnormal{opp}},\mS)$ for the left and right adjoints of $p^*$, respectively. Then:

    \begin{itemize}

        \item [$(1)$] Let $J_{/F}$ be the collection of sieves on objects $U\in \C_{/F}$ of the form $\C^{0}_{/p(U)}\times_{\C_{/p(U)}}(\C_{/F})_{/U}$, where $\C^{0}_{/p(U)}\subseteq \C_{/p(U)}$ is a $J$-covering sieve on $p(U)$. Then, $J_{/F}$ is a Grothendieck topology on $\C_{/F}$.

        \item [$(2)$] Suppose that $F$ is a $J$-sheaf. Then, the categorical equivalence $\textnormal{Fun}((\C_{/F})^{\textnormal{opp}},\mS)\xrightarrow{\sim}\textnormal{Fun}(\C^{\textnormal{opp}},\mS)_{/F}$ over $\textnormal{Fun}(\C^{\textnormal{opp}},\mS)$ given by the left Kan extension of the inclusion $\C_{/F}\hookrightarrow \textnormal{Fun}(\C^{\textnormal{opp}},\mS)_{/F}$ along $\C_{/F}\rightarrow \textnormal{Fun}((\C_{/F})^{\textnormal{opp}},\mS)$ restricts to a categorical equivalence $\textnormal{Shv}_{J_{/F}}(\C_{/F})\xrightarrow{\sim}\textnormal{Shv}_{J}(\C)_{/F}$.

        \item [$(3)$] The adjunction $p^*:\textnormal{Fun}(\C^{\textnormal{opp}},\mS)\rightleftarrows \textnormal{Fun}((\C_{/F})^{\textnormal{opp}},\mS):p_*$ restricts to an adjunction $p^*:\textnormal{Shv}_{J}(\C)\rightleftarrows \textnormal{Shv}_{J_{/F}}(\C_{/F}):p_*$. Moreover, when $F$ is $J$-sheaf, the adjunction $p_!:\textnormal{Fun}((\C_{/F})^{\textnormal{opp}},\mS)\rightleftarrows\textnormal{Fun}(\C^{\textnormal{opp}},\mS):p^*$ restricts to an adjunction $p_!:\textnormal{Shv}_{J_{/F}}(\C_{/F})\rightleftarrows \textnormal{Shv}_{J}(\C):p^*$. 

        \item [$(4)$] Whenever $F$ is a $J$-sheaf, the map $p_*:\textnormal{Shv}_{J_{/F}}(\C_{/F})\rightarrow \textnormal{Shv}_{J}(\C)$ is an \'{e}tale morphism of $(\infty,1)$-topoi.

    \end{itemize}

\end{proposition}

\noindent In the appendix (\ref{Appendix}), we further prove the following pair of results, which may be of independent interest:

\begin{proposition}{}\label{Localization of Modules Intro}

    (Compare to \ref{Localization of Modules}.)  Let $\C^\otimes$ be a monoidal $(\infty,1)$-category satisfying the property that $\C$ admits geometric realizations and the tensor product $\otimes:\C\times \C\rightarrow \C$ preserves geometric realizations separately in each variable. Let $\cD\subseteq \C$ be a reflective subcategory such that the localization functor $L:\C\rightarrow \cD$ is compatible with the monoidal structure on $\C$, and let $A\in \textnormal{Alg}(\C)$. Then:

    \begin{itemize}
    
        \item [$(1)$] The map $L^\prime:\textnormal{LMod}_{A}(\C)\rightarrow \textnormal{LMod}_{L(A)}(\cD)$ induced by post-composition with $L^\otimes$ is a reflective localization.

        \item [$(2)$] An object $M\in \textnormal{LMod}_{A}(\C)$ is $L^\prime$-local if and only if the underlying $\C$-object of $M$ is contained in the full subcategory $\cD\subseteq \C$.
        
    \end{itemize}
    
\end{proposition}

\begin{proposition}{}\label{HyperSheafification Ind of Choice of Universe Intro}

    (Compare to \ref{HyperSheafification Ind of Choice of Universe}.) Let $\C$ be a small $(\infty,1)$-category equipped with a Grothendieck topology $J$. Write $i:\textnormal{Fun}(\C^{\textnormal{opp}},\mS)\hookrightarrow \textnormal{Fun}(\C^{\textnormal{opp}},\widehat{\mS})$ and  $i^\prime:\textnormal{Shv}_J(\C)\hookrightarrow \widehat{\textnormal{Shv}}_J(\C)$ for the corresponding inclusions, let $L:\textnormal{Fun}(\C^{\textnormal{opp}},\mS)\rightarrow \textnormal{Shv}_J(\C)$ denote the sheafification functor, which we identify as an arrow $L:\textnormal{Fun}(\C^{\textnormal{opp}},\mS)\rightarrow \textnormal{Fun}(\C^{\textnormal{opp}},\mS)$, and write $L^{\textnormal{hyp}}:\textnormal{Shv}_J(\C)\rightarrow \textnormal{Shv}_J(\C)^{\textnormal{hyp}}$ for the map left adjoint to the inclusion, which we identify as an arrow $L^{\textnormal{hyp}}:\textnormal{Shv}_J(\C)\rightarrow \textnormal{Shv}_J(\C)$. For each $n\in \ZZ$ let $\tau_{\leq n}:\textnormal{Shv}_J(\C)\rightarrow\tau_{\leq n}\textnormal{Shv}_J(\C)$ be the corresponding truncation functor, which we identify as an arrow $\tau_{\leq n}:\textnormal{Shv}_J(\C)\rightarrow \textnormal{Shv}_J(\C)$. Similarly define the maps $L^\prime:\textnormal{Fun}(\C^{\textnormal{opp}},\widehat{\mS})\rightarrow \widehat{\textnormal{Shv}}_J(\C)\subseteq \textnormal{Fun}(\C^{\textnormal{opp}},\widehat{\mS})$, $L^{\prime,\textnormal{hyp}}:\widehat{\textnormal{Shv}}_J(\C)\rightarrow\widehat{\textnormal{Shv}}_J(\C)^\textnormal{hyp}\subseteq \widehat{\textnormal{Shv}}_J(\C)$ and $\tau^\prime_{\leq n}:\widehat{\textnormal{Shv}}_J(\C)\rightarrow \tau_{\leq n}\widehat{\textnormal{Shv}}_J(\C)\subseteq \widehat{\textnormal{Shv}}_J(\C)$. Then: 

    \begin{itemize}
    
        \item [$(1)$] $i^\prime \circ L\cong L^\prime\circ i$ as maps $\textnormal{Fun}(\C^{\textnormal{opp}},\mS)\rightarrow \textnormal{Fun}(\C^{\textnormal{opp}},\widehat{\mS})$.

        \item [$(2)$] $i^\prime\circ \tau_{\leq n} \cong \tau^\prime_{\leq n}\circ i^\prime$ as maps $\textnormal{Shv}_J(\C)\rightarrow \widehat{\textnormal{Shv}}_J(\C)$. 

        \item [$(3)$] $i^\prime\circ L^{\textnormal{hyp}}\cong L^{\prime,\textnormal{hyp}}\circ i^\prime$ as maps $\textnormal{Shv}_J(\C)\rightarrow\widehat{\textnormal{Shv}}_J(\C)$.
        
    \end{itemize}
    
\end{proposition}

\begin{remark}{}

    We note that \ref{HyperSheafification Ind of Choice of Universe Intro} $(1)$ already appears in the literature as \citeq{lurie2009derivedalgebraicgeometryv} 2.4.10. 
    
\end{remark}

\begin{remark}{}

We would also like to acknowledge that Thorger Gei{\sss} has independently obtained results similar to \ref{Cech Descent Intro} and \ref{Etale Slices of Sheaf Topoi Intro}. We were made aware of this in a private communication after the completion of these portions of the theoretical component of this work.  

\end{remark}

\begin{remark}{}

    As a final note in this introduction, throughout this paper we often assume the existence of a largely unnecessary amount of pullbacks. This is merely for convenience, and many of our results hold in greater generality.
    
\end{remark}

\subsection*{Acknowledgments}

I would like to thank my advisor Mike Hill for his wonderful support and guidance. In addition, I would like to thank Thorger Gei{\sss} for suggesting an approach to proving part of \ref{Chech Hyperdescent}, and I would like to thank Jiacheng Liang for suggesting an approach to proving \ref{small hypersheaves are large hypersheaves}. I would like to further thank Thorger Gei{\sss} for helpful comments on a draft of this article. This work was partially supported by the Department of Education under Award P200A240046.

\subsection{Notation}

\noindent We will use the Joyal-Lurie model structure on simplicial sets, and we assume a sufficient supply of Grothendieck universes. In the majority of cases we adopt the notation of \citeq{lurie2024kerodon} or \citeq{lurie2018sag}, defaulting to the notation of \citeq{lurie2024kerodon} in the case of conflict. However, not all notation used is conventional. A (very incomplete) list of commonly used notation is as follows: 

\begin{itemize}

    \item [$\bullet$] $\mS$ will denote the $(\infty,1)$-category of (small) spaces, and $\widehat{\mS}$ will denote the $(\infty,1)$-category of spaces which are possible large.

    \item [$\bullet$] For all pairs of $(\infty,1)$-categories $\C$ and $\cD$, $\textnormal{Fun}(\C,\cD)$ will denote the $(\infty,1)$-category of $(\infty,1)$-functors $F:\C\rightarrow \cD$.

    \item [$\bullet$] We will stray from the notation of Kerodon, in that we will not distinguish between a category $\C$ and its nerve $\textnormal{N}_\bullet (\C)$.

    \item [$\bullet$] Let $\cX$ be an $(\infty,1)$-topos. Then, $\cX^{\textnormal{hyp}}$ will denote the full subcategory of $\cX$ spanned by the hypercomplete objects. In addition $\cX^\heartsuit$ (as well as $\tau_{\leq 0}\cX$) will be used to denote the full subcategory of $\cX$ spanned by the discrete objects. 

    \item [$\bullet$] $\Delta$ will denote the (skeletal) simplex category with objects $[n]$ for all $n\in \ZZ_{\geq 0}$. The initial object of $\Delta^\triangleleft$ will be denoted by $[-1]$. $\Delta_{\textnormal{inj}}$ will denote subcategory of $\Delta$ spanned by those maps $[m]\rightarrow [n]$ which are injective maps of linearly ordered sets, and $\Delta^\triangleleft_{\textnormal{inj}}\subset \Delta^\triangleleft$ is defined similarly.

    \item [$\bullet$] Let $\C$ be a locally small $(\infty,1)$-category. Then, $h^\C_\bullet:\C\rightarrow \textnormal{Fun}(\C^{\textnormal{opp}},\mS)$ will be used to denote the Yoneda embedding. For each $X\in \C$, the representable functor $\C^{\textnormal{opp}}\rightarrow \mS$ will be denoted by either $\textnormal{Hom}_\C(-,X)$, $h^\C_X$, or $h_X$, and the map $\C\rightarrow \mS$ corepresented by $X$ will be denoted by either $\textnormal{Hom}_\C(X,-)$, $h^X_\C$, or $h^X$. 

    \item [$\bullet$] Let $n\in \ZZ_{\geq 1}\cup \{\infty\}$, and let $\cX$ be an $(n,1)$-topos. We will write $\C an (\cX)$ for the effective epimorphism topology on $\cX$, which we will also refer to as the \emph{canonical topology} on $\cX$ (see \ref{The Canonical Topology on a Topos} for a justification of this terminology).

\end{itemize}

\section{Preliminaries on Sheaves}

\subsection{Topoi, Sheaves, and Sites}

\noindent We begin with our most essential definitions.

\begin{definition}{}\label{presentable localizations}

    Let $\C$ be a presentable $(\infty,1)$-category, and let $S$ be a (possibly large) collection of arrows in $\C$. Write $\overline{S}$ for the strongly saturated class of morphisms in $\C$ generated by $S$ (\citeq{highertopostheory} 5.5.4.7), and suppose that $\overline{S}$ is of small generation (\citeq{highertopostheory} 5.5.4.7). Writing $\C^\prime\subseteq \C$ for the full subcategory of $\C$ spanned by the $S$-local objects, one calls a map $L:\C\rightarrow \C^\prime$ left adjoint to the inclusion $\C^\prime\hookrightarrow \C$ (\citeq{lurie2024kerodon} 06VH) a \emph{Bousfield localization of $\C$ at $S$}.
    
\end{definition}

\begin{remark}{}

    Under the assumptions of \ref{presentable localizations}, the map $L:\C\rightarrow \C^\prime$ always exists, $\C^\prime$ is presentable, and $\C^\prime\hookrightarrow \C$ is accessible (\citeq{highertopostheory} 5.5.4.15).
    
\end{remark}

\noindent Definition \ref{presentable localizations} is not entirely standard: in \citeq{highertopostheory}, such $L:\C\rightarrow \C^\prime$ are simply referred to as \emph{localizations}, which creates a notational conflict with the following:

\begin{definition}{}\label{Regular localizations definition}

    Let $\C$ be an $(\infty,1)$-category, and let $W$ be any collection of morphisms in $\C$. A \emph{localization of $\C$ at $W$} is an $(\infty,1)$-category $W^{-1}\C$ and an $(\infty,1)$-functor $F:\C\rightarrow W^{-1}\C$ satisfying the property that for all $(\infty,1)$-categories $\cE$, the map $F^*:\textnormal{Fun}(W^{-1}\C,\cE)\rightarrow \textnormal{Fun}(\C,\cE)$ is fully faithful, and the essential image of $F^*$ is comprised of those maps $G:\C\rightarrow \cE$ which carry each $w\in W$ to an isomorphism in $\cE$. 
    
\end{definition}

\noindent Under the notation of \ref{presentable localizations}, it is not true that $L:\C\rightarrow \C^\prime$ is a localization at $S$. However, it \emph{is} true that such $L$ are localizations at $\overline{S}$, and we will (somewhat abusively) use the naked term \emph{localization} to refer to such a Bousfield localization. Moreover, the map $L:\C\rightarrow \C^\prime$ does still satisfy the following universal property: 

\begin{lemma}{}\label{Univ property of Presentable Localizations}

    Under the notation of \ref{presentable localizations}, let $\cE$ be a cocomplete $(\infty,1)$-category. Write $\textnormal{Fun}^{co\ell im}(\C^\prime,\cE)$ for the full subcategory of $\textnormal{Fun}(\C^\prime,\cE)$ spanned by those maps $\C^\prime\rightarrow \cE$ which preserve small colimits, and define $\textnormal{Fun}^{co\ell im}(\C,\cE)\subseteq \textnormal{Fun}(\C,\cE)$ similarly. Then, pulling-back along $L$ yields a fully faithful map $\textnormal{Fun}^{co\ell im}(\C^\prime,\cE)\rightarrow \textnormal{Fun}^{co\ell im}(\C,\cE)$, whose essential image is those $(\infty,1)$-functors carrying every edge of $S$ to an isomorphism in $\cE$.
    
\end{lemma}

\begin{proof}

    Combining \citeq{highertopostheory} 5.5.4.15 with \citeq{lurie2024kerodon} 02G3, we deduce that $L$ is a localization at $\overline{S}$, in the sense of \ref{Regular localizations definition}. Thus, $L^*:\textnormal{Fun}(\C^\prime,\cE)\rightarrow \textnormal{Fun}(\C,\cE)$ is fully faithful, and the essential image of $L^*$ is the full subcategory of $\textnormal{Fun}(\C,\cE)$ spanned by those arrows $G:\C\rightarrow \cE$ satisfying the property that, for all $s\in \overline{S}$, $G(s)$ is an isomorphism in $\cE$. Since $L$ preserves small colimits, it is immediately apparent that $L^*$ restricts to an arrow $\textnormal{Fun}^{co\ell im}(\C^\prime,\cE)\rightarrow \textnormal{Fun}^{co\ell im}(\C,\cE)$. It remains to show that a cococontinuous map $G:\C\rightarrow \cE$ carries each $s\in \overline{S}$ to an isomorphism precisely when $G$ carries each $s^\prime\in S$ to an isomorphism, which follows from \citeq{highertopostheory} 5.5.4.10.

\end{proof}

\begin{remark}{}\label{Small error in 1.3.1.7}

     In the case that $\cE$ is also locally small, \ref{Univ property of Presentable Localizations} coincides with \citeq{highertopostheory} 5.5.4.20. Moreover, under the notation of \citeq{lurie2018sag} 1.3.1.7, the proof of \citeq{lurie2018sag} 1.3.1.7 only holds in the case that $\C$ is locally small (since this relies on \citeq{highertopostheory} 5.5.4.20). The more general case may be deduced from \ref{Univ property of Presentable Localizations}.
    
\end{remark}

\begin{definition}{}[HTT 6.1.0.6, HTT 6.1.5.3, HTT 6.5.2.19]\label{Defn of an Inf Topos}

    Let $\cX$ be an $(\infty,1)$-category. We call $\cX$ an \emph{$(\infty,1)$-topos} if $\cX$ satisfies any of the following equivalent definitions: 

    \begin{itemize}
    
        \item [$(1)$] $\cX$ satisfies the following:
        
        \begin{itemize}
             \item [$(a)$] $\cX$ is presentable. 

            \item [$(b)$] (Small) colimits in $\cX$ are universal. 

            \item [$(c)$] Coproducts in $\cX$ are disjoint. 

            \item [$(d)$] For all groupoid objects $U:\Delta^{\textnormal{opp}}\rightarrow \cX$, the colimit diagram $\overline{U}:(\Delta^{\textnormal{opp}})^\triangleright\rightarrow \cX$ is a \v{C}ech nerve.
        \end{itemize}

        \item [$(2)$] $\cX$ is presentable, and for some regular cardinal $\kappa$ such that $\cX$ is $\kappa$-accessible and $\cX_{\leq \kappa}\subseteq \cX$ is closed under finite limits (which always exists by \citeq{highertopostheory} 5.4.7.4), the map left adjoint to $\cX\rightarrow \textnormal{Fun}((\cX_{\leq \kappa})^\textnormal{opp},\mS)$ is left exact. 

        \item [$(3)$] There exists a small $(\infty,1)$-category $\C$, a Grothendieck topology $J$ on $\C$, and a cotopological localization $\textnormal{Shv}_J(\C)\rightarrow \cX$.

    \end{itemize}
    
\end{definition}

\noindent There are many further equivalent variations of \ref{Defn of an Inf Topos}, many of which are covered in detail in \citeq{highertopostheory} Chapter 6. For the case that $n\in \ZZ_{\geq 1}$, we will recall the definition of an $(n,1)$-topos for $n\in \ZZ_{\geq 1}$ in \ref{Truncated Objects} (see \ref{n Topoi definition}).

\begin{definition}{}\label{Sheaf Definition}

    Let $\C$ be an $(\infty,1)$-category (which is not necessarily small), and let $J$ be a Grothendieck topology on $\C$. Let $\cD$ be an arbitrary $(\infty,1)$-category. We will say that a map $F:\C^{\textnormal{opp}}\rightarrow \cD$ is a \emph{$J$-sheaf} if for all $X\in \C$ and $J$-covering sieves $\C^0_{/X}\subseteq \C_{/X}$, the composition $(\C^0_{/X})^\triangleright\hookrightarrow (\C_{/X})^\triangleright\rightarrow \C\xrightarrow{F^{\textnormal{opp}}}\cD^{\textnormal{opp}}$ is a colimit diagram. 
    
\end{definition}

\noindent In practice, many Grothendieck topologies are generated by \emph{sites}:

\begin{definition}{}\label{higher sites definition}

    Let $\C$ be an $(\infty,1)$-category with pullbacks. Let $\tau$ be a (possibly large) collection of (possibly large) families of morphisms with fixed target $\{f_i:U_i\rightarrow X\}_{i\in I}$ in $\C$, satisfying the following:

    \begin{itemize}

        \item [(1)] For all isomorphisms $f:U\rightarrow V$ in $\C$, $\{f:U\rightarrow V\}\in \tau$. 

        \item [(2)] For all $\{f_i:U_i\rightarrow X\}_{i\in I}\in \tau$ and maps $g:U\rightarrow X$, we have that $\{U\times_XU_i\rightarrow U\}_{i\in I}$ is in $\tau$, where the fibre product $U\times_XU_i$ is formed in $\C$.

        \item [(3)] If $\{f_i:U_i\rightarrow X\}_{i\in I}\in \tau$ and for each $i\in I$ there exists some $\{g_{j_i}:V_{j_i}\rightarrow U_i\}_{j_i\in J_i}\in \tau$, then $\coprod_{i\in I}\{f_i\circ g_{j_i}:V_{j_i}\rightarrow U_i\}_{j_i\in J_i}\in \tau$ for all $\coprod_{i\in I}J_i$-tuples of choices of compositions $f_i\circ g_{j_i}$.

    \end{itemize}

   \noindent  We further make the technical assumption that for all strongly inaccessible cardinals $\Omega$ for which $\C$ is $\Omega$-essentially small, each $\{f_i:U_i\rightarrow X\}_{i\in I}$ is an $\Omega$-small set. Then, we call the pair $(\C,\tau)$ a \emph{site}, and refer to $\tau$ as a (Grothendieck) \emph{pretopology} on $\C$. We will refer to each $\{f_i:U_i\rightarrow X\}_{i\in I}\in \tau$ as a \emph{$\tau$-covering}.
    
\end{definition}

\begin{definition}{}\label{Grothendieck topology generated by a site}

    Let $(\C,\tau)$ be an site. Let $J_\tau$ be the collection of sieves on objects $X\in \C$, where $\C^0_{/X}\subseteq \C_{/X}$ is a $J_\tau$-sieve on $X$ precisely when there exists some $\tau$-covering $\{f_i:U_i\rightarrow X\}_{i\in I}$ such that each $f_i\in \C^0_{/X}$. Then, we call $J_\tau$ the \emph{Grothendieck topology on $\C$ generated by $\tau$}. Noting that $J_\tau$ is a Grothendieck topology on $\C$, we will refer to a $J_\tau$-sheaf $F:\C^{\textnormal{opp}}\rightarrow \cD$ as a ($\cD$-valued) \emph{$\tau$-sheaf}. 
    
\end{definition}

\begin{notation}{}

    In the situation of \ref{Grothendieck topology generated by a site}, we define $\textnormal{Shv}_\tau(\C):=\textnormal{Shv}_{J_\tau}(\C)$. 

\end{notation}

\noindent When working with a small site $(\C,\tau)$, we may use the following to reformulate \ref{Sheaf Definition}:

\begin{lemma}{}\label{Prelim. Cech descent}

    Let $\C$ be a small $(\infty,1)$-category with pullbacks, and let $\tau$ be a Grothendieck pretopology on $\C$. Let $\cD$ be a locally small $(\infty,1)$-category which is also cocomplete, and let $F:\C^{\textnormal{opp}}\rightarrow \cD$ be a map. Then, the following are equivalent: 
    
    \begin{itemize}
    
        \item [$(1)$] $F$ is a $\tau$-sheaf.

        \item [$(2)$] For all $X\in \C$ and $\tau$-coverings $\{f_i:U_i\rightarrow X\}_{i\in I}$ of $X$, the composition $\C^0_{/X}\subseteq \C_{/X}$, the composition $(\C^0_{/X})^\triangleright\hookrightarrow (\C_{/X})^\triangleright\rightarrow \C\xrightarrow{F^{\textnormal{opp}}}\cD^{\textnormal{opp}}$ is a colimit diagram, where $\C^0_{/X}\subseteq \C_{/X}$ is the sieve on $X$ generated by the $f_i$.

        \item [$(3)$] For all $\tau$-coverings $\{f_i:U_i\rightarrow X\}_{i\in I}$ the composition $\Delta^\triangleleft\xrightarrow{U}\textnormal{Fun}(\C^{\textnormal{opp}},\mS)^{\textnormal{opp}}\xrightarrow{F_*}\cD$ is a limit diagram, where $U$ is the (opposite of the) \v{C}ech nerve of $\coprod_{i\in I}h^\C_{U_i}\rightarrow h^\C_X$, and $F_*:\textnormal{Fun}(\C^{\textnormal{opp}},\mS)^{\textnormal{opp}}\rightarrow \cD$ is the right Kan extension of $F:\C^{\textnormal{opp}}\rightarrow \cD$ along $\C^{\textnormal{opp}}\rightarrow \textnormal{Fun}(\C^{\textnormal{opp}},\mS)^{\textnormal{opp}}$. 
        
    \end{itemize}

\end{lemma}

\begin{proof}

It suffices to prove our claim in the case that $\cD=\mS$. Let $\{f_i:U_i\rightarrow X\}_{i\in I}$ be a $\tau$-covering, and let $U:(\Delta^{\textnormal{opp}})^\triangleright\rightarrow \textnormal{Fun}(\C^{\textnormal{opp}},\mS)$ be the \v{C}ech nerve of $\coprod_{i\in I}h^\C_{U_i}\rightarrow h^\C_X$, which we identify as an arrow $U_0:\Delta^{\textnormal{opp}}\rightarrow \textnormal{Fun}(\C^{\textnormal{opp}},\mS)_{/h^\C_X}$. Extend $U_0$ to a colimit diagram $\overline{U_0}:(\Delta^{\textnormal{opp}})^\triangleright\rightarrow \textnormal{Fun}(\C^{\textnormal{opp}},\mS)_{/h^\C_X}$, noting that $\overline{U_0}$ carries the $[0]\rightarrow [-1]$ edge of $(\Delta^{\textnormal{opp}})^\triangleright$ to the map $\Delta^2\rightarrow \textnormal{Fun}(\C^{\textnormal{opp}},\mS)$ giving the epi-mono factorization of $\coprod_{i\in I}h^\C_{U_i}\rightarrow h^\C_X$ (see \citeq{lurie2024kerodon} 04XJ). Write said epi-mono factorization as $\coprod_{i\in I}h^\C_{U_i}\xrightarrow{e}S\xrightarrow{m}h^\C_X$. Let  $F\in \textnormal{Fun}(\C^{\textnormal{opp}},\mS)$,  and identify $\overline{U_0}:(\Delta^{\textnormal{opp}})^\triangleright\rightarrow \textnormal{Fun}(\C^{\textnormal{opp}},\mS)_{/h^\C_X}$ as a map $\overline{U_0}^\prime:\Delta^{\textnormal{opp}}\star \Delta^1\rightarrow \textnormal{Fun}(\C^{\textnormal{opp}},\mS)$. Since that the composition 

$$(\Delta^{\textnormal{opp}})^\triangleright\xrightarrow{\overline{U_0}} \textnormal{Fun}(\C^{\textnormal{opp}},\mS)_{/h^\C_X}\rightarrow \textnormal{Fun}(\C^{\textnormal{opp}},\mS)\xrightarrow{(h_F)^{\textnormal{opp}}}\mS^{\textnormal{opp}}$$

\noindent is a colimit diagram, it follows that $F$ is local with respect to $m:S\rightarrow h^\C_X$ precisely when the composition 

$$(\Delta^{\textnormal{opp}})^\triangleright\xrightarrow{U_0} \textnormal{Fun}(\C^{\textnormal{opp}},\mS)\xrightarrow{(h_F)^{\textnormal{opp}}}\mS^{\textnormal{opp}}$$

\noindent is a colimit diagram. Thus, (2) and (3) are equivalent.

We now prove that (1) and (2) are equivalent. Clearly (1) implies (2), so it remains to show the converse. Suppose that (2) holds. Arguing as in the proof of \citeq{lurie2018sag} A.3.3.1, let $X\in \C$, and let $\C^0_{/X}\subseteq \C_{/X}$ be $\tau$-sieve on $X$. Choose a $\tau$-covering $\{f_i:U_i\rightarrow X\}_{i\in I}$ such that each $f_i$ is an element of $\C^0_{/X}\subseteq \C_{/X}$, and write $\C^1_{/X}$ for the $\tau$-sieve on $X$ generated by the $f_i$. Noting that the composition $(\C^0_{/X})^\triangleright\hookrightarrow (\C^1_{/X})^\triangleright\hookrightarrow (\C_{/X})^\triangleright\rightarrow \C\xrightarrow{F^{\textnormal{opp}}}\mS^{\textnormal{opp}}$ is a colimit diagram, in order to prove that the composition $(\C^1_{/X})^\triangleright\hookrightarrow (\C_{/X})^\triangleright\rightarrow \C\xrightarrow{F^{\textnormal{opp}}}\mS^{\textnormal{opp}}$ is also a colimit diagram, it suffices to show that the composition $\C^1_{/X}\hookrightarrow \C_{/X}\rightarrow\C\xrightarrow{F^{\textnormal{opp}}}\mS^{\textnormal{opp}}$ is left Kan extended from $\C^0_{/X}$. Moreover, suffices to prove  this latter condition in the case that $\C^1_{/X}=\C_{/X}$. Let $g:U\rightarrow X$ be a map in $\C$, and consider the map $G:(\C^0_{/X}\times_{\C_{/X}}(\C_{/X})_{/g})^\triangleright\rightarrow \mS^{\textnormal{opp}}$ given by the composition

$$(\C^0_{/X}\times_{\C_{/X}}(\C_{/X})_{/g})^\triangleright\hookrightarrow ((\C_{/X})_{/g})^\triangleright\rightarrow \C_{/X}\rightarrow \C\xrightarrow{F^{\textnormal{opp}}}\mS^{\textnormal{opp}}.$$

\noindent Let $s:\C_{/U}\rightarrow (\C_{/X})_{/g}$ be a section of the trivial Kan fibration $(\C_{/X})_{/g}\rightarrow \C_{/U}$, and pull-back $\C^0_{/X}\times_{\C_{/X}}(\C_{/X})_{/g}\hookrightarrow (\C_{/X})_{/g}$ to the sieve 

$$\C^0_{/X}\times_{\C_{/X}}\C_{/U}=\C_{/U}\times_{(\C_{/X})_{/g}}(\C^0_{/X}\times_{\C_{/X}}(\C_{/X})_{/g})\hookrightarrow \C_{/U}$$

\noindent on $U$ generated by the $\tau$-covering $\{U\times_XU_i\rightarrow U\}_{i\in I}$. Since $F$ satisfies condition (2), it follows from \ref{Kan extensions of Sheaves} that the composition

$$(\C^0_{/X}\times_{\C_{/X}}\C_{/U})^\triangleright\rightarrow (\C_{/U})^\triangleright\rightarrow \C\xrightarrow{F^{\textnormal{opp}}}\mS^{\textnormal{opp}}$$

\noindent is a colimit diagram. Thus, since the composition $(\C^0_{/X}\times_{\C_{/X}}(\C_{/X})_{/g})^\triangleright\rightarrow (\C^0_{/X}\times_{\C_{/X}}\C_{/U})^\triangleright\rightarrow (\C_{/U})^\triangleright\rightarrow \C\xrightarrow{F^{\textnormal{opp}}}\mS^{\textnormal{opp}}$ is equal as a map of simplicial sets to $G$ and $\C^0_{/X}\times_{\C_{/X}}(\C_{/X})_{/g}\rightarrow \C^0_{/X}\times_{\C_{/X}}\C_{/U}$ is right cofinal (as it is an equivalence), it follows that $G$ is also a colimit diagram. This proves our claim. 

\end{proof}

\begin{remark}{}\label{Diagram Sheaf condition for sites}

    Let $(\C,\tau)$ be an site. Then, \ref{Prelim. Cech descent} allows for our sheaf condition on $\C$ to be reformulated in a familiar fashion. Let $F:\C^{\textnormal{opp}}\rightarrow \cD$ be a map. Using the left cofinality of $\Delta_{\textnormal{inj}}\hookrightarrow \Delta$, \ref{Prelim. Cech descent} shows that $F$ is a $\tau$-sheaf precisely when, for all $\tau$-coverings $\{f_i:U_i\rightarrow X\}_{i\in I}$, the diagram $\Delta^\triangleleft_{\textnormal{inj}}\rightarrow \cD$ given informally by

    $$F(X)\rightarrow \prod_{i\in I}F(U_i)\rightrightarrows \prod_{i_0,i_1\in I}F(U_{i_0}\times_XU_{i_1})\ \substack{\rightarrow\\[-0.92em] \rightarrow \\[-0.92em] \rightarrow}\ \cdots$$

    \noindent is a limit diagram. Using the fact that the fully faithful embedding $h^\cD_\bullet:\cD\rightarrow \textnormal{Fun}(\cD^{\textnormal{opp}},\mS)$ preserves small limits in $\cD$, a similar statement holds in the case that $\cD$ is only assumed to admit $I$-indexed products for all $\tau$-coverings $\{f_i:U_i\rightarrow X\}_{i\in I}$ (instead of assuming that $\cD$ is cocomplete). Moreover, using the fact that we have assumed the existence of a sufficient supply of Grothendieck universes, this latter statement may be further generalized to the case that there are no size restrictions on $\C$ and $\cD$.

\end{remark}

\subsection{Slices of Topoi}\label{Slices of Sheaf Topoi}

Throughout this paper, we will have an interest in working with sheaves in the relative setting. Our main tool when working in this context will be the following proposition:

\begin{proposition}{}\label{Etale Slices of Sheaf Topoi}

    Let $\C$ be a small $(\infty,1)$-category equipped with a topology $J$. Let $F:\C^{\textnormal{opp}}\rightarrow \mS$ be a map, and write $p:\C_{/F}\rightarrow \C$ for the right fibration $\C_{/F}:=\C\times_{\textnormal{Fun}(\C^{\textnormal{opp}},\mS)}\textnormal{Fun}(\C^{\textnormal{opp}},\mS)_{/F}\rightarrow \C$ classifying $F$. Write $p^*:\textnormal{Fun}(\C^{\textnormal{opp}},\mS)\rightarrow \textnormal{Fun}((\C_{/F})^{\textnormal{opp}},\mS)$ for the map given by pulling-back along $p^{\textnormal{opp}}:(\C_{/F})^{\textnormal{opp}}\rightarrow \C^{\textnormal{opp}}$, and write $p_!:\textnormal{Fun}((\C_{/F})^{\textnormal{opp}},\mS)\rightarrow \textnormal{Fun}(\C^{\textnormal{opp}},\mS)$ and $p_*:\textnormal{Fun}((\C_{/F})^{\textnormal{opp}},\mS)\rightarrow \textnormal{Fun}(\C^{\textnormal{opp}},\mS)$ for the left and right adjoints of $p^*$, respectively. Then:

    \begin{itemize}

        \item [$(1)$] Let $J_{/F}$ be the collection of sieves on objects $U\in \C_{/F}$ of the form $\C^{0}_{/p(U)}\times_{\C_{/p(U)}}(\C_{/F})_{/U}$, where $\C^{0}_{/p(U)}\subseteq \C_{/p(U)}$ is a $J$-covering sieve on $p(U)$. Then, $J_{/F}$ is a Grothendieck topology on $\C_{/F}$.

        \item [$(2)$] Suppose that $F$ is a $J$-sheaf. Then, the categorical equivalence $\textnormal{Fun}((\C_{/F})^{\textnormal{opp}},\mS)\xrightarrow{\sim}\textnormal{Fun}(\C^{\textnormal{opp}},\mS)_{/F}$ over $\textnormal{Fun}(\C^{\textnormal{opp}},\mS)$ given by the left Kan extension of the inclusion $\C_{/F}\hookrightarrow \textnormal{Fun}(\C^{\textnormal{opp}},\mS)_{/F}$ along $\C_{/F}\rightarrow \textnormal{Fun}((\C_{/F})^{\textnormal{opp}},\mS)$ restricts to a categorical equivalence $\textnormal{Shv}_{J_{/F}}(\C_{/F})\xrightarrow{\sim}\textnormal{Shv}_{J}(\C)_{/F}$.

        \item [$(3)$] The adjunction $p^*:\textnormal{Fun}(\C^{\textnormal{opp}},\mS)\rightleftarrows \textnormal{Fun}((\C_{/F})^{\textnormal{opp}},\mS):p_*$ restricts to an adjunction $p^*:\textnormal{Shv}_{J}(\C)\rightleftarrows \textnormal{Shv}_{J_{/F}}(\C_{/F}):p_*$. Moreover, when $F$ is $J$-sheaf, the adjunction $p_!:\textnormal{Fun}((\C_{/F})^{\textnormal{opp}},\mS)\rightleftarrows\textnormal{Fun}(\C^{\textnormal{opp}},\mS):p^*$ restricts to an adjunction $p_!:\textnormal{Shv}_{J_{/F}}(\C_{/F})\rightleftarrows \textnormal{Shv}_{J}(\C):p^*$. 

        \item [$(4)$] Whenever $F$ is a $J$-sheaf, the map $p_*:\textnormal{Shv}_{J_{/F}}(\C_{/F})\rightarrow \textnormal{Shv}_{J}(\C)$ is an \'{e}tale morphism of $(\infty,1)$-topoi.

    \end{itemize}

\end{proposition}

\begin{proposition}{}\label{Etale Slices of General Topoi}

    Under the notation of \ref{Etale Slices of Sheaf Topoi}, suppose that $F$ is a $J$-sheaf. Let $S$ be the isomorphism class of a small set of $\infty$-connective morphisms in $\textnormal{Shv}_J(\C)$, and write $\textnormal{Shv}_J(\C)^\wedge_S$ for the full subcategory of $\textnormal{Shv}_J(\C)$ spanned by the $S$-local objects. Let $\textnormal{Shv}_{J_{/F}}(\C_{/F})^\wedge_{S_{/F}}$ denote the full subcategory of $\textnormal{Shv}_{J_{/F}}(\C_{/F})$ spanned by the $(p_!)^{-1}(S)$-local objects. Then:

    \begin{itemize}

        \item [$(1)$] A morphism $f:X\rightarrow Y$ is $\infty$-connective in $\textnormal{Shv}_{J_{/F}}(\C_{/F})$ if and only if $p_!f:p_!X\rightarrow p_!Y$ is $\infty$-connective in $\textnormal{Shv}_J(\C)$. In particular, the inclusion $\textnormal{Shv}_{J_{/F}}(\C_{/F})^\wedge_{S_{/F}}\hookrightarrow \textnormal{Shv}_{J_{/F}}(\C_{/F})$ admits a cotopological left adjoint. 

        \item [$(2)$] The adjunction $p^*:\textnormal{Fun}(\C^{\textnormal{opp}},\mS)\rightleftarrows \textnormal{Fun}((\C_{/F})^{\textnormal{opp}},\mS):p_*$ restricts to an adjunction $p^*:\textnormal{Shv}_{J}(\C)^\wedge_{S}\rightleftarrows \textnormal{Shv}_{J_{/F}}(\C_{/F})^\wedge_{S_{/F}}:p_*$. If $F$ is $S$-local, the adjunction $p_!:\textnormal{Fun}((\C_{/F})^{\textnormal{opp}},\mS)\rightleftarrows\textnormal{Fun}(\C^{\textnormal{opp}},\mS):p^*$ restricts to an adjunction $p_!:\textnormal{Shv}_{J_{/F}}(\C_{/F})^\wedge_{S_{/F}}\rightleftarrows \textnormal{Shv}_{J}(\C)^\wedge_{S}:p^*$. 

        \item [$(3)$] Whenever $F$ is $S$-local, the map $p_*:\textnormal{Shv}_{J_{/F}}(\C_{/F})^\wedge_{S_{/F}}\rightarrow \textnormal{Shv}_{J}(\C)^\wedge_{S}$ is an \'{e}tale morphism of $(\infty,1)$-topoi.

    \end{itemize}

\end{proposition}

\begin{remark}{}\label{Left Kan extensions of slice sheaves are sheaves}

    Under the notation of \ref{Etale Slices of Sheaf Topoi}, let $\C$ be a small $(\infty,1)$-category equipped with a Grothendieck topology $J$, and let $F:\C^{\textnormal{opp}}\rightarrow \mS$ be a map. Let $G:(\C_{/F})^{\textnormal{opp}}\rightarrow \mS$ be an arrow. Then, $p_!G:\C^{\textnormal{opp}}\rightarrow \mS$ may be identified as colimit of the composition $(\C_{/F})_{/G}\rightarrow \C_{/F}\xrightarrow{p}  \C \rightarrow \textnormal{Fun}(\C^{\textnormal{opp}},\mS)$. Since the composite $(\C_{/F})_{/G}\rightarrow \C$ is a right fibration with essentially small fibres, it follows that the colimit of said composition may be identified as the contravariant transport of the right fibration $(\C_{/F})_{/G}\rightarrow \C$ (\citeq{lurie2024kerodon} 03W3). In the case that $F$ is a $J$-sheaf, a consequence of \ref{Etale Slices of Sheaf Topoi} is that $G$ is a $J_{/F}$-sheaf if and only if said contravariant transport $\C^{\textnormal{opp}}\rightarrow \mS$ of the right fibration $(\C_{/F})_{/G}\rightarrow \C$ is a $J$-sheaf (in fact, this statement is equivalent to \ref{Etale Slices of Sheaf Topoi} $(2)$). This is a generalization of a classical result for fibrations in groupoids of $(1,1)$-categories (combine \citeq{stacks-project} 06NX with \citeq{stacks-project} 09WX).

\end{remark}

\begin{remark}{}

    Again under the notation of the notation of \ref{Etale Slices of Sheaf Topoi}, let $G:(\C_{/F})^{\textnormal{opp}}\rightarrow \mS$ be a map. Then, $p_*G:\C^{\textnormal{opp}}\rightarrow \mS$ may be identified as the contravariant transport for the right fibration $\C\times_{\textnormal{Fun}((\C_{/F})^{\textnormal{opp}},\mS)}\textnormal{Fun}((\C_{/F})^{\textnormal{opp}},\mS)_{/G}\rightarrow \C$, where the map $\C\rightarrow \textnormal{Fun}((\C_{/F})^{\textnormal{opp}},\mS)$ is given by the composition $\C\rightarrow \textnormal{Fun}(\C^{\textnormal{opp}},\mS)\xrightarrow{p^*}\textnormal{Fun}((\C_{/F})^{\textnormal{opp}},\mS)$. This follows from the fact that the composition $\textnormal{Fun}(\C^{\textnormal{opp}},\mS)^{\textnormal{opp}}\xrightarrow{(p^*)^{\textnormal{opp}}}\textnormal{Fun}((\C_{/F})^{\textnormal{opp}},\mS)^{\textnormal{opp}}\xrightarrow{h_G}\mS$ is representable by $p_*G$.
    
\end{remark}

\begin{remark}{}\label{Etale Morphisms Section Summary Remark}

    Under the notation of \ref{Etale Slices of General Topoi}, the categorical equivalence $\textnormal{Shv}_J(\C)_S^\wedge\xrightarrow{\sim}\infty\cT op^{\textnormal{\'{e}t}}_{/\textnormal{Shv}_J(\C)_S^\wedge}$ of \citeq{highertopostheory} 6.3.5.10 satisfies the following properties: 

    \begin{itemize}

        \item [$(1)$] Each $F\in \textnormal{Shv}_J(\C)_S^\wedge$ is identified with the \'{e}tale morphism of $(\infty,1)$-topoi $p_*:\textnormal{Shv}_{J_{/F}}(\C_{/F})^\wedge_{S_{/F}}\rightarrow \textnormal{Shv}_J(\C)^\wedge_S$ (up to isomorphism).

        \item [$(2)$] Also up to isomorphism, each map $\eta:G\rightarrow H$ of $S$-local $J$-sheaves is identified with the map $\eta_*:\textnormal{Shv}_{J_{/G}}(\C_{/G})^\wedge_{S_{/G}}\rightarrow \textnormal{Shv}_{J_{/H}}(\C_{/H})^\wedge_{S_{/H}}$ over $\textnormal{Shv}_J(\C)^\wedge_S$ right adjoint to the map $\eta^*:\textnormal{Shv}_{J_{/H}}(\C_{/H})^\wedge_{S_{/H}}\rightarrow \textnormal{Shv}_{J_{/G}}(\C_{/G})^\wedge_{S_{/G}}$, where $\eta^*$ is given by pulling-back along the opposite of the map of right fibrations $\eta\circ:\C_{/G}\rightarrow \C_{/H}$ over $\C$ induced by $\eta$. 

        \item [(3)] The map $\eta_*:\textnormal{Shv}_{J_{/G}}(\C_{/G})^\wedge_{S_{/G}}\rightarrow \textnormal{Shv}_{J_{/H}}(\C_{/H})^\wedge_{S_{/H}}$ in (2) may be identified as the restriction of the map $\eta_*:\textnormal{Fun}((\C_{/G})^{\textnormal{opp}},\mS)\rightarrow \textnormal{Fun}((\C_{/H})^{\textnormal{opp}},\mS)$ right adjoint to the arrow $\eta^*:\textnormal{Fun}((\C_{/H})^{\textnormal{opp}},\mS)\rightarrow \textnormal{Fun}((\C_{/G})^{\textnormal{opp}},\mS)$, where $\eta^*$ is given by pulling-back along the opposite of the map of right fibrations $\eta\circ:\C_{/G}\rightarrow \C_{/H}$ over $\C$ induced by $\eta$.

    \end{itemize}

\end{remark}

\begin{remark}{}

    Let $\C$ be a small $(\infty,1)$-category equipped with a  Grothendieck topology $J$. Let $\scO_\cX:\C^{\textnormal{opp}}\rightarrow \textnormal{CAlg}$ be a $J$-sheaf of $\Ei$-rings on $\C$, which we identify as a sheaf of $\Ei$-rings $\scO_\cX:\textnormal{Shv}_J(\C)^{\textnormal{opp}}\rightarrow \textnormal{CAlg}$ on the $(\infty,1)$-topos $\textnormal{Shv}_J(\C)$. Then, \ref{Etale Slices of Sheaf Topoi} implies that for each $J$-sheaf $F:\C^{\textnormal{opp}}\rightarrow \mS$, the \'{e}tale slice of the spectrally ringed $(\infty,1)$-topos $(\textnormal{Shv}_J(\C),\scO_\cX)$ over $F$ may be identified as the pair $(\textnormal{Shv}_{J_{/F}}(\C_{/F}),\scO_\cX|_F)$, where $\scO_\cX|_F:\textnormal{Shv}_{J_{/F}}(\C_{/F})^{\textnormal{opp}}\rightarrow \textnormal{CAlg}$ is the sheaf of $\Ei$-rings on $\textnormal{Shv}_{J_{/F}}(\C_{/F})$ corresponding to the $J_{/F}$-sheaf given by the composition $(\C_{/F})^{\textnormal{opp}}\rightarrow \C^{\textnormal{opp}}\xrightarrow{\scO_\cX}\textnormal{CAlg}$. 
\end{remark}

\begin{remark}{}\label{Equivalence of slice topoi explicit Remark}

    Under the notation of \ref{Etale Slices of Sheaf Topoi}, Proposition \ref{Equivalence of slice topoi explicit} implies that the equivalence $\theta:\textnormal{Shv}_J(\C)_{/F}\xrightarrow{\sim}\textnormal{Shv}_{J_{/F}}(\C_{/F})$ over $\textnormal{Shv}_J(\C)$ which is homotopy inverse to the equivalence $\textnormal{Shv}_{J_{/F}}(\C_{/F})\xrightarrow{\sim}\textnormal{Shv}_J(\C)_{/F}$ carries an arrow $\eta:E\rightarrow F$ to the composition $(\C_{/F})^{\textnormal{opp}}\hookrightarrow (\textnormal{Shv}_J(\C)_{/F})^{\textnormal{opp}}\xrightarrow{\textnormal{Hom}(-,\eta)}\mS$. In particular, under the informal identification of the objects of $\C_{/F}$ as a pair $(X\in \C,x\in F(X))$, the map $\theta(\eta):(\C_{/F})^{\textnormal{opp}}\rightarrow \mS$ carries each such pair $(X\in \C,x\in F(X))$ to the homotopy fibre of $\eta_X:E(X)\rightarrow F(X)$ over $x\in F(X)$. 
    
\end{remark}

\begin{proposition}{}\label{Equivalence of slice topoi explicit}

    Let $\C$ be a small $(\infty,1)$-category, let $F:\C^{\textnormal{opp}}\rightarrow \mS$ be a presheaf on $\C$, and consider the right fibration $\C_{/F}:=\C\times_{\textnormal{Fun}(\C^{\textnormal{opp}},\mS)}\textnormal{Fun}(\C^{\textnormal{opp}},\mS)_{/F}\rightarrow \C$. Let $L:\textnormal{Fun}((\C_{/F})^{\textnormal{opp}},\mS)\xrightarrow{\sim}\textnormal{Fun}(\C^{\textnormal{opp}},\mS)_{/F}$ be the canonical equivalence over $\textnormal{Fun}(\C^{\textnormal{opp}},\mS)$ specified in \ref{Etale Slices of Sheaf Topoi} (2), and let $R:\textnormal{Fun}(\C^{\textnormal{opp}},\mS)_{/F}\rightarrow \textnormal{Fun}((\C_{/F})^{\textnormal{opp}},\mS)$ be the map given by the composition 

    $$\textnormal{Fun}(\C^{\textnormal{opp}},\mS)_{/F}\xrightarrow{h_\bullet} \textnormal{Fun}((\textnormal{Fun}(\C^{\textnormal{opp}},\mS)_{/F})^{\textnormal{opp}},\mS)\rightarrow \textnormal{Fun}((\C_{/F})^{\textnormal{opp}},\mS).$$

    \noindent Then, $R$ is homotopy inverse to $L$.

\end{proposition}

\begin{proof}

    We first wish to prove that $R$ preserves small colimits. Let $X\in \C$, and let $\eta:h^\C_X\rightarrow F$ be a map. Identifying the pair $(X,\eta)$ as an object of $\C_{/F}$, we observe that the composition 

    $$\textnormal{Fun}((\C_{/F})^{\textnormal{opp}},\mS)\xrightarrow{L}\textnormal{Fun}(\C^{\textnormal{opp}},\mS)_{/F}\xrightarrow{h^\eta}\mS$$

    \noindent may be identified as $\textnormal{ev}_{(X,\eta)}$. Since $L$ is an equivalence, it follows that $h^\eta:\textnormal{Fun}(\C^{\textnormal{opp}},\mS)_{/F}\rightarrow \mS$ preserves small colimits, implying that $R:\textnormal{Fun}(\C^{\textnormal{opp}},\mS)_{/F}\rightarrow \textnormal{Fun}((\C_{/F})^{\textnormal{opp}},\mS)$ preserves small colimits. This implies that the composition $\textnormal{Fun}(\C^{\textnormal{opp}},\mS)_{/F}\xrightarrow{R}\textnormal{Fun}((\C_{/F})^{\textnormal{opp}},\mS)\xrightarrow{L}\textnormal{Fun}(\C^{\textnormal{opp}},\mS)_{/F}$ is the left Kan extension of the composition $\C_{/F}\rightarrow \textnormal{Fun}(\C^{\textnormal{opp}},\mS)_{/F}\xrightarrow{R}\textnormal{Fun}((\C_{/F})^{\textnormal{opp}},\mS)\xrightarrow{L}\textnormal{Fun}(\C^{\textnormal{opp}},\mS)_{/F}$ along $\C_{/F}\rightarrow \textnormal{Fun}(\C^{\textnormal{opp}},\mS)_{/F}$. Similarly, the composition $\textnormal{Fun}((\C_{/F})^{\textnormal{opp}},\mS)\xrightarrow{L}\textnormal{Fun}(\C^{\textnormal{opp}},\mS)\xrightarrow{R}\textnormal{Fun}((\C_{/F})^{\textnormal{opp}},\mS)$ is the left Kan extension of the composition $\C_{/F}\rightarrow \textnormal{Fun}((\C_{/F})^{\textnormal{opp}},\mS)\xrightarrow{L}\textnormal{Fun}(\C^{\textnormal{opp}},\mS)\xrightarrow{R}\textnormal{Fun}((\C_{/F})^{\textnormal{opp}},\mS)$ along $\C_{/F}\rightarrow \textnormal{Fun}((\C_{/F})^{\textnormal{opp}},\mS)$. Thus, it suffices to prove the following two claims:

    \begin{itemize}
    
        \item [$(1)$] The composition $\C_{/F}\rightarrow \textnormal{Fun}(\C^{\textnormal{opp}},\mS)_{/F}\xrightarrow{R}\textnormal{Fun}((\C_{/F})^{\textnormal{opp}},\mS)\xrightarrow{L}\textnormal{Fun}(\C^{\textnormal{opp}},\mS)_{/F}$ is naturally isomorphic to $\C_{/F}\rightarrow \textnormal{Fun}(\C^{\textnormal{opp}},\mS)_{/F}$.

        \item [$(2)$] The composition $\C_{/F}\rightarrow \textnormal{Fun}((\C_{/F})^{\textnormal{opp}},\mS)\xrightarrow{L}\textnormal{Fun}(\C^{\textnormal{opp}},\mS)_{/F}\xrightarrow{R}\textnormal{Fun}((\C_{/F})^{\textnormal{opp}},\mS)$ is naturally isomorphic to the Yoneda embedding $\C_{/F}\rightarrow \textnormal{Fun}((\C_{/F})^{\textnormal{opp}},\mS)$.

    \end{itemize}

    \noindent Claim $(2)$ follows from combining the fact that the composition $\C_{/F}\rightarrow \textnormal{Fun}((\C_{/F})^{\textnormal{opp}},\mS)\xrightarrow{L}\textnormal{Fun}(\C^{\textnormal{opp}},\mS)_{/F}$ is naturally isomorphic to $\C_{/F}\rightarrow \textnormal{Fun}(\C^{\textnormal{opp}},\mS)_{/F}$, with the fact that $\C_{/F}\rightarrow \textnormal{Fun}(\C^{\textnormal{opp}},\mS)_{/F}$ is fully faithful (\citeq{lurie2024kerodon} 03LX). Claim $(1)$ immediately follow from the fact that $R$ is the left Kan extension of the Yoneda embedding $\C_{/F}\rightarrow \textnormal{Fun}((\C_{/F})^{\textnormal{opp}},\mS)$ along $\C_{/F}\rightarrow \textnormal{Fun}(\C^{\textnormal{opp}},\mS)_{/F}$ (again by \citeq{lurie2024kerodon} 03LX).

\end{proof}

\noindent As a final note before beginning our proof of \ref{Etale Slices of Sheaf Topoi}, we demonstrate that not all size conditions asked in \ref{Etale Slices of Sheaf Topoi} are necessary:

\begin{proposition}{}\label{Etale Slices of Sheaf Topoi With no Size Conditions}

    Let $\C$ be a (possibly large) $(\infty,1)$-category equipped with a Grothendieck topology $J$, let $\kappa$ be an uncountable regular cardinal, and let $F:\C^{\textnormal{opp}}\rightarrow \widehat{\mS}$ be a $J$-sheaf taking values in $\kappa$-small spaces. Write $\textnormal{Shv}^J_{<\kappa}(\C)$ for the full subcategory of $\widehat{\textnormal{Shv}}_J(\C)$ spanned by those $E:\C^{\textnormal{opp}}\rightarrow \widehat{\mS}$ taking values in $\mS_{<\kappa}$, and let $\eta:E\rightarrow X$ be a vertex of $\textnormal{Shv}^J_{<\kappa}(\C)_{/F}\subseteq \widehat{\textnormal{Shv}}_J(\C)_{/F}$. Then:

    \begin{itemize}
    
        \item [$(1)$] Up to equivalence, the composition $(\C_{/F})^\textnormal{opp}\hookrightarrow (\widehat{\textnormal{Shv}}_J(\C)_{/F})^\textnormal{opp}\xrightarrow{\textnormal{Hom}(-,\eta)}\widehat{\mS}$ takes values in $\mS_{<\kappa}$. 

        \item [$(2)$] Using (1) to identify the composition $\textnormal{Shv}^J_{<\kappa}(\C)_{/F}\hookrightarrow \widehat{\textnormal{Shv}}_J(\C)_{/F}\xrightarrow{\sim}\widehat{\textnormal{Shv}}_{J_{/F}}(\C_{/F})$ as a map $\theta:\textnormal{Shv}^J_{<\kappa}(\C)_{/F}\rightarrow \textnormal{Shv}^{J_{/F}}_{<\kappa}(\C_{/F})$, we have that $\theta$ is an equivalence.
        
    \end{itemize}

\end{proposition}

\begin{proof}

\citeq{lurie2024kerodon} 05DV implies that given a map of spaces $f:A\rightarrow B$ with $B$ essentially $\kappa$-small, then $A$ is essentially $\kappa$-small if and only if $A$ is essentially $\kappa$-small. Our result now follows from \ref{Equivalence of slice topoi explicit Remark}. 

\end{proof}

\noindent Our proof of \ref{Etale Slices of Sheaf Topoi} will require some preliminaries.

\begin{lemma}{}\label{Sieves and Trivial Kan Fibrations}

    Let $F:\C\xrightarrow{\sim}\cD$ be an equivalence of $(\infty,1)$-categories. Then:

    \begin{itemize}
    
        \item [$(1)$] The construction $\cD^0\mapsto \C\times_\cD\cD^0$ gives a bijection between replete full subcategories of $\cD$ and $\C$. Moreover, an explicit inverse is given by the assignment $\C^0\mapsto \cD\times_\C\C^0$, where $\cD\times_\C\C^0$ is given by pulling-back $\C^0\hookrightarrow \C$ along a map $F^{-1}:\cD\rightarrow \C$ homotopy inverse to $F$. 

        \item [$(2)$] The bijection of $(1)$ restricts to a bijection between sieves on $\cD$ and sieves on $\C$. 
        
    \end{itemize}

\end{lemma}

\begin{proof}

    We first prove (1). Let $\cD^0\subseteq \cD$ be a replete full subcategory (i.e., a full subcategory such that $\cD^0\hookrightarrow \cD$ is an isofibration), and note that $\C\times_\cD\cD^0\subseteq \C$ is a replete full subcategory of $\C$. Write $F^{-1}:\C\xrightarrow{\sim}\cD$ for a homotopy inverse to $F$, and consider now the replete full subcategory $\cD\times_\C(\C\times_\cD\cD^0)\hookrightarrow \cD$. We may identify $\cD\times_\C(\C\times_\cD\cD^0)\hookrightarrow \cD$ as $\cD\times_{F^{-1}\circ F,\cD,i}\cD^0$, where $i:\cD^0\hookrightarrow \cD$ is the inclusion, noting the projection map $\cD\times_{F^{-1}\circ F,\cD,i}\cD^0\rightarrow \cD$ is a replete full subcategory of $\cD$. We will proceed to prove that $\cD\times_{F^{-1}\circ F,\cD,i}\cD^0=\cD^0$ as subsets of $\cD$, which will immediately prove our claim. Let $X\in \cD\times_{F^{-1}\circ F,\cD,i}\cD^0\subseteq \cD$, which implies that $F^{-1}(F(X))\in \cD^0$. Since $F^{-1}(F(X))\cong X$ and $\cD^0\subseteq \cD$ is replete, it follows that $X\in \cD^0$. Thus, $\cD\times_{F^{-1}\circ F,\cD,i}\cD^0\subseteq \cD^0$. Now, let $Y\in \cD^0$, and note that $F^{-1}(F(Y))\cong Y$, implying that $F^{-1}(F(Y))\in \cD^0$. Thus, $Y\in \cD\times_{F^{-1}\circ F,\cD,i}\cD^0$, so $\cD\times_{F^{-1}\circ F,\cD,i}\cD^0=\cD^0$.

    We now prove (2). Let $\cD^0\subseteq \cD$ again be a replete full subcategory of $\cD$. If $\cD^0\subseteq \cD$ is a sieve on $\cD$, then $\C\times_\cD\cD^0\hookrightarrow \C$ is a right fibration, implying that $\C\times_\cD\cD^0\hookrightarrow \C$ is a sieve on $\C$. If $\C\times_\cD\cD^0\hookrightarrow \C$ is a sieve on $\C$, then $\cD\times_\C(\C\times_\cD\cD^0)\hookrightarrow \cD$ is a sieve on $\cD$. Since $\cD^0=\cD\times_\C(\C\times_\cD\cD^0)$, this implies that $\cD^0\subseteq \cD$ is a sieve on $\cD$ if and only if $\C\times_\cD\cD^0\hookrightarrow \C$ is a sieve on $\C$, as desired.

\end{proof}

\begin{lemma}{}\label{Strongly Saturated Classes of Maps Tested By Rep Functors}

    Let $\cX$ be a cocomplete $(\infty,1)$-category with pullbacks, satisfying the property that all small colimits in $\cX$ are universal. Let $\cX_0\subseteq \cX$ be a small full subcategory satisfying the property that every object in $\cX$ may be written as a small colimit of objects in $\cX_0$. Let $\overline{S}$ be a collection of arrows in $\cX$ satisfying the property that $\overline{S}$ is closed under small colimits in $\textnormal{Fun}(\Delta^1,\cX)$, and satisfying the property that $\overline{S}$ is closed under pullbacks in $\cX$. Let $f:X\rightarrow Y$ be a map in $\cX$. Then, the following are equivalent:

    \begin{itemize}
    
        \item [$(1)$] The map $f$ is contained in $\overline{S}$.

        \item [$(2)$] For all $V\in \cX_0$ and maps $g:V\rightarrow Y$, the pullback $V\times_YX\rightarrow V$ is in $\overline{S}$.
        
    \end{itemize}
    
\end{lemma}

\begin{proof}

    Since $\overline{S}$ is closed under pullbacks, we have that (1) implies (2). We now prove the converse. Choose some small diagram $F:K\rightarrow \cX_0$ such that there exists a natural transformation $\eta_Y:F\rightarrow \underline{Y}_K$ in $\textnormal{Fun}(K,\cX)$ exhibiting $Y$ as the colimit of $F$. Let $U:\Delta^1\times \Delta^1\rightarrow \textnormal{Fun}(K,\cX)$ be a diagram witnessing the pullback of $\underline{f}_K:\underline{X}_K\rightarrow \underline{Y}_K$ along $\eta$, which we identify as a map $U^\prime:\Delta^1\rightarrow \textnormal{Fun}(\Delta^1\times K,\cX)$ from some $\eta_X:F^\prime\rightarrow \underline{X}_K$ to $\eta_Y:F\rightarrow \underline{Y}_K$. Since $\eta_Y:F\rightarrow \underline{Y}_K$ exhibits $Y$ as the colimit of $F$, and colimits in $\cX$ are universal, it follows that $\eta_X:F^\prime\rightarrow \underline{X}_K$ exhibits $\underline{X}_K$ as the colimit of a diagram $F^\prime:K\rightarrow \cX$. It follows that we may identify $U^\prime$ as an edge of $\textnormal{Fun}(K,\textnormal{Fun}(\Delta^1,\cX))$ witnessing $f:X\rightarrow Y$ as the colimit of a diagram $\eta^\prime:K\rightarrow \textnormal{Fun}(\Delta^1,\cX)$. It follows from (2) that $\eta^\prime$ takes values in the full subcategory of $\textnormal{Fun}(\Delta^1,\cX)$ spanned by $\overline{S}$, so (1) now follows from the fact that $\overline{S}$ is closed under small colimits.

\end{proof}

\begin{example}{}\label{Strongly Saturated Classes of Maps Tested By Rep Functors Cor}

    Let $\C$ be a small $(\infty,1)$-category, and let $L:\textnormal{Fun}(\C^{\textnormal{opp}},\mS)\rightarrow \cX$ be a left exact reflective localization, which is not necessarily accessible. Then, \ref{Strongly Saturated Classes of Maps Tested By Rep Functors} implies that an edge $\eta:F\rightarrow G$ in $\textnormal{Fun}(\C^{\textnormal{opp}},\mS)$ is an $L$-equivalence if and only if, for all $X\in \C$ and maps $h^\C_X\rightarrow G$, the pullback $h^\C_X\times_GF\rightarrow h^\C_X$ is an $L$-equivalence. 
    
\end{example}

\begin{lemma}{}\label{Reflective Localizations and Slices}

    Let $\C$ be an $(\infty,1)$-category, let $S$ be a collection of arrows in $\C$, and let $\C^\prime\subseteq \C$ be the full subcategory spanned by the $S$-local objects. Let $X\in \C^\prime$, and write $p:\C_{/X}\rightarrow \C$ for the projection. Then, $\C^\prime_{/X}\subseteq \C_{/X}$ may be identified as the full subcategory spanned by the $p^{-1}(S)$-local objects. 

\end{lemma}

\begin{proof}

    Let $h:U\rightarrow X$ be an object in $\C_{/X}$, which we will denote by $\overline{U}$. Let $f:A\rightarrow B$ be a map in $S$. Consider the homotopy commuting square 

    $$
            \begin{tikzpicture}[node distance=1.5cm, auto]
                \node (A) {$ \textnormal{Hom}_\C(B,U)
                $};
                \node (A1) [right of=A] {$ $};
                \node (B) [right of=A1] {$ \textnormal{Hom}_\C(A,U)$};
                \node (C) [below of=A] {$ \textnormal{Hom}_\C(B,X) $};
                \node (D) [below of=B] {$  \textnormal{Hom}_\C(A,X) $};
                \draw[->] (A) to node {$ f^* $} (B);
                \draw[->] (A) to node [swap] {$ h\circ $} (C);
                \draw[->] (C) to node [swap] {$ f^* $} (D);
                \draw[->] (B) to node {$ h\circ $} (D);
            \end{tikzpicture}
    $$

    \noindent This is a homotopy pullback square if and only if, for each $g:B\rightarrow U$ in $\C$, the induced map on homotopy fibres $\textnormal{Hom}_\C(B,U)_{g}\rightarrow \textnormal{Hom}_\C(A,U)_{g\circ f}$ is a homotopy equivalence. Since the bottom arrow is a homotopy equivalence (as $f\in S$ and $X\in \C^\prime$), this implies that $f^*:\textnormal{Hom}_\C(B,U)\rightarrow \textnormal{Hom}_\C(A,U)$ is a homotopy equivalence if and only if each $\textnormal{Hom}_\C(B,U)_{g}\rightarrow \textnormal{Hom}_\C(A,U)_{g\circ f}$ is a homotopy equivalence. Placing $B$ over $X$ via $h\circ g:B\rightarrow X$ (call this $\overline{B}\in \C_{/X}$), and placing $A$ over $X$ via $h\circ g\circ f:A\rightarrow X$ (call this $\overline{A}\in \C_{/X}$), we see that the induced map $\textnormal{Hom}_\C(B,U)_{g}\rightarrow \textnormal{Hom}_\C(A,U)_{g\circ f}$ may be identified as the arrow $\textnormal{Hom}_{\C_{/X}}(\overline{B},\overline{U})\rightarrow \textnormal{Hom}_{\C_{/X}}(\overline{A},\overline{U})$ given by pulling-back along the obvious arrow $\overline{f}:\overline{A}\rightarrow \overline{B}$ satisfying $p(\overline{f})=f$.  Thus, $\overline{U}\in \C_{/X}$ is local with respect to $p^{-1}(S)$ precisely when $U\in \C$ is local with respect to $S$.
    
\end{proof}

\noindent We now prove \ref{Etale Slices of Sheaf Topoi}:

\begin{proof} (\ref{Etale Slices of Sheaf Topoi}.) We first prove $(1)$. Let $V\in \C_{/F}$, noting that $(\C_{/F})_{/V}\subseteq (\C_{/F})_{/V}$ is a $J_{/F}$-sieve on $V$. Choose some $J$-sieve $\C^0_{/p(V)}\subseteq \C_{/p(V)}$ on $p(V)$, and write $(\C_{/F})^0_{/V}:=\C^{0}_{/p(V)}\times_{\C_{/p(V)}}(\C_{/F})_{/V}$. Choose some $f:U\rightarrow V$ in $\C_{/F}$. The elements of the sieve $f^*(\C_{/F})^0_{/V}\subseteq(\C_{/F})_{/U}$ on $U\in \C_{/F}$ may be identified as those maps $g:A\rightarrow U$ in $\C_{/F}$ satisfying the condition that the composition $f\circ g:A\rightarrow V$ (which is well-defined up to homotopy) is an element of $(\C_{/F})^0_{/V}$. Thus for each such $g:A\rightarrow U$ in $f^*(\C_{/F})^0_{/V}$, the composition $p(f)\circ p(g):p(A)\rightarrow p(V)$ may be identified as an object of $\C^{0}_{/p(V)}$, implying that $g:A\rightarrow U$ is in $p(f)^*(\C^{0}_{/p(V)})\times_{\C_{/p(U)}}(\C_{/F})_{/U}$. Next, let $h:B\rightarrow U$ be an arrow in $p(f)^*(\C^{0}_{/p(V)})\times_{\C_{/p(U)}}(\C_{/F})_{/U}$, noting that $p(f)\circ p(h):p(B)\rightarrow p(V)$ is an element of $\C^{0}_{/p(V)}$. This implies that $f\circ h:B\rightarrow V$ is an element of $(\C_{/F})^0_{/V}$, so $h:B\rightarrow U$ is an element of $f^*(\C_{/F})^0_{/V}$. In particular, we have an isomorphism of simplicial sets $f^*(\C_{/F})^0_{/V}=p(f)^*(\C^{0}_{/p(V)})\times_{\C_{/p(U)}}(\C_{/F})_{/U}$. It follows that $f^*(\C_{/F})^0_{/V}\subseteq (\C_{/F})_{/U}$ is a $J_{/F}$-covering sieve. 

Now let $(\C^1_{/F})_{/V}\subseteq (\C_{/F})_{/V}$ be a sieve on $V\in \C_{/F}$, and suppose there exists some $J$-sieve $\C^0_{/p(V)}\subseteq \C_{/p(V)}$ on $p(V)\in \C$ such that, for all $g:A\rightarrow V$ in $(\C_{/F})^0_{/V}:=\C^{0}_{/p(V)}\times_{\C_{/p(V)}}(\C_{/F})_{/V}$, $g^*(\C^1_{/F})_{/V}$ is a $J_{/F}$-sieve on $A\in \C_{/F}$. Choose a section $s:\C_{/p(V)}\rightarrow (\C_{/F})_{/V}$ of the trivial Kan fibration $(\C_{/F})_{/V}\rightarrow \C_{/p(V)}$, and pull-back $(\C^1_{/F})_{/V}\hookrightarrow (\C_{/F})_{/V}$ along $s$ to a sieve $\C^1_{/p(V)}\subseteq \C_{/p(V)}$ on $V$, noting that $(\C^1_{/F})_{/V}=\C^1_{/p(V)}\times_{\C_{/p(V)}}(\C_{/F})_{/V}$ as sieves on $V\in \C_{/F}$ (\ref{Sieves and Trivial Kan Fibrations}). Arguing just as above, we deduce that for all $g:A\rightarrow V$ in $(\C_{/F})^0_{/V}$, $g^*(\C^1_{/F})_{/V}=p(g)^*\C^1_{/p(V)}\times_{\C_{/p(V)}}(\C_{/F})_{/V}$. Since $g^*(\C^1_{/F})_{/V}\in J_{/F}$, it follows from the fact that $(\C_{/F})_{/A}\rightarrow \C_{/p(A)}$ is a trivial Kan fibration that $p(g)^*\C^1_{/p(V)}$ is a $J$-covering sieve on $p(A)$ (\ref{Sieves and Trivial Kan Fibrations}). Thus, $\C^1_{/p(V)}\subseteq \C_{/p(V)}$ is a $J$-covering sieve on $p(V)$, implying that $(\C^1_{/F})_{/V}\subseteq (\C_{/F})_{/V}$ is a $J_{/F}$-covering sieve on $V$, proving $(1)$.

We now show $(2)$. Since $\C_{/F}$ is essentially small, we will prove our claim under the assumption that $\C_{/F}$ is small. Write $L:\textnormal{Fun}(\C^{\textnormal{opp}},\mS)\rightarrow \textnormal{Shv}_J(\C)$ for the $J$-sheafification functor, and let $T$ be the class of $L$-equivalences in $\textnormal{Fun}(\C^{\textnormal{opp}},\mS)$ which are monomorphisms of the form $\eta:F\rightarrow G$, where $G$ is representable. Similarly, write $L^\prime:\textnormal{Fun}((\C_{/F})^{\textnormal{opp}},\mS)\rightarrow \textnormal{Shv}_{J_{/F}}(\C_{/F})$ for the $J_{/F}$-sheafification functor, and write $T^\prime$ for the class of $L^\prime$-equivalences in $\textnormal{Fun}((\C_{/F})^{\textnormal{opp}},\mS)$ which are monomorphisms of the form $\eta^\prime:F^\prime\rightarrow G^\prime$, with $G^\prime$ representable. We now wish to show the following:

\begin{itemize}

    \item [$(*)$] We have an equality $T^\prime=(p_!)^{-1}(T)$. 

\end{itemize}

\noindent Since $p_!:\textnormal{Fun}((\C_{/F})^{\textnormal{opp}},\mS)\rightarrow \textnormal{Fun}(\C^{\textnormal{opp}},\mS)$ factors as an equivalence $\textnormal{Fun}((\C_{/F})^{\textnormal{opp}},\mS)\xrightarrow{\sim}\textnormal{Fun}(\C^{\textnormal{opp}},\mS)_{/F}$, it follows that a map $\eta^\prime:F^\prime\rightarrow G^\prime$ is a monomorphism in $\textnormal{Fun}((\C_{/F})^{\textnormal{opp}},\mS)$ if and only if $p_!(\eta^\prime)$ is a monomorphism in $\textnormal{Fun}(\C^{\textnormal{opp}},\mS)$. Since we also have that $G^\prime\in \textnormal{Fun}((\C_{/F})^{\textnormal{opp}},\mS)$ is representable if and only if $p_!G^\prime\in \textnormal{Fun}(\C^{\textnormal{opp}},\mS)$ is representable, it suffices to prove that, for all $U\in \C_{/F}$, a monomorphism $m^\prime:S^\prime\rightarrow h^{\C_{/F}}_U$ corresponds to a $J_{/F}$-sieve on $U$ under the bijection of \citeq{highertopostheory} 6.2.2.5 precisely when the monomorphism $m:S\rightarrow h^{\C}_{p(U)}$ (isomorphic to $p_!m$) corresponds to a $J$-sieve on $p(U)\in \C$ (see \citeq{highertopostheory} 6.2.2.16). Write $(\C_{/F})^0_{/U}\subseteq (\C_{/F})_{/U}$ for the sieve on $U\in \C_{/F}$ corresponding to $m:S\rightarrow h^{\C_{/F}}_U$, noting that there exists an epi-mono factorization $\coprod_{f:A\rightarrow U\in (\C_{/F})_{/U}^0}h^{\C_{/F}}_A\xrightarrow{e^\prime} S^\prime\xrightarrow{m^\prime}h^{\C_{/F}}_U$ of the canonical map $\coprod_{f:A\rightarrow U\in (\C_{/F})_{/U}^0}h^{\C_{/F}}_A\rightarrow h^{\C_{/F}}_U$ (\citeq{highertopostheory} 6.2.3.18). Since $p_!:\textnormal{Fun}((\C_{/F})^{\textnormal{opp}},\mS)\rightarrow \textnormal{Fun}(\C^{\textnormal{opp}},\mS)$ preserves effective epimorphisms and monomorphisms, $p_!$ preserves epi-mono factorizations. Thus, there exists an epi-mono factorization $\coprod_{f:A\rightarrow U\in (\C_{/F})_{/U}^0}h^{\C_{/F}}_{p(A)}\xrightarrow{e}S\xrightarrow{m} h^{\C_{/F}}_{p(U)}$ in $\textnormal{Fun}(\C^{\textnormal{opp}},\mS)$ of the map $\coprod_{f:A\rightarrow U\in (\C_{/F})_{/U}^0}h^{\C_{/F}}_{p(A)}\rightarrow h^{\C_{/F}}_{p(U)}$. Note that the arrow $e:\coprod_{f:A\rightarrow U\in (\C_{/F})_{/U}^0}h^{\C_{/F}}_{p(A)}\rightarrow S$ factors as $\coprod_{f:A\rightarrow U\in (\C_{/F})_{/U}^0}h^{\C_{/F}}_{p(A)}\rightarrow \coprod_{g:B\rightarrow U\in \C^0_{/p(U)}}h^\C_{B}\xrightarrow{f}S$, where $\C^0_{/p(U)}\subseteq \C_{/p(U)}$ is the sieve on $U$ corresponding to $(\C_{/F})_{/U}^0\subseteq (\C_{/F})_{/U}$ under the bijection of \ref{Sieves and Trivial Kan Fibrations}, and $f:\coprod_{g:B\rightarrow U\in \C^0_{/p(U)}}h^\C_{B}\rightarrow S$ is induced by the canonical map $\coprod_{g:B\rightarrow U\in \C^0_{/p(U)}}h^\C_{B}\rightarrow h^\C_{p(U)}$. Since $e$ is an effective epimorphism, so too is $f$ (\citeq{highertopostheory} 6.2.3.12). Again invoking \citeq{highertopostheory} 6.2.3.18 implies $(*)$.

Write $q:\textnormal{Fun}(\C^{\textnormal{opp}},\mS)_{/F}\rightarrow \textnormal{Fun}(\C^{\textnormal{opp}},\mS)$ for the projection map, and write $\phi:\textnormal{Fun}((\C_{/F})^{\textnormal{opp}},\mS)\xrightarrow{\sim}\textnormal{Fun}(\C^{\textnormal{opp}},\mS)_{/F}$ for the equivalence over $\textnormal{Fun}(\C^{\textnormal{opp}},\mS)$. In the case that $F$ is a $J$-sheaf, \ref{Reflective Localizations and Slices} implies that $\textnormal{Shv}_J(\C)_{/F}\subseteq \textnormal{Fun}(\C^{\textnormal{opp}},\mS)_{/F}$ is the full subcategory of $\textnormal{Fun}(\C^{\textnormal{opp}},\mS)_{/F}$ spanned by the $q^{-1}(T)$-local objects. Since $\textnormal{Shv}_{J_{/F}}(\C_{/F})\subseteq \textnormal{Fun}((\C_{/F})^{\textnormal{opp}},\mS)$ is the full subcategory spanned by the $(p_!)^{-1}(T)$-local objects and $q\circ \phi\cong p_!$, it follows that $\phi$ restricts to an equivalence $\textnormal{Shv}_{J_{/F}}(\C_{/F})\xrightarrow{\sim}\textnormal{Shv}_J(\C)_{/F}$, proving (2). 

We now show (3). Since $p^*$ is right adjoint to $p_!$, it follows from $(*)$ that for all $J$-sheaves $G:\C^{\textnormal{opp}}\rightarrow \mS$, $p^*G:(\C_{/F})^{\textnormal{opp}}\rightarrow \mS$ is a $J_{/F}$-sheaf. Next, let $G^\prime\in \textnormal{Fun}((\C_{/F})^{\textnormal{opp}},\mS)$, and consider $p_*G^\prime:\C^{\textnormal{opp}}\rightarrow \mS$. Since $p_*$ is right adjoint to $p^*$, it suffices to show that, for all $X\in \C$ and monomorphisms $m:S\rightarrow h^\C_X$ in $T$, $p^*(m):p^*(S)\rightarrow p^*(h^\C_X)$ is an $L^\prime$ equivalence. Applying \ref{Strongly Saturated Classes of Maps Tested By Rep Functors Cor}, it now suffices to show that for all $U\in \C_{/F}$ and maps $f:h^{\C_{/F}}_U\rightarrow p^*h^\C_X$, the map $h^{\C_{/F}}_U\times_{p^*h^\C_X}p^*S\rightarrow h^{\C_{/F}}_U$ corresponds to a $J_{/F}$-sieve on $U$. Since $p_!:\textnormal{Fun}((\C_{/F})^{\textnormal{opp}},\mS) \rightarrow \textnormal{Fun}(\C^{\textnormal{opp}},\mS)$ preserves pullbacks, it follows that $p_!(h^{\C_{/F}}_U\times_{p^*h^\C_X}p^*S)\rightarrow p_!h^{\C_{/F}}_U$ may be identified as the pullback of the map $p_!p^*m:p_!p^*S\rightarrow p_!p^*h^\C_X$ along $p_!f:p_!h^{\C_{/F}}_U\rightarrow p_!p^*h^\C_X$. Consider now the diagram

$$
            \begin{tikzpicture}[node distance=1.5cm, auto]
                \node (A) {$ p_! p^*S $};
                \node (A1) [left of=A] {$ $};
                \node (B) [right of=A] {$ S $};
                \node (C) [below of=A] {$ p_!p^*h^\C_X $};
                \node (D) [below of=B] {$ h^\C_X $};
                \node (E) [left of=A1] {$ p_!(h^{\C_{/F}}_U\times_{p^*h^\C_X}p^*S) $};
                \node (F) [below of=E] {$ p_!h^{\C_{/F}}_U $};
                \draw[->] (A) to node {$ $} (B);
                \draw[->] (A) to node [swap] {$ $} (C);
                \draw[->] (C) to node [swap] {$ $} (D);
                \draw[->] (B) to node {$ m $} (D);
                \draw[->] (E) to node {$ $} (A);
                \draw[->] (E) to node {$ $} (F);
                \draw[->] (F) to node {$ $} (C);
            \end{tikzpicture}
    $$

\noindent in $\textnormal{Fun}(\C^{\textnormal{opp}},\mS)$, where the right hand square is induced by the counit $p_!\circ p^*\rightarrow \textnormal{id}_{\textnormal{Fun}(\C^{\textnormal{opp}},\mS)}$. Under the identifications $p_! p^*S\cong F\times S$ and $p_! p^*h^\C_X\cong F\times h^\C_X$, the center vertical map may be identified as $\textnormal{id}_F\times m:F\times S\rightarrow F\times h^\C_X$, the right upper horizontal map may be identified as the projection $F\times S\rightarrow S$ and the right lower horizontal map may be identified as the projection $F\times h^\C_X\rightarrow h^\C_X$, implying that said right hand square is a pullback square. Thus, the outer square of said diagram is a pullback square (\citeq{lurie2024kerodon} 03FZ). Since $p_!h^{\C_{/F}}_U\cong h^\C_{p(U)}$ is representable we have that $p_!(h^{\C_{/F}}_U\times_{p^*h^\C_X}p^*S)\rightarrow p_!h^{\C_{/F}}_U$ is in $T^\prime$, so $(*)$ implies that $h^{\C_{/F}}_U\times_{p^*h^\C_X}p^*S\rightarrow h^{\C_{/F}}_U$ is in $T$, completing the proof that $p_*:\textnormal{Fun}((\C_{/F})^{\textnormal{opp}},\mS)\rightarrow \textnormal{Fun}(\C^{\textnormal{opp}},\mS)$ restricts to a map $\textnormal{Shv}_{J_{/F}}(\C_{/F})\rightarrow \textnormal{Shv}_J(\C)$. Finally, in the case that $F$ is a $J$-sheaf, the fact that $p_!:\textnormal{Fun}((\C_{/F})^{\textnormal{opp}},\mS)\rightarrow \textnormal{Fun}(\C^{\textnormal{opp}},\mS)$ restricts to a map $\textnormal{Shv}_{J_{/F}}(\C_{/F})\rightarrow \textnormal{Shv}_J(\C)$ follows immediately from (2). This completes the proof of (3), and claim (4) follows from combining (2) and (3).

\end{proof}

\noindent We may use the proof of \ref{Etale Slices of Sheaf Topoi}, along with some of the techniques used in the proof of \ref{Etale Slices of Sheaf Topoi}, to deduce \ref{Etale Slices of General Topoi}:

\begin{proof}

    Claim (1) follows from combining \ref{Etale Slices of Sheaf Topoi} (1) with \citeq{highertopostheory} 6.5.1.19. Using (1), (2) may be deduced from a similar logic used in our proof of \ref{Etale Slices of Sheaf Topoi} (3). The fact that $p_*:\textnormal{Shv}_{J_{/F}}(\C_{/F})^\wedge_{S_{/F}}\rightarrow \textnormal{Shv}_{J}(\C)^\wedge_{S}$ is \'{e}tale whenever $F$ is $S$-local now follows from (2).

\end{proof}

\begin{remark}{}

    The proof of \ref{Etale Slices of Sheaf Topoi} $(1)$ does not depend at all on the size conditions of the right fibration $\C_{/F}\rightarrow \C$, and as such an analogous claim holds for an arbitrary right fibration $\cE\rightarrow \cD$, with $\cD$ an $(\infty,1)$-category which is not necessarily small. 
    
\end{remark}

\section{Geometric Categories}\label{Geometric Sites}

\subsection{Geometric Categories and Geometric Sites}\label{Geometric Categories and Geometric Sites}

\begin{definition}{}\label{Geom Cat Definition}

    Let $\kappa$ be a regular cardinal. We call an $(\infty,1)$-category $\C$ \emph{$\kappa$-geometric} if:

    \begin{itemize}

        \item [$(1)$] $\C$ has pullbacks,

        \item [$(2)$] $\C$ admits $\kappa$-small coproducts, and

        \item [$(3)$] Said $\kappa$-small coproducts in $\C$ are disjoint and universal. 

    \end{itemize}

    \noindent In the case that $\kappa=\aleph_0$, we will call an $\aleph_0$-geometric $(\infty,1)$-catgeory $\C$ a \emph{finitary $(\infty,1)$-category}. In the case that $\kappa=\Omega$ is the bounding cardinal of our implicitly chosen universe, we will call $\C$ a \emph{geometric $(\infty,1)$-category}.

\end{definition}

\begin{remark}{}

    Let $\C$ be a finitary $(\infty,1)$-category. Then, $\C$ admits finite limits if and only if $\C$ is \emph{disjunctive} in the sense of \citeq{barwick2014spectralmackeyfunctorsequivariant} 4.2.
    
\end{remark}

\begin{remark}{}

    Let $\C$ be an $(\infty,1)$-category, and let $\kappa$ and $\kappa^\prime$ be regular cardinals satisfying $\kappa\leq \kappa^\prime$. If $\C$ is $\kappa^\prime$-geometric, then $\C$ is $\kappa$-geometric.
    
\end{remark}

\begin{example}{}\label{Geom Cats. Remark}

   The following are several examples of $\kappa$-geometric $(\infty,1)$-categories: 

    \begin{itemize}
        \item [$(1)$] Any $(n,1)$-topos $\cX$ is geometric, for all $n\in \ZZ_{\geq 1}\cup \{\infty\}$. 
    
        \item [$(2)$] Both $\textnormal{CAlg}^{\textnormal{opp}}$ and $\textnormal{CAlg}^{\textnormal{cn},\textnormal{opp}}$ are finitary $(\infty,1)$-categories.

        \item [$(3)$] $\infty\cT op$ is geometric. 

        \item [$(4)$] The $(\infty,1)$-categories $\textnormal{SpDM}$, $\textnormal{SpDM}^{\textnormal{nc}}$, $\textnormal{SpSch}$, and $\textnormal{SpSch}^{\textnormal{nc}}$ are all geometric.

        \item [(5)] If $\C$ is a $\kappa$-geometric $(\infty,1)$-category, then for all $X\in \C$ we have that $\C_{/X}$ is $\kappa$-geometric. 

        \item [(6)] The (ordinary) category of schemes and the (ordinary) category of topological spaces are both geometric.

        \item [(7)] The $(\infty,1)$-category of small $(\infty,1)$-categories is geometric.

    \end{itemize}

\end{example}

\noindent We quickly verify $(3)$ and (part of) $(4)$ of \ref{Geom Cats. Remark}:

\begin{proposition}{}\label{Coproducts of Inf.Topoi}

    Small coproducts in $\infty\cT op$ are disjoint, and colimits of small diagrams in $\infty\cT op$ whose transition maps are \'{e}tale, are universal. 
    
\end{proposition}

\begin{proof}

    Combining \citeq{highertopostheory} 6.5.3.8 with \citeq{highertopostheory} 6.3.5.13, we deduce that $\infty\cT op^{\textnormal{\'{e}t}}$ admits small colimits and pullbacks, which are preserved by $\infty\cT op^{\textnormal{\'{e}t}}\hookrightarrow \infty\cT op$. Since for each $(\infty,1)$-topos $\cX$ we have an equivalence $\cX\cong \infty\cT op^{\textnormal{\'{e}t}}_{/\cX}$, the fact that coproducts in $\infty\cT op$ are disjoint now follows from the fact that for every $(\infty,1)$-topos $\cX$, coproducts in $\cX$ are universal: for a pair of $(\infty,1)$-topoi $\cX$ and $\cY$, the diagram $\Lambda^2_2\rightarrow \infty\cT op^{\textnormal{\'{e}t}}$ exhibiting $\cX\coprod \cY$ as the coproduct of $\cX$ and $\cY$ lifts to a colimit diagram $\Lambda^2_2\rightarrow \infty\cT op^{\textnormal{\'{e}t}}_{/\cX\coprod \cY}$ exhibiting the final object as a coproduct of $\cX\rightarrow \cX\coprod \cY$ and $\cY\rightarrow \cX\coprod \cY$. Extending $\Lambda^2_2\rightarrow \infty\cT op^{\textnormal{\'{e}t}}_{/\cX\coprod \cY}$ to a pullback square  $U:\Delta^1\times \Delta^1\rightarrow \infty\cT op^{\textnormal{\'{e}t}}_{/\cX\coprod \cY}$, we see that $U(0,0)$ is initial in $\infty\cT op^{\textnormal{\'{e}t}}_{/\cX\coprod \cY}$, implying that the composition $\Delta^1\times \Delta^1\xrightarrow{U}\infty\cT op^{\textnormal{\'{e}t}}_{/\cX\coprod \cY}\rightarrow \infty\cT op^{\textnormal{\'{e}t}}$ carries $(0,0)\in \Delta^1\times \Delta^1$ to an initial object in $\infty\cT op^{\textnormal{\'{e}t}}$. This part of our claim now follows from the fact that said composition is a pullback square. 
    
    Next, given a geometric morphism of $(\infty,1)$-topoi $f_*:\cX\rightarrow \cY$, the pullback functor $\infty\cT op_{/\cY}\rightarrow \infty\cT op_{/\cX}$ restricts to a map $\infty\cT op^{\textnormal{\'{e}t}}_{/\cY}\rightarrow \infty\cT op^{\textnormal{\'{e}t}}_{/\cX}$, which upon applying the equivalences $\cX\cong \infty\cT op^{\textnormal{\'{e}t}}_{/\cX}$ and $\cY\cong \infty\cT op^{\textnormal{\'{e}t}}_{/\cY}$, we may identify as the map $f^*:\cY\rightarrow \cX$ left adjoint to $f_*$ (combine the proof of \citeq{highertopostheory} 6.3.5.21 with \citeq{lurie2018sag} 14.4.2.4). This proves the remainder of our claim.

\end{proof}

\begin{corollary}{}\label{Coproducts in SpDM}

    Small coproducts in $\textnormal{SpDM}^{\textnormal{nc}}$ are disjoint, and colimits of small diagrams in $\textnormal{SpDM}^{\textnormal{nc}}$ whose transition maps are \'{e}tale, are universal.

\end{corollary}

\begin{proof}

    Applying \citeq{lurie2018sag} 21.4.6.4, we have that $\textnormal{SpDM}^{\textnormal{nc},\textnormal{\'{e}t}}$ admits small colimits, which are preserved by $\textnormal{SpDM}^{\textnormal{nc},\textnormal{\'{e}t}}\hookrightarrow \textnormal{SpDM}^{\textnormal{nc}}$. Moreover, since the class of \'{e}tale morphisms in $\textnormal{SpDM}^{\textnormal{nc}}$ is stable under pullbacks, the cartesian fibration $\textnormal{ev}_1:\textnormal{Fun}(\Delta^1,\textnormal{SpDM}^{\textnormal{nc}})\rightarrow \textnormal{SpDM}^{\textnormal{nc}}$ restricts to a cartesian fibration $\textnormal{Fun}^{\textnormal{\'{e}t}}(\Delta^1,\textnormal{SpDM}^{\textnormal{nc}})\rightarrow \textnormal{SpDM}^{\textnormal{nc}}$. Since for each $\sX=(\cX,\scO_\cX)\in \textnormal{SpDM}^{\textnormal{nc}}$ the canonical map $(\textnormal{SpDM}^{\textnormal{nc},\textnormal{\'{e}t}})_{/\sX}\rightarrow \infty\cT op^{\textnormal{\'{e}t}}_{/\cX}$ is an equivalence (\citeq{lurie2018sag} 21.4.6.4 (5)) and the map $F:\textnormal{Fun}(\Delta^1,\textnormal{SpDM}^{\textnormal{nc}})\rightarrow \textnormal{SpDM}^{\textnormal{nc}}\times_{\infty\cT op}\textnormal{Fun}^{\textnormal{\'{e}t}}(\Delta^1,\infty\cT op)$ induced by the forgetful functor $\textnormal{SpDM}^{\textnormal{nc}}\rightarrow \infty\cT op$ preserves cartesian arrows (\citeq{lurie2018sag} 21.4.6.7), we see that $F$ is an equivalence of cartesian fibrations over $\textnormal{SpDM}^{\textnormal{nc}}$ (\citeq{lurie2024kerodon} 023M). It now follows from  \ref{Coproducts of Inf.Topoi} that colimits of small diagrams in $\textnormal{SpDM}^{\textnormal{nc}}$ whose transition maps are \'{e}tale, are universal. It remains to show that (small) coproducts in $\textnormal{SpDM}^{\textnormal{nc}}$ are disjoint. Since for each pair $\sX,\sY\in \textnormal{SpDM}^{\textnormal{nc}}$ the arrows $\sX,\sY\rightarrow \sX\coprod\sY$ are \'{e}tale and $\textnormal{SpDM}^{\textnormal{nc}}\rightarrow \infty\cT op$ preserves pullbacks along \'{e}tale morphisms (\citeq{lurie2018sag} 21.4.6.7), the fact that (small) coproducts in $\textnormal{SpDM}^{\textnormal{nc}}$ are disjoint now follows from combining \ref{Coproducts of Inf.Topoi} with the fact that the initial $(\infty,1)$-topos admits an essentially unique structure of a nonconnective spectral Deligne-Mumford stack (which is then necessarily an initial object of $\textnormal{SpDM}^{\textnormal{nc}}$).

\end{proof}

\noindent 

\noindent Such $\kappa$-geometric $(\infty,1)$-categories may all be endowed the following structure of an $(\infty,1)$-site:

\begin{definition}{}\label{Conditions for the kappa Product Topology}

    Fix a regular cardinal $\kappa$, and let $\C$ be a $\kappa$-geometric $(\infty,1)$-category. We will refer to the topology on $\C$ consisting of those $\kappa$-small coverings of the form $\{U_i\rightarrow \coprod_{j\in I}U_j\}_{i\in I}$ as the \emph{$\kappa$-topology} on $\C$. For any $(\infty,1)$-category $\cD$ and map $F:\C^{\textnormal{opp}}\rightarrow \cD$, we will call $F$ a ($\cD$-valued) \emph{$\kappa$-sheaf} if $F$ satisfies descent with respect to the $\kappa$-topology. We will write $\textnormal{Shv}^\kappa_\cD(\C)$ for the $(\infty,1)$-category of $\cD$-valued $\kappa$-sheaves, and when $\cD=\mS$ we will write $\textnormal{Shv}_\kappa(\C)$ for the $(\infty,1)$-category of $\kappa$-sheaves on $\C$. In the case that $\kappa=\aleph_0$, we will refer to the $\aleph_0$-topology on $\C$ as the \emph{finitary topology} or the \emph{$fin$-topology}, and we will call an $\aleph_0$-sheaf $F:\C^{\textnormal{opp}}\rightarrow \cD$ a \emph{finitary sheaf} or a \emph{$fin$-sheaf}.

\end{definition}

\noindent The key result of this section is the following:

\begin{proposition}{}\label{Sheaves for the kappa Product Topology}

    Fix a regular cardinal $\kappa$, and let $\C$ be a $\kappa$-geometric $(\infty,1)$-category. Then, for all $(\infty,1)$-categories $\cD$, a map $F:\C^{\textnormal{opp}}\rightarrow \cD$ is a sheaf for the $\kappa$-topology on $\C$ if and only if $F$ preserves $\kappa$-small products. 
    
\end{proposition}

\begin{proof} Changing universe if necessary, assume that both $\C$ and $\kappa$ are small, and that $\cD$ is locally small. Since $F:\C^{\textnormal{opp}}\rightarrow \cD$ is a sheaf for the $\kappa$-topology on $\C$ if and only if the composition $\C^{\textnormal{opp}}\xrightarrow{F} \cD\xrightarrow{\textnormal{Hom}_\cD(U,-)}\mS$ is a sheaf for the $\kappa$-topology on $\C$ for all $U\in \cD$, and $\C^{\textnormal{opp}}\xrightarrow{F} \cD$ preserves $\kappa$-small products if and only if each such composition $\C^{\textnormal{opp}}\xrightarrow{F} \cD\xrightarrow{\textnormal{Hom}_\cD(U,-)}\mS$ preserves $\kappa$-small products, it suffices to prove our claim in the simplified case that $\cD=\mS$. Applying \ref{Diagram Sheaf condition for sites}, it suffices to prove that a map $G:\C^{\textnormal{opp}}\rightarrow \mS$ preserves $\kappa$-small products if and only if, for all $\kappa$-small sets $\{U_i\in \C\}_{i\in I}$ of objects in $\C$ the composition $\Delta^\triangleleft\xrightarrow{\overline{U}}\textnormal{Fun}(\C^{\textnormal{opp}},\mS)^{\textnormal{opp}}\xrightarrow{\textnormal{Hom}_{\textnormal{Fun}(\C^{\textnormal{opp}},\mS)}(-,G)}\mS$ is a limit diagram, where $\overline{U}$ is the (opposite of the) \v{C}ech nerve of $\coprod_{j\in I}h^\C_{U_j}\rightarrow h^\C_{\coprod_{i\in I}U_i}$. To begin, one inductively shows that (the opposite of) $\overline{U}$ carries the $[n]\rightarrow [-1]$ edge of $(\Delta^{\textnormal{opp}})^\triangleright$ to a map of the form $(\coprod_{i\in I}h^\C_{U_i})\coprod (\coprod_{j\in J}h^\C_{\varnothing_\C})\rightarrow h^\C_{\coprod_{i\in I}{U_i}}$, where $J$ is some (small) set, $\varnothing_\C\in \C$ is an initial object, and $\coprod_{i\in I}h^\C_{U_i}\rightarrow h^\C_{\coprod_{i\in I}{U_i}}$ is the canonical map. If $G:\C^{\textnormal{opp}}\rightarrow \mS$ preserves $\kappa$-small products, then $\textnormal{Hom}_{\textnormal{Fun}(\C^{\textnormal{opp}},\mS)}(\overline{U}(-),G)$ carries each edge of the form $[-1]\rightarrow [n]$ in $\Delta^\triangleleft$ to an isomorphism in $\mS$. By 2-out-of-3, it follows that $\textnormal{Hom}_{\textnormal{Fun}(\C^{\textnormal{opp}},\mS)}(\overline{U}(-),G)$ carries every edge of $\Delta^\triangleleft$ to an isomorphism in $\mS$, and the fact that the composition $\Delta^\triangleleft\xrightarrow{\overline{U}}\textnormal{Fun}(\C^{\textnormal{opp}},\mS)^{\textnormal{opp}}\xrightarrow{\textnormal{Hom}_{\textnormal{Fun}(\C^{\textnormal{opp}},\mS)}(-,G)}\mS$ is a limit diagram may now be deduced from the fact that $\Delta$ is weakly contractible.  

To see the converse, suppose now that the map $G:\C^{\textnormal{opp}}\rightarrow \mS$ is a $\kappa$-sheaf, which implies that $\textnormal{Hom}_{\textnormal{Fun}(\C^{\textnormal{opp}},\mS)}(\overline{U}(-),G):\Delta^\triangleleft\rightarrow \mS$ is a limit diagram (\ref{Diagram Sheaf condition for sites}). Since the empty sieve $\varnothing\hookrightarrow \C_{/\varnothing_\C}$ is a covering sieve for the $\kappa$-topology on $\C$, it follows that $G(\varnothing_\C)\in \mS$ is final. Writing $U$ for the underlying cosimplicial object of $\overline{U}$, it follows that the map $\textnormal{Hom}_{\textnormal{Fun}(\C^{\textnormal{opp}},\mS)}(U(-),G):\Delta\rightarrow \mS$ carries every edge of $\Delta$ to an isomorphism, which immediately implies that $\textnormal{Hom}_{\textnormal{Fun}(\C^{\textnormal{opp}},\mS)}(\overline{U}(-),G):\Delta^\triangleleft\rightarrow \mS$ carries every edge of $\Delta^\triangleleft$ to an isomorphism, proving our claim.

\end{proof}

\begin{remark}{}

    Write $\Delta_{\textnormal{inj}}$ for the subcategory of $\Delta$ spanned by the injective morphisms. It follows from the left cofinality of the inclusion $\Delta_{\textnormal{inj}}\hookrightarrow \Delta$ that the statement analogous to \ref{Prelim. Cech descent} using $\Delta_{\textnormal{inj}}$ in place of $\Delta$ also holds. 
    
\end{remark}

\noindent In particular, we have the following:

\begin{corollary}{}\label{Left Exact Localization for kappa Topology}

    Let $\kappa$ be a small regular cardinal, and let $\C$ be a small $\kappa$-geometic $(\infty,1)$-category. Write $\textnormal{Fun}^{\times,\kappa}(\C^{\textnormal{opp}},\mS)$ for the full subcategory of $\textnormal{Fun}(\C^{\textnormal{opp}},\mS)$ spanned by those maps $F:\C^{\textnormal{opp}}\rightarrow \mS$ which preserve $\kappa$-small products. Then, the localization $L:\textnormal{Fun}(\C^{\textnormal{opp}},\mS)\rightarrow \textnormal{Fun}^{\times,\kappa}(\C^{\textnormal{opp}},\mS)$ is left exact. 

\end{corollary}

\noindent In good cases the converse also holds:

\begin{proposition}{}\label{Converse to Left Exact Localization for kappa Topology}

    Let $\kappa$ be a small regular cardinal, and let $\C$ be a small $(\infty,1)$-category which admits $\kappa$-small coproducts, and has pullbacks. Then, the following are equivalent:

    \begin{itemize}
        \item [$(1)$] The localization $L_\kappa:\textnormal{Fun}(\C^{\textnormal{opp}},\mS)\rightarrow \textnormal{Fun}^{\times,\kappa}(\C^{\textnormal{opp}},\mS)$ is left exact. 

        \item [$(2)$] $\C$ is a $\kappa$-geometric $(\infty,1)$-category. 
        
    \end{itemize}

\end{proposition}

\begin{proof}

    The fact that $(2)$ implies $(1)$ is established in \ref{Left Exact Localization for kappa Topology}. Assume now that $(1)$ holds. We first prove that $\kappa$-small coproducts in $\C$ are universal. Let $\{U_i\}_{i\in I}$ be some $\kappa$-small collection of objects in $\C$, and let $f:X\rightarrow \coprod_{i\in I}U_i=:U$ be a map. We wish to show that, for all $Y\in \C$, the arrow $\textnormal{Hom}_\C(X,Y)\rightarrow \prod_{j\in I}\textnormal{Hom}_\C(X\times_{U} U_j,Y)$ is a homotopy equivalence. Since $h^\C_Y:\C^{\textnormal{opp}}\rightarrow \mS$ preserves $\kappa$-small products, it suffices to show that the arrow $\coprod_{j\in I}(h^\C_{X\times_{U} U_j})\rightarrow h^\C_X$ is an $L_\kappa$-equivalence. However, since $\coprod_{j\in I}(h^\C_{X\times_{U} U_j})\rightarrow h^\C_X$ may be identified as the map $h^\C_{X}\times_{h^\C_{U}}\coprod_{j\in I}h^\C_{U_j}\rightarrow h^\C_{X}$, this follows from the fact that the class of $L_\kappa$-equivalences is stable under pullback (\citeq{highertopostheory} 6.2.1.1).

    Next, let $X,Y\in \C$. Since $h^\C_X\coprod h^\C_Y\rightarrow h^\C_{X\coprod Y}$ is an $L_\kappa$-equivalence and $L_\kappa$ is left exact, it follows that we may identify $L_\kappa(h^\C_X\times_{h^\C_X\coprod h^\C_Y}h^\C_Y)\cong L(\underline{\varnothing}_{\C^{\textnormal{opp}}})$ as the fibre product $h^\C_X\times_{h^\C_{X\coprod Y}}h^\C_Y\cong h^\C_{X\times_{X\coprod Y}Y}$ in $\textnormal{Fun}^{\times,\kappa}(\C^{\textnormal{opp}},\mS)$. Our claim now follows from the fact that $\underline{\varnothing}_{\C^{\textnormal{opp}}}\rightarrow h^\C_{\varnothing_\C}$ is an $L_\kappa$-equivalence, where $\varnothing_\C\in \C$ is initial. 

\end{proof}

\begin{remark}{}

    Proposition \ref{Converse to Left Exact Localization for kappa Topology} admits several variations. In particular, $(1)$ may be replaced with:

    \begin{itemize}

        \item [$(1^\prime)$] \emph{The localization $L_\kappa:\textnormal{Fun}(\C^{\textnormal{opp}},\mS)\rightarrow \textnormal{Fun}^{\times,\kappa}(\C^{\textnormal{opp}},\mS)$ is topological.}
        
    \end{itemize}

\end{remark}

\begin{corollary}{}\label{kappa cats. through sheaves}

    Let $\kappa$ be a (possibly large) regular cardinal, and let $\C$ be an $(\infty,1)$-category which admits $\kappa$-small coproducts, and has pullbacks. Then, the following are equivalent:

    \begin{itemize}
        \item [$(1)$] $\C$ is a $\kappa$-geometric $(\infty,1)$-category. 

        \item [$(2)$] There exists a subcanonical Grothendieck topology $J$ on $\C$ such that all $J$-sheaves preserve $\kappa$-small products. 
        
    \end{itemize}

\end{corollary}

\noindent One may also identify the $\kappa$-topology by means of a universal property:

\begin{proposition}{}\label{Universal Property of the Disjoint Product Topology}

     Fix a small regular cardinal $\kappa$, and let $\C$ be a small $\kappa$-geometric $(\infty,1)$-category. Then, for all cocomplete $(\infty,1)$-categories $\cE$, pulling-back along $\C\rightarrow \textnormal{Shv}_{\kappa}(\C)$ gives a categorical equivalence 

     $$\textnormal{Fun}^{co \ell im}(\textnormal{Shv}_{\kappa}(\C),\cE)\rightarrow \textnormal{Fun}^{\amalg,\kappa}(\C,\cE),$$

     \noindent where $\textnormal{Fun}^{co \ell im}(\textnormal{Shv}_{\kappa}(\C),\cE)$ is the full subcategory of $\textnormal{Fun}(\textnormal{Shv}_{\kappa}(\C),\cE)$ spanned by those maps $F:\textnormal{Shv}_{\kappa}(\C)\rightarrow \cE$ which preserve small colimits, and $\textnormal{Fun}^{\amalg,\kappa}(\C,\cE)$ is the full subcategory of $\textnormal{Fun}(\C,\cE)$ spanned by those maps $\C\rightarrow \cE$ which preserve $\kappa$-small coproducts.

\end{proposition}

\begin{proof}

    The universal property of $\C\rightarrow \textnormal{Fun}(\C^{\textnormal{opp}},\mS)$ yields an equivalence $\textnormal{Fun}^{co \ell im}(\textnormal{Fun}(\C^{\textnormal{opp}},\mS),\cE)\xrightarrow{\sim} \textnormal{Fun}(\C,\cE)$. Moreover, pulling-back along $\textnormal{Fun}(\C^{\textnormal{opp}},\mS)\rightarrow \textnormal{Shv}_{\kappa}(\C)$ yields a fully faithful $(\infty,1)$-functor $\textnormal{Fun}^{co\ell im}(\textnormal{Shv}_{\kappa}(\C),\cE)\rightarrow \textnormal{Fun}^{co \ell im}(\textnormal{Fun}(\C^{\textnormal{opp}},\mS),\cE)$, the essential of which consists of those maps $F:\textnormal{Fun}(\C^{\textnormal{opp}},\mS)\rightarrow \cE$ carrying each $\coprod_{i\in I}h^\C_{U_i}\rightarrow h^\C_{\coprod_{j\in I}U_j}$ to an isomorphism, where $\{U_i\}_{i\in I}$ is some $\kappa$-small collections of objects in $\C$ (\ref{Univ property of Presentable Localizations}). Thus, the composition $\textnormal{Fun}^{co \ell im}(\textnormal{Shv}_{\kappa}(\C),\cE)\rightarrow \textnormal{Fun}^{co \ell im}(\textnormal{Fun}(\C^{\textnormal{opp}},\mS),\cE)\rightarrow \textnormal{Fun}(\C,\cE)$ is fully faithful, and the essential image of said composition consists precisely of those maps $\C\rightarrow \cE$ which preserve $\kappa$-small coproducts, as desired. 
    
\end{proof}

\begin{corollary}{}
    
     Fix a small regular cardinal $\kappa$, and let $\C$ be a small $\kappa$-geometric $(\infty,1)$-category which admits finite limits. Then, for all $(\infty,1)$-topoi $\cY$, pulling-back along $\C\rightarrow \textnormal{Shv}_{\kappa}(\C)$ yields a categorical equivalence 

     $$\textnormal{Fun}^*(\textnormal{Shv}_{\kappa}(\C),\cY)\xrightarrow{\sim}\textnormal{Fun}(\C,\cY)^{\prime},$$

     \noindent where $\textnormal{Fun}(\C,\cY)^\prime\subseteq \textnormal{Fun}(\C,\cY)$ denotes the full subcategory spanned by those maps $F:\C\rightarrow \cY$ which preserve $\kappa$-small coproducts, and are left exact. 
    
\end{corollary}

\begin{proof}

    Applying \citeq{highertopostheory} 6.2.3.20, we see that pulling-back along $\C\rightarrow \textnormal{Shv}_{\kappa}(\C)$ yields a fully faithful $(\infty,1)$-functor $\textnormal{Fun}^*(\textnormal{Shv}_{\kappa}(\C),\cY)\rightarrow \textnormal{Fun}(\C,\cY)$, the essential image of which is spanned by the maps $G:\C\rightarrow \cY$ which are left exact, and satisfy the property that for all $\kappa$-small collections $\{U_i\}_{i\in I}$ of objects in $\C$, the arrow $\coprod_{i\in I}G(U_i)\rightarrow G(\coprod_{j\in I}U_i)$ is an effective epimorphism in $\cY$. Since $\textnormal{Fun}(\C,\cY)^{\prime}\subseteq \textnormal{Fun}(\C,\cY)$ is contained in the essential image of $\textnormal{Fun}^*(\textnormal{Shv}_{\kappa}(\C),\cY)\rightarrow \textnormal{Fun}(\C,\cY)$, it suffices to prove that for left exact left adjoints $(\infty,1)$-topoi $f^*:\textnormal{Shv}_{\kappa}(\C)\rightarrow \cY$, the composition $\C\rightarrow \textnormal{Shv}_{\kappa}(\C)\xrightarrow{f^*}\cY$ preserves $\kappa$-small coproducts. Our claim now follows from \ref{Universal Property of the Disjoint Product Topology}.

\end{proof}

\noindent The geometry of the $(\infty,1)$-topos $\textnormal{Shv}_\kappa(\C)$ is in part captured by the following proposition:

\begin{proposition}{}

    Fix a small regular cardinal $\kappa$, and let $\C$ be a small $\kappa$-geometric $(\infty,1)$-category. Then: 

    \begin{itemize}

        \item [$(1)$] For each $n\in \ZZ$, the map $\tau_{\leq n}:\textnormal{Fun}(\C^{\textnormal{opp}},\mS)\rightarrow \textnormal{Fun}(\C^{\textnormal{opp}},\tau_{\leq n}\mS)$ restricts to the map $\tau_{\leq n}:\textnormal{Shv}_\kappa(\C)\rightarrow \tau_{\leq n}\textnormal{Shv}_\kappa(\C)$.
    
        \item [$(1)$] An arrow $f:X\rightarrow Y$ in $\textnormal{Shv}_\kappa(\C)$ is an effective epimorphism if and only if the image of $f$ under the inclusion $\textnormal{Shv}_\kappa(\C)\hookrightarrow \textnormal{Fun}(\C^{\textnormal{opp}},\mS)$ is an effective epimorphism. 

        \item [$(2)$] The $(\infty,1)$-topos $\textnormal{Shv}_\kappa(\C)$ is hypercomplete.

    \end{itemize}

\end{proposition}

\begin{proof}

Claim $(1)$ follows from the fact that $\tau_{\leq n}:\mS\rightarrow \tau_{\leq n}\mS$ preserves ($\kappa$-)small products. One direction of $(2)$ is obvious. To prove the converse, suppose that $\eta:F\rightarrow G$ is an effective epimorphism in $\textnormal{Shv}_\kappa(\C)$. We wish to show that for each $X\in \C$, the map $\eta_X:F(X)\rightarrow G(X)$ is an effective epimorphism of spaces. It follows from \citeq{cnossen2026httnotes} 3.12 that $\tau_{\leq 0}F\rightarrow \tau_{\leq 0}G$ is an effective epimorphism in $\tau_{\leq 0}\textnormal{Shv}_\kappa(\C)$. Moreover, $(1)$ implies that $\tau_{\leq 0}\textnormal{Shv}_\kappa(\C)\cong \textnormal{Shv}^\kappa_{\textnormal{Sets}}(\textnormal{h}\C)$, and that the canonical map $\textnormal{Shv}_\kappa(\C)\rightarrow \textnormal{Shv}^\kappa_{\textnormal{Sets}}(\textnormal{h}\C)$ carries each $A$ to the arrow $(\textnormal{h}\C)^\textnormal{opp}\rightarrow \textnormal{Sets}$ induced by the composition $\C^{\textnormal{opp}}\xrightarrow{A}\mS\xrightarrow{\pi_0(\cdot)}\textnormal{Sets}$. In particular, these facts together imply that for all $X\in \C$ and $u\in G(X)$ there exists some $\kappa$-small set of maps $\{f_i:X_i\rightarrow X\}_{i\in I}$ in $\C$ such that the arrow $\coprod_{i\in I}X_i\rightarrow X$ is an equivalence, and for each $i\in I$ the image of $u$ in $G(X_i)$ is in the image of $\eta_{X_i}:F(X_i)\rightarrow G(X_i)$. Since $F(X)\rightarrow \prod_{i\in I}F(X_i)$ and $G(X)\rightarrow \prod_{i\in I}G(X_i)$ are equivalences, it immediately follows that $u\in G(X)$ is in the image of $F(X)\rightarrow G(X)$, completing the proof of $(2)$. To see $(3)$, we observe that $(2)$ implies that for each $H\in \textnormal{Shv}_\kappa(\C)$ a simplicial object $U_\bullet:\Delta^{\textnormal{opp}}\rightarrow \textnormal{Shv}_\kappa(\C)_{/H}$ is a hypercovering of $\textnormal{Shv}_\kappa(\C)_{/H}$ (in the sense of \citeq{highertopostheory} 6.5.3.2) if and only if the composition $\Delta^{\textnormal{opp}}\xrightarrow{U_\bullet}\textnormal{Shv}_\kappa(\C)_{/X}\hookrightarrow \textnormal{Fun}(\C^{\textnormal{opp}},\mS)_{/H}$ is a hypercovering of $\textnormal{Fun}(\C^{\textnormal{opp}},\mS)_{/H}$. In particular, the colimit of said composition is final in $\textnormal{Fun}(\C^{\textnormal{opp}},\mS)_{/H}$, implying that the colimit of $U_\bullet$ is final in $\textnormal{Shv}_\kappa(\C)_{/H}$, so $U_\bullet$ is effective (in the sense of \citeq{highertopostheory} 6.5.3.2). Invoking \citeq{highertopostheory} 6.5.3.12 now proves our claim. 
    
\end{proof}

\noindent We now turn our attention to studying an especially well-behaved class of topologies on $\kappa$-geometric $(\infty,1)$-categories.

\begin{definition}{}\label{kappa geometric sites}

    Let $\C$ be a $\kappa$-geometric $(\infty,1)$-category, and let $\tau$ be a Grothendieck pretopology on $\C$. We call a pair $(\C,\tau)$ a \emph{$\kappa$-geometric site} if all coverings in $\tau$ are $\kappa$-small, and if $\tau$ is finer than the $\kappa$-pretopology. When $\kappa=\aleph_0$ we will call an $\aleph_0$-site $(\C,\tau)$ a \emph{finitary site}.
    
\end{definition}

\begin{remark}{}
    A more detailed study of finitary sites $(\C,\tau)$ may be found in section A of the appendix of \citeq{lurie2018sag}.
\end{remark}

\begin{example}{}\label{Geom Sites Examples}

    The following are all examples of $\kappa$-geometric sites:

    \begin{itemize}
    
        \item [$(1)$] The pair $(\cX,\C an(\cX))$ is a geometric site, where $\cX$ is an $(n,1)$-topos for $n\in \ZZ_{\geq 1}\cup \{\infty\}$, and $\C an(\cX)$ is the effective epimorphism pretopology on $\cX$. We will sometimes refer to $\C an(\cX)$ as the \emph{canonical topology} on $\cX$: see \ref{The Canonical Topology on a Topos} for a justification of this name.

        \item [$(2)$] The pair $(\cX,\C an(\cX))$ is generated by a finitary site, where $\cX$ is a coherent $(\infty,1)$-topos.

        \item [$(3)$] The pair $(\textnormal{CAlg}^{\textnormal{cn},\textnormal{opp}},\textnormal{fpqc})$ is a finitary site. So too are each of the other standard pretopologies on affine spectral Deligne-Mumford stacks. Moreover, the analogous statement holds when considering affine nonconnective spectral Deligne-Mumford stacks.

        \item [$(4)$] The pair $(\infty\cT op,\textnormal{\'{e}t})$ is a geometric site, where $\textnormal{\'{e}t}$ denotes the \'{e}tale pretopology on $\infty\cT op$. 

        \item [$(5)$] The \'{e}tale topology on $\textnormal{SpDM}$, $\textnormal{SpDM}^{\textnormal{nc}}$, $\textnormal{SpSch}$, and $\textnormal{SpSch}^{\textnormal{nc}}$ are each generated by a geometric site. The analogous statements hold for the Zariski, Nisnevich, and fpqc topologies.

        \item [$(6)$] Let $(\C,\tau)$ be a $\kappa$-geometric site, where $\C$ is a locally small $(\infty,1)$-category. Let $F:\C^{\textnormal{opp}}\rightarrow \mS$ be a $\kappa$-sheaf on $\C$ classified by the right fibration $\C\times_{\textnormal{Fun}(\C^{\textnormal{opp}},\mS)}\textnormal{Fun}(\C^{\textnormal{opp}},\mS)_{/F}=:\C_{/F}\xrightarrow{p} \C$, and write $\tau_{/F}$ for the pretopology on $\C_{/F}$ where $\{f_i:U_i\rightarrow X\}_{i\in I}\in \tau_{/F}$ precisely when $\{p(f_i):p(U_i)\rightarrow p(X)\}_{i\in I}\in \tau$. Then, $(\C_{/F},\tau_{/F})$ is a $\kappa$-geometric site. 
        
    \end{itemize}

\end{example}

\noindent The following lemma (which we will make important use of in \ref{Global Sheaves Section}) may be used to verify \ref{Geom Sites Examples} (6):

\begin{lemma}{}\label{Colimits in Right Fibrations}

    Let $\C$ be a locally small $(\infty,1)$-category and let $p:\cE\rightarrow \C$ be a right fibration with essentially small fibres classifying a presheaf $F:\C^{\textnormal{opp}}\rightarrow \mS$. Let $K$ be a small simplicial set, let $U:K\rightarrow \C$ be a diagram admitting a colimit $X\in \C$, and let $\overline{U}:K^\triangleright\rightarrow \C$ be a diagram exhibiting $X\in \C$ as said colimit of $F$. Then, the following are equivalent: 

    \begin{itemize}
    
        \item [$(1)$] Every lift of $U$ along $p$ to an arrow $K\rightarrow \cE$ admits an extension to diagram $K^\triangleright \rightarrow \cE$ over $\overline{U}$, and every such extension is a colimit diagram.

        \item [$(2)$] Every lift of $U$ along $p$ to an arrow $\widetilde{U}:K\rightarrow \cE$ may be extended to a colimit diagram $K^\triangleright\rightarrow \cE$, and an arbitrary extension $Z:K^\triangleright\rightarrow \cE$ of $\widetilde{U}$ is a colimit diagram if and only if the composition $K^\triangleright\xrightarrow{Z} \cE\xrightarrow{p}\C$ is a colimit diagram. 

        \item [$(3)$] The composition $(K^{\textnormal{opp}})^\triangleleft\xrightarrow{\overline{U}^{\textnormal{opp}}}\C^{\textnormal{opp}}\xrightarrow{F}\mS$ is a limit diagram. 
        
    \end{itemize}

    \noindent Suppose now that $\C$ admits $K$-indexed colimits. Then, the following are equivalent: 

    \begin{itemize}
    
        \item [$(1^\prime)$] $\cE$ admits $K$-indexed colimits, and a diagram $\overline{V}:K^\triangleright\rightarrow \cE$ is a colimit diagram if and only if the composition $K^\triangleright\xrightarrow{\overline{V}} \cE\xrightarrow{p}\C$ is a colimit diagram. 

        \item [$(2^\prime)$] $F$ preserves $K^{\textnormal{opp}}$-indexed limits. 
        
    \end{itemize}

\end{lemma}

\noindent We will prove \ref{Colimits in Right Fibrations} at the end of this section.

\begin{remark}{}\label{Induced Smaller Geometric Sites}

    Let $\kappa$ and $\kappa^\prime$ be regular cardinals, and assume that $\kappa^\prime\leq \kappa$. Then, for any $\kappa$-geometric site $(\C,\tau)$, we have that $(\C,\tau^\prime)$ is a $\kappa^\prime$-geometric site, where $\tau^\prime$ is the pretopology on $\C$ whose coverings are the subset of $\tau$-coverings $\{U_i\rightarrow X\}_{i\in I}$ where $I$ is $\kappa^\prime$-small. 
    
\end{remark}

Under the notation of \ref{Geom Sites Examples}, one can use the following to classify sheaves on $\kappa$-geometric sites of the form $(\C_{/F},\tau_{/F})$, whenever $F$ is a $\tau$-sheaf:

\begin{lemma}{}\label{Sheaves on Slices of Sites}

    Let $\C$ be a small $(\infty,1)$-category with pullbacks, and let $\tau$ be a pretopology on $\C$. Let $F:\C^{\textnormal{opp}}\rightarrow \mS$ be a $\tau$-sheaf, and write $p:\C_{/F}\rightarrow \C$ for the right fibration classifying $F$. Write $\tau_{/F}$ for the pretopology on $\C_{/F}$ whose coverings are those families of maps $\{f_i:U_i\rightarrow U\}_{i\in I}$ in $\C_{/F}$ satisfying the property that $\{p(f_i):p(U_i)\rightarrow p(U)\}_{i\in I}$ is a $\tau$-covering. Then, the equivalence of $(\infty,1)$-categories $\textnormal{Fun}(\C^{\textnormal{opp}},\mS)_{/F}\xrightarrow{\sim}\textnormal{Fun}((\C_{/F})^{\textnormal{opp}},\mS)$ of \ref{Equivalence of slice topoi explicit} restricts to an equivalence $\textnormal{Shv}_\tau(\C)_{/F}\xrightarrow{\sim}\textnormal{Shv}_{\tau_{/F}}(\C_{/F})$.
    
\end{lemma}

\begin{proof}

    Combine \ref{Etale Slices of Sheaf Topoi} with \ref{Equivalence of slice topoi explicit}.
    
\end{proof}

\begin{corollary}{}\label{Slices of sheaves of eff. epis.}

    Let $n\in \ZZ_{\geq 1}\cup \{\infty\}$, let $\cX$ be an $(n,1)$-topos, and let $Z\in \cX$. Then, the canonical equivalence $\textnormal{Fun}((\cX_{/Z})^{\textnormal{opp}},\widehat{\mS})\xrightarrow{\sim}\textnormal{Fun}(\cX^{\textnormal{opp}},\widehat{\mS})_{/h_Z}$ restricts to an equivalence $\widehat{\textnormal{Shv}}_{\C an(\cX_{/Z})}(\cX_{/Z})\xrightarrow{\sim}\widehat{\textnormal{Shv}}_{\C an(\cX)}(\cX)_{/h^\cX_Z}$.
    
\end{corollary}

\begin{corollary}{}\label{Slices of kappa sheaf topoi}

    Fix a small regular cardinal $\kappa$, let $\C$ be a small $\kappa$-geometric $(\infty,1)$-category, and let $F:\C^{\textnormal{opp}}\rightarrow \mS$ be a map which preserves $\kappa$-small products. Then, the canonical categorical equivalence $\textnormal{Fun}(\C^{\textnormal{opp}},\mS)_{/F}\xrightarrow{\sim}\textnormal{Fun}((\C_{/F})^{\textnormal{opp}},\mS)$ restricts to a categorical equivalence $\textnormal{Shv}_{\kappa}(\C)_{/F}\xrightarrow{\sim}\textnormal{Shv}_{\kappa}(\C_{/F})$.
    
\end{corollary}

\begin{proof}

    (\ref{Colimits in Right Fibrations}.) To begin, we show that $(1)$ is equivalent to $(2)$. Supposing that $(1)$ holds, choose some lift $\widetilde{U}:K\rightarrow \cE$ of $U$ along $p$ (noting if no such $\widetilde{U}$ exists, then $(2)$ is vacuously true). It follows immediately that $\widetilde{U}$ may be extended to a colimit diagram $Z^\prime:K^\triangleright\rightarrow \cE$. Next, suppose we are given an arbitrary extension $Z:K^\triangleright\rightarrow \cE$ of $\widetilde{U}$. If the composition $K^\triangleright\xrightarrow{Z} \cE\xrightarrow{p}\C$ is \emph{not} a colimit diagram, then $Z$ is not isomorphic to $Z^\prime$, and as such, is not a colimit diagram. Suppose now that the composition $K^\triangleright\xrightarrow{Z} \cE\xrightarrow{p}\C$ is a colimit diagram, and consider the right fibration $\textnormal{Fun}(K^\triangleright,\cE)\rightarrow \textnormal{Fun}(K^\triangleright,\C)$. There exists an isomorphism $\eta:p\circ Z\xrightarrow{\sim}\overline{U}$, which yields an isomorphism in $\textnormal{Fun}(K^\triangleright,\cE)$ from $Z$ to some arrow $K^\triangleright\rightarrow \cE$ over $\overline{U}$ (since right fibrations are conservative, see \citeq{lurie2024kerodon} 019K). Next, suppose that $(2)$ holds, and choose some lift $\widetilde{U}:K\rightarrow \cE$ of $U$ along $p$. In order to prove that $(1)$ holds, it suffices to show that there exists an extension of $\widetilde{U}$ over $\overline{U}:K^\triangleright\rightarrow \C$. However, extending $\widetilde{U}:K\rightarrow \cE$ to a colimit diagram $Z:K^\triangleright\rightarrow \cE$, we have an isomorphism of diagram $K^\triangleright\rightarrow \cE$ from $\overline{U}$ to $p\circ Z$, so the fact that $\textnormal{Fun}(K^\triangleright,\cE)\rightarrow \textnormal{Fun}(K^\triangleright,\C)$ is a right fibration proves our claim.

    We now show that $(1)$ implies $(3)$. Suppose that $(1)$ holds. Combining \citeq{lurie2024kerodon} 02KR and \citeq{lurie2024kerodon} 02ZA, we deduce that every lift of $\overline{U}$ along $p$ to an arrow $K^\triangleright\rightarrow \cE$ is $p$-left Kan extended from $K\subset K^\triangleright$. Consider now the map $\textnormal{Fun}(K^\triangleright,\cE)\rightarrow \textnormal{Fun}(K,\cE)\times_{\textnormal{Fun}(K,\C)}\textnormal{Fun}(K^\triangleright,\C)$, and pull-back along the inclusion $\textnormal{Fun}_{/\C}(K,\cE)=\textnormal{Fun}(K,\cE)\times_{\textnormal{Fun}(K,\C)}\{\overline{U}\}\hookrightarrow \textnormal{Fun}(K,\cE)\times_{\textnormal{Fun}(K,\C)}\textnormal{Fun}(K^\triangleright,\C)$ to yield the restriction map $\phi:\textnormal{Fun}_{/\C}(K^\triangleright,\cE)\rightarrow \textnormal{Fun}_{/\C}(K,\cE)$. It now follows from  \citeq{lurie2024kerodon} 030R that $\phi$ is a trivial Kan fibration, so $(3)$ follows from combining \citeq{highertopostheory} 3.3.3.3 with \citeq{lurie2024kerodon} 04CE. Next we wish to show that $(3)$ implies $(1)$, so suppose that $(3)$ holds. Combining \citeq{highertopostheory} 3.3.3.3 with \citeq{lurie2024kerodon} 04CE implies that $\phi:\textnormal{Fun}_{/\C}(K^\triangleright,\cE)\rightarrow \textnormal{Fun}_{/\C}(K,\cE)$ is a trivial Kan fibration. If $\textnormal{Fun}_{/\C}(K,\cE)$ is empty then $(1)$ holds vacuously, so suppose that $\textnormal{Fun}_{/\C}(K,\cE)$ is nonempty. Choosing now some map $G:K\rightarrow \cE$ over $U:K \rightarrow \C$, which we extend to a map $\overline{G}:K^\triangleright\rightarrow \cE$ over $\overline{U}$ (noting such a map $\overline{G}$ exists since $\phi$ is a trivial Kan fibration), it remains to show that $\overline{G}:K^\triangleright\rightarrow \cE$ is a colimit diagram. Since it now suffices to prove that any diagram isomorphic to $\overline{G}:K^\triangleright\rightarrow \cE$ is a colimit diagram, it suffices to prove that every lift of $h^\C_\bullet\circ \overline{U}:K^\triangleright \rightarrow \textnormal{Fun}^{\textnormal{rep}}(\C^{\textnormal{opp}},\mS)$ along the right fibration $\textnormal{Fun}^{\textnormal{rep}}(\C^{\textnormal{opp}},\mS)_{/F}:=\textnormal{Fun}^{\textnormal{rep}}(\C^{\textnormal{opp}},\mS)\times_{\textnormal{Fun}(\C^{\textnormal{opp}},\mS)}\textnormal{Fun}(\C^{\textnormal{opp}},\mS)_{/F}\rightarrow \textnormal{Fun}^{\textnormal{rep}}(\C^{\textnormal{opp}},\mS)$ is a colimit diagram. Let $H\in \textnormal{Fun}(\C^{\textnormal{opp}},\mS)$ be the colimit of $h^\C_\bullet\circ U$, which yields a canonical map $f:H\rightarrow h^\C_X$. Writing $\textnormal{Fun}(\C^{\textnormal{opp}},\mS)^\prime$ for the full subcategory of $\textnormal{Fun}(\C^{\textnormal{opp}},\mS)$ which are $f$-local, we deduce that the map $h^\C_\bullet\circ \overline{U}:K^\triangleright\rightarrow \textnormal{Fun}(\C^{\textnormal{opp}},\mS)^\prime$ is a colimit diagram. Noting also that $F\in \textnormal{Fun}(\C^{\textnormal{opp}},\mS)^\prime$, it follows from \citeq{lurie2024kerodon} 02KC that every lift of $h^\C_\bullet\circ \overline{U}:K^\triangleright\rightarrow \textnormal{Fun}(\C^{\textnormal{opp}},\mS)^\prime$ along $\textnormal{Fun}(\C^{\textnormal{opp}},\mS)^\prime_{/F}\rightarrow \textnormal{Fun}(\C^{\textnormal{opp}},\mS)^\prime$ is a colimit diagram. Our claim now follows from the fact that the map $h^\C_\bullet\circ \overline{U}:K^\triangleright\rightarrow \textnormal{Fun}(\C^{\textnormal{opp}},\mS)^\prime$ takes values in $\textnormal{Fun}^{\textnormal{rep}}(\C^{\textnormal{opp}},\mS)\subseteq \textnormal{Fun}(\C^{\textnormal{opp}},\mS)^\prime$.

    The equivalence of $(1^\prime)$ and $(2^\prime)$ now follows from the equivalence of $(2)$ and $(3)$, where we allow $U$ to vary over $\textnormal{Fun}(K,\C)$. 
    
\end{proof}

\subsection{Application: External vs. Internal Group Actions}\label{App. Ext vs Int}

Whenever considering an $(\infty,1)$-category $\C$, there are at least two possible interpretations of a ``group action on an object $X\in \C$'':

\begin{itemize}

    \item [$(1)$] An \emph{external} action of a discrete group $G$ on $X\in \C$ via an $(\infty,1)$-functor $\rho:\textnormal{B}G\rightarrow \C$ which carries the lone vertex of $\textnormal{B}G$ to $X\in \C$. 

    \item [$(2)$] An \emph{internal} action of a group object $G:\Delta^{\textnormal{opp}}\rightarrow \C$ on $X\in \C$ via some diagram $\rho:\Delta^{\textnormal{opp}}\times \Delta^1\rightarrow \C$ in $\textnormal{LMon}_G(\C)$, where the underlying $\C$-object of $\rho\in \textnormal{LMon}_G(\C)$ may be identified as $X$. 

\end{itemize}

\noindent In many classical geometric settings, $(2)$ is a generalization of $(1)$:

\begin{example}{}\label{Discrete Group Actions in Ord. Geometric Categories}

\noindent For a discrete group $G$, we have the following:

    \begin{itemize}
    
        \item [\textnormal{(i)}] The category of (Borel) $G$-equivariant topological spaces is canonically equivalent to the category of topological spaces equipped with a left $G^{top}$-action, where $G^{top}$ denotes the discrete topological group associated to $G$. 

        \item [\textnormal{(ii)}] For a commutative ring $R$, the category of schemes over $R$ equipped with a left $G$-action is canonically equivalent to the the category of $R$-schemes equipped with a left $G_{R}$-action, where $G_{R}$ denotes the discrete group scheme over $R$ associated to $G$.
        
    \end{itemize} 

\noindent Informally, \textnormal{(i)} identifies a Borel equivariant $G$-space $X$ with the left $G^{top}$-space $a:G^{top}\times X\rightarrow X$, where said action map $a:G^{top}\times X\rightarrow X$ is identified as the composition $G^{top}\times X\xrightarrow{\sim}\coprod_{g\in G}X\rightarrow X$, where for each $g\in G$, the corresponding map $X\rightarrow X$ is induced by the left action of $g\in G$ on $X$. The identification given in \textnormal{(ii)} is similar. 

\end{example}

As an application of \ref{Sheaves for the kappa Product Topology}, one can show that the relationships of \ref{Discrete Group Actions in Ord. Geometric Categories} carry over to the $(\infty,1)$-categorical setting:

\begin{proposition}{}\label{Internal Vs External Group Objects Theorem}

    Fix a (possibly large) regular cardinal $\kappa$, and let $\C$ be a $\kappa$-geometric $(\infty,1)$-category with final object $\bm{1}_\C\in \C$. Let $G$ be a $\kappa$-small discrete group, and write $G_\C$ for the group object in $\C$ given by the canonical group object structure on $\coprod_{g\in G}\bm{1}_\C\in \C$. Then, there exists a categorical equivalence $\Phi:\textnormal{Fun}(\textnormal{B}G,\C)\xrightarrow{\sim} \textnormal{LMon}_{G_\C}(\C)$ satisfying the property that the composition $\textnormal{Fun}(\textnormal{B}G,\C)\xrightarrow{\Phi} \textnormal{LMon}_{G_\C}(\C)\xrightarrow{F}\C$ is naturally isomorphic to $\textnormal{ev}_*:\textnormal{Fun}(\textnormal{B}G,\C)\rightarrow \C$, where $F:\textnormal{LMon}_{G_\C}(\C)\rightarrow \C$ is the canonical forgetful functor and $*\in \textnormal{B}G$ is the lone vertex.

\end{proposition}

\begin{proof} Just as in the proof of \ref{Sheaves for the kappa Product Topology}, we may reduce to the case that $\C$ and $\kappa$ are both small. Consider the fully faithful $(\infty,1)$-functor $\phi:\textnormal{Fun}(\textnormal{B}G,\C)\rightarrow \textnormal{Fun}(\C^{\textnormal{opp}},\mS)_{/\underline{\textnormal{B}G}_{\C^{\textnormal{opp}}}}$ given by the composition $\textnormal{Fun}(\textnormal{B}G,\C)\rightarrow\textnormal{Fun}(\C^{\textnormal{opp}},\textnormal{Fun}(\textnormal{B}G,\mS))\xrightarrow{\sim}\textnormal{Fun}(\C^{\textnormal{opp}},\mS)_{/\underline{\textnormal{B}G}_{\C^{\textnormal{opp}}}}$. Since $\underline{\textnormal{B}G}_{\C^{\textnormal{opp}}}$ along with its canonical base point $\underline{\Delta}_{\C^{\textnormal{opp}}}\rightarrow \underline{\textnormal{B}G}_{\C^{\textnormal{opp}}}$ may be identified as an object of $\textnormal{Fun}(\C^{\textnormal{opp}},\mS)^{\leq 1}_*$, one can use \citeq{lurie2017higheralgebra} 5.2.6.28 to induce an equivalence $\theta:\textnormal{Fun}(\C^{\textnormal{opp}},\mS)_{/\underline{\textnormal{B}G}_{\C^{\textnormal{opp}}}}\xrightarrow{\sim}\textnormal{LMon}_{\underline{G}_{_{\C^{\textnormal{opp}}}}}(\textnormal{Fun}(\C^{\textnormal{opp}},\mS))$. Moreover, said equivalence $\theta:\textnormal{Fun}(\C^{\textnormal{opp}},\mS)_{/\underline{\textnormal{B}G}_{\C^{\textnormal{opp}}}}\xrightarrow{\sim}\textnormal{LMon}_{\underline{G}_{_{\C^{\textnormal{opp}}}}}(\textnormal{Fun}(\C^{\textnormal{opp}},\mS))$ is over $\textnormal{Fun}(\C^{\textnormal{opp}},\mS)$, where $\textnormal{Fun}(\C^{\textnormal{opp}},\mS)_{/\underline{\textnormal{B}G}_{\C^{\textnormal{opp}}}}\rightarrow \textnormal{Fun}(\C^{\textnormal{opp}},\mS)$ is induced by basechange along $\underline{\Delta^0}_{{\C^{\textnormal{opp}}}}\rightarrow \underline{\textnormal{B}G}_{\C^{\textnormal{opp}}}$ (see \citeq{lurie2017higheralgebra} 5.2.6.29). We next wish to verify that the map $\textnormal{Fun}(\C^{\textnormal{opp}},\textnormal{Fun}(\textnormal{B}G,\mS))\xrightarrow{\sim}\textnormal{Fun}(\C^{\textnormal{opp}},\mS)_{/\underline{\textnormal{B}G}_{\C^{\textnormal{opp}}}}$, which we will denote by $\gamma$, is over $\textnormal{Fun}(\C^{\textnormal{opp}},\mS)$. It suffices to show that the composition $\textnormal{Fun}(\textnormal{B}G,\mS)\xrightarrow{\sim}\mS_{/\textnormal{B}G}\xrightarrow{\Delta^0\times_{\textnormal{B}G}(\cdot)}\mS_{/\Delta^0}\xrightarrow{\sim}\mS$ is corepresented by the lone vertex $*\in \textnormal{B}G$. It follows from \citeq{highertopostheory} 5.5.2.9 that said composition is a right adjoint, implying the existence of a left adjoint $L:\mS\rightarrow \textnormal{Fun}(\textnormal{B}G,\mS)$. Combining the fact that the map left adjoint to the composition $\mS_{/\textnormal{B}G}\xrightarrow{\Delta^0\times_{\textnormal{B}G}(\cdot)}\mS_{/\Delta^0}\xrightarrow{\sim}\mS$ carries $\Delta^0$ to an object (isomorphic to) $\Delta^0\rightarrow \textnormal{B}G$, along with the observation that the categorical equivalence $\mS_{/\textnormal{B}G}\xrightarrow{\sim}\textnormal{LFib}(\textnormal{B}G)$ carries $*\rightarrow \textnormal{B}G$ to a left fibration $U_*\rightarrow \textnormal{B}G$ corepresentable by $*\in \textnormal{B}G$ (see \citeq{lurie2024kerodon} 043M), we have that $L$ carries $\Delta^0$ to a functor corepresented by $*$, proving this part of our claim. Thus, the composition 

$$\textnormal{Fun}(\textnormal{B}G,\textnormal{Fun}(\C^{\textnormal{opp}},\mS))\xrightarrow{\gamma}\textnormal{Fun}(\C^{\textnormal{opp}},\mS)_{/\underline{\textnormal{B}G}_{\C^{\textnormal{opp}}}}\xrightarrow{\theta} \textnormal{LMon}_{\underline{G}_{_{\C^{\textnormal{opp}}}}}(\textnormal{Fun}(\C^{\textnormal{opp}},\mS))$$

\noindent is over $\textnormal{Fun}(\C^{\textnormal{opp}},\mS)$. Moreover, $\theta\circ \phi:\textnormal{Fun}(\textnormal{B}G,\C)\rightarrow \textnormal{LMon}_{\underline{G}_{_{\C^{\textnormal{opp}}}}}(\textnormal{Fun}(\C^{\textnormal{opp}},\mS))$ is a fully faithful embedding, whose essential image is precisely $\textnormal{LMon}_{\underline{G}_{_{\C^{\textnormal{opp}}}}}(\textnormal{Fun}^{\textnormal{rep}}(\C^{\textnormal{opp}},\mS))\subseteq \textnormal{LMon}_{\underline{G}_{_{\C^{\textnormal{opp}}}}}(\textnormal{Fun}(\C^{\textnormal{opp}},\mS))$.

Next, since the sheafification functor $L_\kappa:\textnormal{Fun}(\C^{\textnormal{opp}},\mS)\rightarrow \textnormal{Shv}_{\kappa}(\C)$ is symmetric monoidal with respect to the relevant cartesian symmetric monoidal structures, the map $\Psi:\textnormal{Fun}(\textnormal{B}G,\C)\rightarrow \textnormal{LMon}_{L_\kappa(\underline{G}_{_{\C^{\textnormal{opp}}}})}(\textnormal{Shv}_\kappa(\C))$ given by the composition 

$$\textnormal{Fun}(\textnormal{B}G,\C)\xrightarrow{\theta\circ \phi}\textnormal{LMon}_{\underline{G}_{_{\C^{\textnormal{opp}}}}}(\textnormal{Fun}(\C^{\textnormal{opp}},\mS))\xrightarrow{L_\kappa^\times}\textnormal{LMon}_{L_\kappa(\underline{G}_{_{\C^{\textnormal{opp}}}})}(\textnormal{Shv}_\kappa(\C))$$

\noindent induces an $(\infty,1)$-functor $\textnormal{Fun}(\textnormal{B}G,\C)\rightarrow \textnormal{LMon}_{h^\C_{G_\C}}(\textnormal{Shv}_\kappa(\C))$ (here we have used the fact that the canonical map of group objects $\underline{G}_{_{\C^{\textnormal{opp}}}}\rightarrow h^\C_{G_\C}$ is an $L_\kappa$-equivalence). Applying \ref{Localization of Modules}, we deduce that $\Psi$ is fully faithful with essential image $\textnormal{LMon}_{h^\C_{G_\C}}(\textnormal{Shv}^{\textnormal{rep}}_\kappa(\C))$, where $\textnormal{Shv}^{\textnormal{rep}}_\kappa(\C)$ denotes the full subcategory of $\textnormal{Shv}_\kappa(\C)$ spanned by the representable functors. Furthermore, the composition $\textnormal{Fun}(\textnormal{B}G,\C)\xrightarrow{\Phi}\textnormal{LMon}_{h^\C_{G_\C}}(\textnormal{Shv}^{\textnormal{rep}}_\kappa(\C))\rightarrow \textnormal{Shv}^{\textnormal{rep}}_\kappa(\C)$ may be identified as the composition $\textnormal{Fun}(\textnormal{B}G,\C)\xrightarrow{\textnormal{ev}_*}\C\xrightarrow{h_\bullet^\C}\textnormal{Shv}^{\textnormal{rep}}_\kappa(\C)$. Writing $\Theta:\textnormal{LMon}_{h^\C_{G_\C}}(\textnormal{Shv}^{\textnormal{rep}}_\kappa(\C))\xrightarrow{\sim}\textnormal{LMon}_{G_\C}(\C)$ for the equivalence induced by the map $\textnormal{Shv}^{\textnormal{rep}}_\kappa(\C)\xrightarrow{\sim}\C$ inverse to the Yoneda embedding, we may choose $\Phi$ to be the arrow $\Theta\circ \Psi:\textnormal{Fun}(\textnormal{B}G,\C)\rightarrow \textnormal{LMon}_{G_\C}(\C)$.

\end{proof}

\begin{remark}{}\label{Generalization For Groups Internal to N Topoi}

    In the case that $\C=\cX$ is an $(\infty,1)$-topos, the contents of \ref{Internal Vs External Group Objects Theorem} can be generalized. In particular, let $H$ be a small connected space, and let $L:\textnormal{Fun}(\cX^{\textnormal{opp}},\mS)\rightarrow \cX$ denote the (left exact) left adjoint to the Yoneda embedding. Then, by a similar argument to the one presented in the proof of \ref{Internal Vs External Group Objects Theorem}, there exists an equivalence $\textnormal{Fun}(H,\cX)\xrightarrow{\sim}\textnormal{LMon}_{L(\underline{\Omega H}_{\cX^\textnormal{opp}})}(\cX)$ over $\cX$. In the case that $\C=\cY$ is an $(n,1)$-topos for $n\in \ZZ_{\geq 1}$, we instead have a canonical equivalence $\textnormal{Fun}(H,\cY)\xrightarrow{\sim}\textnormal{LMon}_{L^\prime(\underline{\Omega (\tau_{\leq n}H)})}(\cY)$ over $\cY$, where $L^\prime:\textnormal{Fun}(\cY^{\textnormal{opp}},\tau_{\leq n-1}\mS)\rightarrow \cY$ is the (left exact) left adjoint to the inclusion (\ref{Sheafification For n Topoi}).
    
\end{remark}

\subsection{\v{C}ech (Hyper)descent}\label{Cech Descent}

The goal of this section is to prove the following theorem:

\begin{theorem}{}\label{Cech Descent Theorem}

    Fix a (possibly large) regular cardinal $\kappa$, and let $(\C,\tau)$ be a $\kappa$-geometric site. Then, for all $(\infty,1)$-categories $\cD$, a map $F:\C^{\textnormal{opp}}\rightarrow \cD$ is a $\tau$-sheaf if and only if $F$ satisfies the following two conditions:

    \begin{itemize}
    
        \item [$(1)$] $F$ preserves $\kappa$-small products. 

        \item [$(2)$] For all $\tau$-coverings $\{U_i\rightarrow X\}_{i\in I}$, the composition $\Delta_{\textnormal{inj}}^\triangleleft\xrightarrow{U}\C^{\textnormal{opp}}\xrightarrow{F}\cD$ is a limit diagram, where $U$ is the (opposite of the) \v{C}ech nerve of the induced map $\coprod_{i\in I}U_i\rightarrow X$ in $\C$. 

    \end{itemize}

    \noindent Moreover, $F$ is a $\tau$-hypersheaf precisely when $F$ satisfies (1) along with the following condition:

    \begin{itemize}
    
        \item [$(2^\prime)$] For all $\tau$-hypercoverings $X_\bullet:(\Delta^{\textnormal{opp}}_{\textnormal{inj}})^\triangleright\rightarrow \C$ (as defined in \ref{Hypercoverings Definition}), the composition $\Delta_{\textnormal{inj}}^\triangleleft\xrightarrow{X_\bullet^{\textnormal{opp}}}\C^{\textnormal{opp}}\xrightarrow{F}\cD$ is a limit diagram.

    \end{itemize}

\end{theorem}

\noindent We first make precise what is meant by the terms \emph{$\tau$-hypersheaf} and \emph{$\tau$-hypercovering} in \ref{Cech Descent Theorem}.

\begin{definition}{}\label{Hypersheaf Definition}

    Let $\C$ be an $(\infty,1)$-category equipped with a Grothendieck topology $J$, and let $\cD$ be any other $(\infty,1)$-category. Choose some strongly inaccessible cardinal $\Omega^\prime$ such that $\C$ is $\Omega^\prime$-small, and $\cD$ is $\Omega^\prime$-locally small. We say that a map $F:\C^{\textnormal{opp}}\rightarrow \cD$ is a \emph{$J$-hypersheaf} if for all $D\in \cD$, the composition $\C^{\textnormal{opp}}\xrightarrow{F}\cD\xrightarrow{\textnormal{Hom}_{\cD}(D,-)}\mS_{<\Omega^\prime}$ is a hypercomplete object of the $\Omega^\prime$-$(\infty,1)$-topos $\textnormal{Shv}^J_{<\Omega^\prime}(\C)$.

\end{definition}

\begin{remark}{}\label{Hypersheaf independent of universe}

    Under the notation of \ref{Hypersheaf Definition}, the property that a map $F:\C^{\textnormal{opp}}\rightarrow \cD$ is a $J$-hypersheaf is independent of choice of $\Omega^\prime$ (see \ref{small hypersheaves are large hypersheaves}).
    
\end{remark}

\noindent In order to define a \emph{$\tau$-hypercovering}, we will need the following lemma:

\begin{lemma}{}\label{Constructing Semisimplicial Objects}

    Let $\C$ be an $(\infty,1)$-category with pullbacks, and choose some $n\in \ZZ_{\geq -1}$. Write $\Delta_{\textnormal{inj}}$ for the subcategory of $\Delta$ spanned by the injective morphisms, and write $\Delta_{\textnormal{inj},\leq n}$ for the full subcategory of $\Delta_{\textnormal{inj}}$ spanned by those $[m]$ where $m\leq n$. Let $X_{\leq n,\bullet}:(\Delta^{\textnormal{opp}}_{\textnormal{inj},\leq n})^\triangleright\rightarrow \C$ be a diagram. Then: 

    \begin{itemize}
    
        \item [$(1)$] The right Kan extension of $X_{\leq n,\bullet}$ along $(\Delta^{\textnormal{opp}}_{\textnormal{inj},\leq n})^\triangleright\hookrightarrow (\Delta^{\textnormal{opp}}_{\textnormal{inj}})^\triangleright$ exists. 

        \item [$(2)$] For all $m\in \ZZ_{\geq n}$, the right Kan extension of $X_{\leq n,\bullet}$ along $(\Delta^{\textnormal{opp}}_{\textnormal{inj},\leq n})^\triangleright\hookrightarrow (\Delta^{\textnormal{opp}}_{\textnormal{inj},\leq m})^\triangleright$ exists, and is given by the restriction of the map $(\Delta^{\textnormal{opp}}_{\textnormal{inj}})^\triangleright\rightarrow \C$ induced by (1).

        \item [$(3)$] Let $X^{R,n}_{\leq n+1,\bullet}:(\Delta^{\textnormal{opp}}_{\textnormal{inj},\leq n+1})^\triangleright\rightarrow \C$ be the map induced by $(2)$, and define $M_{n+1}(X):=X^{R,n}_{\leq n+1,n+1}\in \C$. Then, there exists a canonical equivalence of $(\infty,1)$-categories 

        $$\textnormal{Fun}((\Delta^{\textnormal{opp}}_{\textnormal{inj},\leq n+1})^\triangleright,\C)\times_{\textnormal{Fun}((\Delta^{\textnormal{opp}}_{\textnormal{inj},\leq n})^\triangleright,\C)}\{X_{\leq n,\bullet}\}\xrightarrow{\sim}\C_{/M_{n+1}(X)}.$$

        \item [$(4)$] Choose an arrow $f:U\rightarrow M_{n+1}(X)$ in $\C$, and let $U_{\leq n+1,\bullet}:(\Delta^{\textnormal{opp}}_{\textnormal{inj},\leq n+1})^\triangleright\rightarrow \C$ be the map induced by (3). Let $\mu:U_{\leq n+1,\bullet}\rightarrow X^{R,n}_{\leq n+1,\bullet}$ be the natural transformation induced by (2), which we identify as an edge of $\textnormal{Fun}((\Delta^{\textnormal{opp}}_{\textnormal{inj},\leq n+1})^\triangleright,\C)\times_{\textnormal{Fun}((\Delta^{\textnormal{opp}}_{\textnormal{inj},\leq n})^\triangleright,\C)}\{X_{\leq n,\bullet}\}$. Then, the edge $\mu_{n+1}:U_{\leq n+1,n+1}\rightarrow M_{n+1}(X)$ in $\C$ is isomorphic to the edge $f:U\rightarrow M_{n+1}(X)$. 

    \end{itemize}

\end{lemma}

\begin{definition}{}\label{Hypercoverings Definition}

    Let $\C$ be an $(\infty,1)$-category which admits pullbacks, and choose some augmented semisimplicial object $X_\bullet:(\Delta^{\textnormal{opp}}_{\textnormal{inj}})^\triangleright\rightarrow \C$. For each $n\in \ZZ_{\geq 0}$, write $\textnormal{cosk}_{n-1}X_\bullet:(\Delta^{\textnormal{opp}}_{\textnormal{inj}})^\triangleright\rightarrow \C$ for the right Kan extension of $X_\bullet|_{(\Delta^{\textnormal{opp}}_{\textnormal{inj},\leq n-1})^\triangleright}$ along $(\Delta^{\textnormal{opp}}_{\textnormal{inj},\leq n-1})^\triangleright\hookrightarrow (\Delta^{\textnormal{opp}}_{\textnormal{inj}})^\triangleright$, and consider the canonical map $\mu:X_\bullet\rightarrow \textnormal{cosk}_{n-1}X_\bullet$. We will refer to the edge $\mu_{n}:X_{n}\rightarrow M_n(X):=\textnormal{cosk}_{n-1}X_{n}$ of $\C$ as the \emph{$n$th covering map} of $X_\bullet$. 

    Next, choose a regular cardinal $\kappa$, and suppose that $(\C,\tau)$ is a $\kappa$-geometric site. We will call $X_\bullet:(\Delta^{\textnormal{opp}}_{\textnormal{inj}})^\triangleright\rightarrow \C$ a \emph{$\tau$-hypercovering}, or a \emph{hypercovering for the $\tau$-topology on $\C$} if for all $n\in \ZZ_{\geq 0}$, the $n$th covering map $\mu_{n}:X_{n}\rightarrow M_n(X_\bullet)$ is isomorphic to an arrow of the form $\coprod_{i\in I}U_i\rightarrow M_n(X)$, where $\{U_i\rightarrow M_n(X)\}_{i\in I}$ is some ($\kappa$-small) $\tau$-covering.

\end{definition}

\begin{example}{}

    Let $\C$ be an $(\infty,1)$-category which admits pullbacks, and let $X_\bullet:(\Delta^{\textnormal{opp}}_{\textnormal{inj}})^\triangleright\rightarrow \C$ be an augmented simplicial object of $\C$. Then, the $0$th covering map $X_0\rightarrow M_0(X)$ may be identified as the arrow $X_0\rightarrow X_{-1}$.
    
\end{example}

\begin{example}{}\label{Cech nerves Hypercovering Example}

    Fix a regular cardinal $\kappa$, and let $(\C,\tau)$ be a $\kappa$-geometric site. Let $f:X\rightarrow Y$ be a map in $\C$. Then, the \v{C}ech nerve $(\Delta^{\textnormal{opp}}_{\textnormal{inj}})^\triangleright\rightarrow \C$ of $f$ is a $\tau$-hypercovering if and only if $f:X\rightarrow Y$ is isomorphic to an arrow of the form $\coprod_{i\in I}U_i\rightarrow Y$, where $\{U_i\rightarrow Y\}_{i\in I}$ is some ($\kappa$-small) $\tau$-covering.

\end{example}

\begin{example}{}\label{Comparison of Hypercoverings to SAG}

    Let $\cX$ be an $(\infty,1)$-topos, and let $U_\bullet:(\Delta^{\textnormal{opp}}_{\textnormal{inj}})^\triangleright\rightarrow \cX$ be a diagram, which we identify as a map $U^\prime_\bullet:\Delta^{\textnormal{opp}}_{\textnormal{inj}}\rightarrow \cX_{/U_{-1}}$. Then, $U_\bullet$ is a hypercovering for the effective epimorphism topology on $\cX$ precisely when $U^\prime_\bullet$ is a hypercovering of $\cX_{/U_{-1}}$, in the sense of \citeq{lurie2018sag} A.5.2.3.

\end{example}

\begin{proof} (\ref{Constructing Semisimplicial Objects}.) We first prove $(1)$. Identify $X_{\leq n,\bullet}:(\Delta^{\textnormal{opp}}_{\textnormal{inj},\leq n})^\triangleright\rightarrow \C$ with a map $X_{\leq n,\bullet}^\prime:\Delta^{\textnormal{opp}}_{\textnormal{inj},\leq n}\rightarrow \C_{/X_{-1}}$, where $X_{-1}$ denotes the value of $X_{\bullet}$ at the cone point of $(\Delta^{\textnormal{opp}}_{\textnormal{inj},\leq n})^\triangleright$. Since $\C$ has pullbacks, $\C_{/X_{\leq n,-1}}$ has finite limits. In particular, the map $X_{\leq n,\bullet}^\prime:\Delta^{\textnormal{opp}}_{\textnormal{inj},\leq n}\rightarrow \C_{/X_{\leq n,-1}}$ admits a right Kan extension to an arrow $X_{\bullet}^{R,n,\prime}:\Delta^{\textnormal{opp}}_{\textnormal{inj}}\rightarrow \C_{/X_{\leq n,-1}}$ (see the proof of \citeq{lurie2024kerodon} 05H7), which we identify as a map $X_{\bullet}^{R,n,\prime}:(\Delta^{\textnormal{opp}}_{\textnormal{inj}})^\triangleright\rightarrow \C$. Invoking \citeq{lurie2024kerodon} 04K2, we have that $X_{\bullet}^{R,n,\prime}$ is right Kan extended from $(\Delta^{\textnormal{opp}}_{\textnormal{inj},\leq n})^\triangleright$, proving $(1)$. Claim $(2)$ now follows from \citeq{lurie2024kerodon} 0312.

    We now prove $(3)$. To begin, observe that the factorization system $(L,R)$ endows $(\Delta^{\textnormal{opp}}_{\textnormal{inj}})^\triangleright$ with the structure of a Reedy category (in the sense of \citeq{highertopostheory} A.2.9.1), where $L$ is the collection of all morphisms in $(\Delta^{\textnormal{opp}}_{\textnormal{inj}})^\triangleright$, and $R$ is the collection of isomorphisms in $(\Delta^{\textnormal{opp}}_{\textnormal{inj}})^\triangleright$. Applying \citeq{highertopostheory} A.2.9.14, we have that for all $(\infty,1)$-categories $\cD$, the canonical map of $(\infty,1)$-categories

    $$\textnormal{Fun}((\Delta^{\textnormal{opp}}_{\textnormal{inj},\leq n+1})^\triangleright,\cD)\rightarrow \textnormal{Fun}((((\Delta^{\textnormal{opp}}_{\textnormal{inj},\leq n})^\triangleright)_{[n+1]/})^\triangleleft,\cD)\times_{\textnormal{Fun}(((\Delta^{\textnormal{opp}}_{\textnormal{inj},\leq n})^\triangleright)_{[n+1]/},\cD)}\textnormal{Fun}(\Delta^{\textnormal{opp}}_{\textnormal{inj},\leq n},\cD)$$

    \noindent is a categorical equivalence. In particular, for any diagram $F:\Delta^{\textnormal{opp}}_{\textnormal{inj},\leq n}\rightarrow \cD$, the arrow 

    $$\{F\}\times_{\textnormal{Fun}((\Delta^{\textnormal{opp}}_{\textnormal{inj},\leq n})^\triangleright,\cD)}\textnormal{Fun}((\Delta^{\textnormal{opp}}_{\textnormal{inj},\leq n+1})^\triangleright,\cD)\rightarrow \textnormal{Fun}((((\Delta^{\textnormal{opp}}_{\textnormal{inj},\leq n})^\triangleright)_{[n+1]/})^\triangleleft,\cD)\times_{\textnormal{Fun}(((\Delta^{\textnormal{opp}}_{\textnormal{inj},\leq n})^\triangleright)_{[n+1]/},\cD)}\{F^\prime\}$$

    \noindent is an equivalence, where $F^\prime$ denotes the composition $((\Delta^{\textnormal{opp}}_{\textnormal{inj},\leq n})^\triangleright)_{[n+1]/}\rightarrow (\Delta^{\textnormal{opp}}_{\textnormal{inj},\leq n})^\triangleright\xrightarrow{F}\cD$. \citeq{lurie2024kerodon} 04J4 supplies an equivalence $\textnormal{Fun}((((\Delta^{\textnormal{opp}}_{\textnormal{inj},\leq n})^\triangleright)_{[n+1]/})^\triangleleft,\cD)\times_{\textnormal{Fun}(((\Delta^{\textnormal{opp}}_{\textnormal{inj},\leq n})^\triangleright)_{[n+1]/},\cD)}\{F^\prime\}\xrightarrow{\sim}\cD_{/F^\prime}$, yielding an equivalence $\{F\}\times_{\textnormal{Fun}((\Delta^{\textnormal{opp}}_{\textnormal{inj},\leq n})^\triangleright,\cD)}\textnormal{Fun}((\Delta^{\textnormal{opp}}_{\textnormal{inj},\leq n+1})^\triangleright,\cD)\xrightarrow{\sim}\cD_{/F^\prime}$. Suppose now that $F$ admits a right Kan extension along $(\Delta^{\textnormal{opp}}_{\textnormal{inj},\leq n})^\triangleright\hookrightarrow (\Delta^{\textnormal{opp}}_{\textnormal{inj},\leq n+1})^\triangleright$. Then, $\cD_{/F^\prime}\rightarrow \cD$ is representable, so $(3)$ immediately follows. Claim $(4)$ follows from the proof of $(3)$.

\end{proof}

\noindent We will prove \ref{Cech Descent Theorem} in pieces. To begin, we consider when an arrow $F:\C^{\textnormal{opp}}\rightarrow \mS$ is a $\tau$-sheaf:

\begin{proposition}{}[\v{C}ech Descent]\label{Sheaf condition for topos-like categories}

    Fix a (possibly large) regular cardinal $\kappa$, and let $\C$ be a $\kappa$-geometric $(\infty,1)$-category. Let $\tau$ be a Grothendieck pretopology on $\C$ such that the pair $(\C,\tau)$ is a $\kappa$-geometric site. Then, for all $(\infty,1)$-categories $\cD$, a map $F:\C^{\textnormal{opp}}\rightarrow \cD$ is a $\tau$-sheaf if and only if $F$ satisfies the following two conditions:

    \begin{itemize}
    
        \item [$(1)$] $F$ preserves $\kappa$-small products. 

        \item [$(2)$] For all $\tau$-coverings $\{U_i\rightarrow X\}_{i\in I}$, the composition $\Delta^\triangleleft\xrightarrow{U}\C^{\textnormal{opp}}\xrightarrow{F}\cD$ is a limit diagram, where $U$ is the (opposite of the) \v{C}ech nerve of the induced map $\coprod_{i\in I}U_i\rightarrow X$ in $\C$. 
    
    \end{itemize}

\end{proposition}

\begin{proof}

     Just as in the proof of \ref{Sheaves for the kappa Product Topology}, we may reduce to the case that $\C$ is small, and we may suppose that $\cD=\mS$. We begin by proving that all $\tau$-sheaves $F:\C^{\textnormal{opp}}\rightarrow \mS$ satisfy $(1)$ and $(2)$. Since $\tau$ is a finer pretopology than the $\kappa$-(pre)topology, condition $(1)$ is automatically satisfied by $F$. We now show $(2)$. Let $\{f_i:U_i\rightarrow X\}_{i\in I}$ be a $\tau$-covering, and let $A:(\Delta^{\textnormal{opp}})^\triangleright\rightarrow \C$ be the \v{C}ech nerve of $\coprod_{i\in I}U_i\rightarrow X$, and let $B:(\Delta^{\textnormal{opp}})^\triangleright\rightarrow \textnormal{Fun}(\C^{\textnormal{opp}},\mS)$ be the \v{C}ech nerve of $\coprod_{i\in I}h^\C_{U_i}\rightarrow h^\C_X$. Writing $L_\tau:\textnormal{Fun}(\C^{\textnormal{opp}},\mS)\rightarrow \textnormal{Shv}_\tau(\C)$ for the $\tau$-sheafification functor, it follows from $(1)$ that the compositions $(\Delta^{\textnormal{opp}})^\triangleright\xrightarrow{A} \C\rightarrow \textnormal{Fun}(\C^{\textnormal{opp}},\mS)\xrightarrow{L_\tau}\textnormal{Shv}_\tau(\C)$ and $(\Delta^{\textnormal{opp}})^\triangleright\xrightarrow{B} \textnormal{Fun}(\C^{\textnormal{opp}},\mS)\xrightarrow{L_\tau}\textnormal{Shv}_\tau(\C)$ agree. Since \ref{Prelim. Cech descent} implies that this latter composition is a colimit diagram, we deduce that the composition $(\Delta^{\textnormal{opp}})^\triangleright\xrightarrow{A} \C\rightarrow \textnormal{Fun}(\C^{\textnormal{opp}},\mS)\xrightarrow{L_\tau}\textnormal{Shv}_\tau(\C)$ is also a colimit diagram, proving $(2)$.

     Next, suppose that $F:\C^{\textnormal{opp}}\rightarrow \mS$ is a map satisfying $(1)$ and $(2)$. Applying \ref{Prelim. Cech descent}, it suffices to show that, for all $\tau$-coverings $\{f_i:U_i\rightarrow X\}_{i\in I}$, the composition $\Delta^\triangleleft \xrightarrow{B^{\textnormal{opp}}}\textnormal{Fun}(\C^{\textnormal{opp}},\mS)^{\textnormal{opp}}\xrightarrow{\textnormal{Hom}_{\textnormal{Fun}(\C^{\textnormal{opp}},\mS)}(-,F)}\mS$ is a limit diagram, where $B$ denotes the \v{C}ech nerve of  $\coprod_{i\in I}h^\C_{U_i}\rightarrow h^\C_X$. Applying $(2)$, we deduce that the composition $\Delta^\triangleleft \xrightarrow{A^{\textnormal{opp}}}\C^{\textnormal{opp}}\rightarrow \textnormal{Fun}(\C^{\textnormal{opp}},\mS)^{\textnormal{opp}}\xrightarrow{\textnormal{Hom}_{\textnormal{Fun}(\C^{\textnormal{opp}},\mS)}(-,F)}\mS$ is a limit diagram, where $A$ denotes the \v{C}ech nerve of $\coprod_{i\in I}U_i\rightarrow X$. Thus, it suffices to show that the compositions $\Delta^\triangleleft \xrightarrow{A^{\textnormal{opp}}}\C^{\textnormal{opp}}\rightarrow \textnormal{Fun}(\C^{\textnormal{opp}},\mS)^{\textnormal{opp}}\xrightarrow{\textnormal{Hom}_{\textnormal{Fun}(\C^{\textnormal{opp}},\mS)}(-,F)}\mS$ and $\Delta^\triangleleft \xrightarrow{B^{\textnormal{opp}}}\textnormal{Fun}(\C^{\textnormal{opp}},\mS)^{\textnormal{opp}}\xrightarrow{\textnormal{Hom}_{\textnormal{Fun}(\C^{\textnormal{opp}},\mS)}(-,F)}\mS$ coincide. Writing $L_\kappa:\textnormal{Fun}(\C^{\textnormal{opp}},\mS)\rightarrow \textnormal{Shv}_{\kappa}(\C)$ for the sheafification functor, we observe that $(1)$ implies that $F\in \textnormal{Shv}_{\kappa}(\C)$. Thus, $\textnormal{Hom}_{\textnormal{Fun}(\C^{\textnormal{opp}},\mS)}(-,F):\textnormal{Fun}(\C^{\textnormal{opp}},\mS)^{\textnormal{opp}}\rightarrow \mS$ factors as $\textnormal{Fun}(\C^{\textnormal{opp}},\mS)^{\textnormal{opp}}\xrightarrow{L_\kappa^{\textnormal{opp}}}\textnormal{Shv}_\kappa(\C)^{\textnormal{opp}}\xrightarrow{\textnormal{Hom}_{\textnormal{Shv}_\kappa(\C)}(-,F)}\mS$. Moreover, since $L_\kappa$ is left exact, the compositions 

     $$\Delta^\triangleleft \xrightarrow{A^{\textnormal{opp}}}\C^{\textnormal{opp}}\rightarrow \textnormal{Fun}(\C^{\textnormal{opp}},\mS)^{\textnormal{opp}}\xrightarrow{L_\kappa^{\textnormal{opp}}}\textnormal{Shv}_\kappa(\C)^{\textnormal{opp}}\xrightarrow{\textnormal{Hom}_{\textnormal{Shv}_\kappa(\C)}(-,F)}\mS$$

     \noindent and 

     $$\Delta^\triangleleft \xrightarrow{B^{\textnormal{opp}}}\textnormal{Fun}(\C^{\textnormal{opp}},\mS)^{\textnormal{opp}}\xrightarrow{L_\kappa^{\textnormal{opp}}}\textnormal{Shv}_\kappa(\C)^{\textnormal{opp}}\xrightarrow{\textnormal{Hom}_{\textnormal{Shv}_\kappa(\C)}(-,F)}\mS$$

     \noindent coincide, completing the proof of our claim.

\end{proof}

\begin{remark}{}

    Combining the left cofinality of $\Delta_{\textnormal{inj}}\hookrightarrow \Delta$ with the fact that all \v{C}ech nerves $(\Delta^{\textnormal{opp}}_{\textnormal{inj}})^\triangleright\rightarrow \C$ may be promoted to a \v{C}ech nerve $(\Delta^{\textnormal{opp}})^\triangleright\rightarrow \C$ (up to isomorphism), \ref{Sheaf condition for topos-like categories} is equivalent to the sheaf condition stated in \ref{Cech Descent Theorem}. 
    
\end{remark}

\begin{example}{}\label{preservation of trunactions 1}

    Let $\cX$ be an $(n,1)$-topos for $n\in \ZZ_{\geq 1}\cup \{\infty\}$. It follows from \ref{Sheaf condition for topos-like categories} that for any $(\infty,1)$-category $\cD$, a map $F:\cX^{\textnormal{opp}}\rightarrow \cD$ is a sheaf for the canonical topology on $\cX$ if and only if the following two conditions hold:

    \begin{itemize}
        \item [$(1)$] $F$ preserves small products.
        
        \item [$(2)$] Let $G:(\Delta_{\textnormal{inj}}^{\textnormal{opp}})^\triangleright\rightarrow \cX$ be the \v{C}ech nerve of an effective epimorphism. Then, the composition $\Delta_{\textnormal{inj}}^\triangleleft\xrightarrow{G^{\textnormal{opp}}} \cX^{\textnormal{opp}}\xrightarrow{F}\cD$ is a limit diagram in $\cD$. 
        
    \end{itemize}

\end{example}

\begin{corollary}{}\label{Subcanonical Topologies for topos-like categories}

    Under the assumptions of \ref{Sheaf condition for topos-like categories}, suppose also that $\C$ is a locally small $(\infty,1)$-category. Then, the following are equivalent:

    \begin{itemize}
        \item [$(1)$] For each $\tau$-covering $\{U_i\rightarrow X\}_{i\in I}$, the induced map $\coprod_{i\in I}U_i\rightarrow X$ is an effective epimorphism in $\C$. 
    
        \item [$(2)$] The topology on $\C$ generated by $\tau$ is subcanonical. 
    \end{itemize}

\end{corollary}

\begin{remark}{}\label{counterexample of sheaf for subcanonical topology not preserving limits}

    Let $\cX$ be an $(\infty,1)$-topos. Example \ref{preservation of trunactions 1} provides a plethora of counterexamples to the often made, yet incorrect claim that ``a map $F:\cX^{\textnormal{opp}}\rightarrow \mS$ preserves small limits if and only if $F$ is a sheaf for the effective epimorphism topology on $\cX$''. We were first made aware of this family of counterexamples via an online comment of Marc Hoyois: https://chat.stackexchange.com/transcript/message/52499804\#52499804. Indeed, let $\C$ be a small $(\infty,1)$-category equipped with a Grothendieck topology $J$ such that $(1)$ $\textnormal{Shv}_J(\C)$ is \emph{not} hypercomplete, and $(2)$ every representable functor $\C^{\textnormal{opp}}\rightarrow \mS$ is a hypercomplete $J$-sheaf. It follows from \citeq{highertopostheory} 6.5.2.16 that a map $f:X\rightarrow Y$ in $\textnormal{Shv}_J(\C)^{\textnormal{hyp}}\subset \textnormal{Shv}_J(\C)$ is an effective epimorphism in $\textnormal{Shv}_J(\C)^{\textnormal{hyp}}$ if and only if $f:X\rightarrow Y$ is an effective epimorphism in $\textnormal{Shv}_J(\C)$, implying that the inclusion $i:\textnormal{Shv}_J(\C)^{\textnormal{hyp}}\hookrightarrow \textnormal{Shv}_J(\C)$ preserves effective epimorphisms. Since $i$ preserves pullbacks and small coproducts (\citeq{lurie2018sag} D.6.7.3), we deduce that for all $J$-sheaves $H:\C^{\textnormal{opp}}\rightarrow \mS$ the composition $\textnormal{Shv}_J(\C)^{\textnormal{hyp},\textnormal{opp}}\xrightarrow{i^{\textnormal{opp}}} \textnormal{Shv}_J(\C)^{\textnormal{opp}}\xrightarrow{h_H}\mS$ is a sheaf for the effective epimorphism topology on $\textnormal{Shv}_J(\C)$ (\ref{preservation of trunactions 1}). However, since every representable functor $\C^{\textnormal{opp}}\rightarrow \mS$ is a $J$-hypersheaf, it follows (from a very similar logic as that used in the proof of \ref{Kan extensions of Sheaves}) that said composition $\textnormal{Shv}_J(\C)^{\textnormal{hyp},\textnormal{opp}}\xrightarrow{i^{\textnormal{opp}}}\textnormal{Shv}_J(\C)^{\textnormal{opp}}\xrightarrow{h_H}\mS$ preserves small limits if and only if $H$ is a $J$-hypersheaf. Choosing some $H\in \textnormal{Shv}_J(\C)$ which is \emph{not} hypercomplete now provides our desired counterexample.

    Applying \ref{The Canonical Topology on a Topos}, we further observe that if there exists a Grothendieck topology $J$ on $\cX$ such that $J$-sheaves $F:\cX^{\textnormal{opp}}\rightarrow \mS$ are precisely the representable maps, then $J$ is the effective epimorphism topology on $\cX$. Combining this fact with the above family of counterexamples, we recover \citeq{lurie2018sag} 1.3.1.5: there exists $(\infty,1)$-topoi $\cY$ such that there \emph{does not} exist a Grothendieck topology $J^\prime$ on $\cY$ satisfying the property that $J^\prime$-sheaves $G:\cX^{\textnormal{opp}}\rightarrow \mS$ are precisely the representable functors. 
    
\end{remark}

We now prove the remainder of \ref{Cech Descent Theorem}. The main arguments used are either from the proof of \citeq{lurie2018sag} A.5.7.2, or as suggested by Thorger Gei\sss:

\noindent

\begin{corollary}{}[\v{C}ech Hyperdescent]\label{Chech Hyperdescent} Fix a (possibly large) regular cardinal $\kappa$, and let $\C$ be a $\kappa$-geometric $(\infty,1)$-category. Let $\tau$ be a Grothendieck pretopology on $\C$ such that the pair $(\C,\tau)$ is a $\kappa$-geometric site. Then, for all $(\infty,1)$-categories $\cD$, a map $F:\C^{\textnormal{opp}}\rightarrow \cD$ is a $\tau$-hypersheaf (in the sense of \ref{Hypersheaf Definition}) precisely when the following two conditions hold:

    \begin{itemize}

        \item [$(1)$] $F$ preserves $\kappa$-small products. 

        \item [$(2)$] For all $\tau$-hypercoverings $X_\bullet:(\Delta^{\textnormal{opp}}_{\textnormal{inj}})^\triangleright\rightarrow \C$ (as defined in \ref{Hypercoverings Definition}), the composition $\Delta_{\textnormal{inj}}^\triangleleft\xrightarrow{X_\bullet^{\textnormal{opp}}}\C^{\textnormal{opp}}\xrightarrow{F}\cD$ is a limit diagram.

    \end{itemize}
    
\end{corollary}

\begin{proof}

Just as in the proofs of \ref{Sheaves for the kappa Product Topology} and \ref{Sheaf condition for topos-like categories}, it suffices to prove our claim in the case that $\C$ is small and $\cD=\mS$ (see also \ref{Hypersheaf independent of universe}). To begin, suppose that $F:\C^{\textnormal{opp}}\rightarrow \mS$ is a $\tau$-hypersheaf, noting that \ref{Sheaf condition for topos-like categories} implies that that $F$ satisfies condition $(1)$. Next, let $X_\bullet:(\Delta^{\textnormal{opp}}_{\textnormal{inj}})^\triangleright\rightarrow \C$ be a $\tau$-hypercovering, which we identify as an arrow $X_\bullet^\prime:\Delta^{\textnormal{opp}}_{\textnormal{inj}}\rightarrow \C_{/X_{-1}}$. Writing $L:\textnormal{Fun}(\C^{\textnormal{opp}},\mS)\rightarrow \textnormal{Shv}_\tau(\C)$ for the sheafification functor, it follows from the fact that $L$ is left exact that the induced map $\Delta^{\textnormal{opp}}_{\textnormal{inj}}\rightarrow \textnormal{Shv}_\tau(\C)_{/L(h^\C_{X_{-1}})}$ is a hypercovering, in the sense of \citeq{lurie2018sag} A.5.2.3 (see \ref{Comparison of Hypercoverings to SAG}). Since the colimit of $\Delta^{\textnormal{opp}}_{\textnormal{inj}}\rightarrow \textnormal{Shv}_\tau(\C)_{/L(h^\C_{X_{-1}})}$ is $\infty$-connective (see \citeq{lurie2018sag} A.5.3.3), it follows immediately that for all $\tau$-hypersheaves $F$, the composition 

   $$\Delta^\triangleleft_{\textnormal{inj}}\xrightarrow{X_\bullet^{\textnormal{opp}}}\C\xrightarrow{(h^\C_\bullet)^{\textnormal{opp}}}\textnormal{Fun}(\C^{\textnormal{opp}},\mS)^{\textnormal{opp}}\xrightarrow{L^{\textnormal{opp}}}\textnormal{Shv}_\tau(\C)^{\textnormal{opp}}\xrightarrow{h_F}\mS$$

   \noindent is a limit diagram. Since this agrees (up to natural isomorphism) with the composition $\Delta_{\textnormal{inj}}^\triangleleft\xrightarrow{X_\bullet^{\textnormal{opp}}}\C^{\textnormal{opp}}\xrightarrow{F}\mS$, this proves (2).

   We now prove the converse. To begin, suppose that $F:\C^{\textnormal{opp}}\rightarrow \mS$ satisfies conditions $(1)$ and $(2)$. Combining \ref{Cech nerves Hypercovering Example} with \ref{Sheaf condition for topos-like categories}, we deduce that $F$ is a $\tau$-sheaf. Since we have shown that all $\tau$-hypersheaves satisfy condition (2), in order to show that $F$ is hypercomplete it suffices to prove the following: 

   \begin{itemize}
   
       \item [$(*)$] Let $F$ and $G$ be $\tau$-sheaves on $\C$ which satisfy condition (2). Then, any $\infty$-connective map $\eta:F\rightarrow G$ in $\textnormal{Shv}_\tau(\C)$ is an isomorphism. 
       
   \end{itemize}

   In order to prove $(*)$, we will inductively verify that for all such $\eta:F\rightarrow F$ and $X\in \C$, the map $\eta_X:F(X)\rightarrow G(X)$ is $n$-connective in $\mS$ for all $n\in \ZZ_{\geq 0}$. Temporarily delaying the proof of the base case $n=0$, let $n\in \ZZ_{\geq 1}$ and suppose that for each $X\in \C$ we have that $\eta_X:F(X)\rightarrow G(X)$ is $n$-connective. In order to show that $\eta_X:F(X)\rightarrow G(X)$ is $(n+1)$-connective, it suffices to show that each $F(X)\rightarrow F(X)\times_{G(X)}F(X)$ is $n$-connective (\citeq{lurie2024kerodon} 04HA (b)). However, noting that the collection of $\tau$-sheaves in $\textnormal{Shv}_\tau(\C)$ satisfying $(2)$ is closed under all existing limits in $\textnormal{Shv}_\tau(\C)$, the fact that $F(X)\rightarrow F(X)\times_{G(X)}F(X)$ is $n$-connective follows from applying our inductive assumption to the arrow $F\rightarrow F\times_GF$ (which is also an $\infty$-connective map: see \citeq{highertopostheory} 6.5.1.18). So, in order to complete the proof of our claim, it suffices to prove the following:

   \begin{itemize}

       \item [$(*^\prime)$] Let $\eta:F\rightarrow G$ be an $\infty$-connective map of $\tau$-sheaves satisfying condition $(2)$. Then, for all $X\in \C$, the map $\eta_X:F(X)\rightarrow G(X)$ is a $0$-connective map of spaces.

   \end{itemize}

   \noindent For the remainder of the proof, we fix such a map $\eta:F\rightarrow G$. Let $X\in \C$, and let $Z_\bullet:(\Delta^{\textnormal{opp}}_{\textnormal{inj}})^\triangleright\rightarrow \C$ be a $\tau$-hypercovering of $X\in \C$. Since $F$ and $G$ satisfy $(2)$, we have canonical equivalences $F(X)\cong \textnormal{Hom}(h^\C_\bullet\circ Z^s_\bullet,\underline{F})$ and $G(X)\cong \textnormal{Hom}(h^\C_\bullet\circ Z^s_\bullet,\underline{G})$ (where these are mapping spaces in $\textnormal{Fun}(\Delta_{\textnormal{inj}}^{\textnormal{opp}},\textnormal{Fun}(\C^{\textnormal{opp}},\mS))$), where $Z^s_\bullet$ denotes the underlying semisimplicial object of $Z_\bullet$. Under these identifications, the map $\eta_X:F(X)\rightarrow G(X)$ may be identified as the arrow $\textnormal{Hom}(h^\C_\bullet\circ Z^s_\bullet,\underline{F})\rightarrow \textnormal{Hom}(h^\C_\bullet\circ Z^s_\bullet,\underline{G})$ given by post-composition with $\underline{\eta}:\underline{F}\rightarrow \underline{G}$. Thus, in order to prove $(*^\prime)$, it suffices to show that for all $u:h^\C_X\rightarrow G(X)$ there exists a $\tau$-hypercovering $X_\bullet:(\Delta^{\textnormal{opp}}_{\textnormal{inj}})^\triangleright\rightarrow \C$ of $X\in \C$ and a natural transformation $\nu: h^\C_\bullet\circ X^s_\bullet\rightarrow \underline{F}$ in $\textnormal{Fun}(\Delta_{\textnormal{inj}}^{\textnormal{opp}},\textnormal{Fun}(\C^{\textnormal{opp}},\mS))$ such that the composition $h^\C_\bullet\circ X^s_\bullet\xrightarrow{\nu} \underline{F}\xrightarrow{\underline{\eta}}\underline{G}$ coincides with the composition $h^\C_\bullet\circ X^s_\bullet\rightarrow \underline{h^\C_X}\xrightarrow{\underline{u}}\underline{G}$, where $h^\C_\bullet\circ X^s_\bullet\rightarrow \underline{h^\C_X}$ is the canonical map. Fixing some $u:h^\C_X\rightarrow G$, we will now apply \ref{Constructing Semisimplicial Objects} to inductively construct a collection of maps $U_{\leq n,\bullet}:(\Delta^{\textnormal{opp}}_{\textnormal{inj},\leq n})^\triangleright\rightarrow \textnormal{Fun}(\Delta^1,\textnormal{Fun}(\C^{\textnormal{opp}},\mS))$ (as $n$ varies over $\ZZ_{\geq 0}$) satisfying the following properties:

   \begin{itemize}
   
       \item [\textnormal{(i)}] $U_{\leq n,\bullet}$ is equal (as a map of simplicial sets) to the restriction of $U_{\leq n+1,\bullet}$ to $(\Delta^{\textnormal{opp}}_{\textnormal{inj},\leq n})^\triangleright$. 

       \item [\textnormal{(ii)}] $U_{\leq n,-1}:\{-1\}\rightarrow \textnormal{Fun}(\Delta^1,\textnormal{Fun}(\C^{\textnormal{opp}},\mS))$ is the edge $u:h^\C_X\rightarrow G$. 
       
       \item [\textnormal{(iii)}] Every vertex in the image of the composition $(\Delta^{\textnormal{opp}}_{\textnormal{inj},\leq n})^\triangleright\xrightarrow{U_{\leq n,\bullet}} \textnormal{Fun}(\Delta^1,\textnormal{Fun}(\C^{\textnormal{opp}},\mS))\xrightarrow{\textnormal{ev}_0}\textnormal{Fun}(\C^{\textnormal{opp}},\mS)$ is representable, and when identified as an arrow $(\Delta^{\textnormal{opp}}_{\textnormal{inj},\leq n})^\triangleright\rightarrow \C$, is a hypercovering of $X\in \C$ (up to isomorphism).

       \item [(iv)] When identified as an arrow $\Delta^{\textnormal{opp}}_{\textnormal{inj},\leq n}\rightarrow \textnormal{Shv}_\tau(\C)_{/G}$, we have that the composition $(\Delta^{\textnormal{opp}}_{\textnormal{inj},\leq n})^\triangleright\xrightarrow{U_{\leq n,\bullet}} \textnormal{Fun}(\Delta^1,\textnormal{Fun}(\C^{\textnormal{opp}},\mS))\xrightarrow{\textnormal{ev}_1}\textnormal{Fun}(\C^{\textnormal{opp}},\mS)$ is naturally isomorphic to the constant map $\underline{\eta}:\Delta^{\textnormal{opp}}_{\textnormal{inj},\leq n}\rightarrow\textnormal{Shv}_\tau(\C)_{/G}$.

   \end{itemize}

   \noindent Supposing that this $\ZZ_{\geq 0}$-indexed family of maps $U_{\leq n,\bullet}:(\Delta^{\textnormal{opp}}_{\textnormal{inj},\leq n})^\triangleright\rightarrow \textnormal{Fun}(\Delta^1,\textnormal{Fun}(\C^{\textnormal{opp}},\mS))$ has been constructed, we will define $X^\prime_\bullet:(\Delta^{\textnormal{opp}}_{\textnormal{inj}})^\triangleright\rightarrow \textnormal{Fun}(\Delta^1,\textnormal{Fun}(\C^{\textnormal{opp}},\mS))$ to be the map of simplicial sets induced by $(1)$ (noting that $(\Delta^{\textnormal{opp}}_{\textnormal{inj}})^\triangleright$ is the colimit of the $\{(\Delta^{\textnormal{opp}}_{\textnormal{inj},\leq n})^\triangleright\}$ in the ordinary category of simplicial sets), which we identify as a natural transformation $\Delta^1\times (\Delta^{\textnormal{opp}}_{\textnormal{inj}})^\triangleright\rightarrow \textnormal{Fun}(\C^{\textnormal{opp}},\mS)$. Pulling-back along the map $\Delta^1\times \Delta^1\times \Delta^{\textnormal{opp}}_{\textnormal{inj}}\rightarrow \Delta^1\times (\Delta^{\textnormal{opp}}_{\textnormal{inj}})^\triangleright$ induced by the canonical arrow $\Delta^1\times \Delta^{\textnormal{opp}}_{\textnormal{inj}}\rightarrow (\Delta^{\textnormal{opp}}_{\textnormal{inj}})^\triangleright$, which we further identify as a square $A:\Delta^1\times \Delta^1\rightarrow \textnormal{Fun}(\Delta_{\textnormal{inj}}^{\textnormal{opp}},\textnormal{Fun}(\C^{\textnormal{opp}},\mS))$. Conditions (i)-(iv) now guarantee that there exists a $\tau$-hypercovering $X_\bullet$ of $X\in \C$ such that $h_\bullet^\C\circ X_\bullet\cong X^\prime_\bullet$ and such that (up to isomorphism) $A$ is a diagram in $\textnormal{Fun}(\Delta_{\textnormal{inj}}^{\textnormal{opp}},\textnormal{Fun}(\C^{\textnormal{opp}},\mS))$ of our desired form 

   $$
            \begin{tikzpicture}[node distance=1.8cm, auto]
                \node (A) {$ h^\C_\bullet\circ X^s_\bullet $};
                % \node (A1) [right of=A] {$ $};
                \node (B) [right of=A] {$ \underline{h^\C_X} $};
                \node (C) [below of=A] {$ \underline{F} $};
                \node (D) [below of=B] {$ \underline{G} $};
                \draw[->] (A) to node {$  $} (B);
                \draw[->] (A) to node [swap] {$ $} (C);
                \draw[->] (C) to node [swap] {$ \underline{\eta} $} (D);
                \draw[->] (B) to node {$ \underline{u} $} (D);
            \end{tikzpicture}
        $$

    \noindent The remainder of our proof will now go to constructing said $U_{\leq n,\bullet}:(\Delta^{\textnormal{opp}}_{\textnormal{inj},\leq n})^\triangleright\rightarrow \textnormal{Fun}(\Delta^1,\textnormal{Fun}(\C^{\textnormal{opp}},\mS))$. 

    For $n=0$, we note that since $\eta:F\rightarrow G$ is an effective epimorphism in $\textnormal{Shv}_\tau(\C)$, it follows that the map $L(F\times_Gh^\C_X)\rightarrow L(h^\C_X)$ is an effective epimorphism in $\textnormal{Shv}_\tau(\C)$. Next, choose some small set $I$ and an $I$-indexed coproduct of representables $\coprod_{i\in I}h^\C_{U_i}$ along with an effective epimorphism $\coprod_{i\in I}h^\C_{U_i}\rightarrow F\times_Gh^\C_X$ in $\textnormal{Fun}(\C^{\textnormal{opp}},\mS)$. Let $\coprod_{i\in I}h^\C_{U_i}\xrightarrow{e} S\xrightarrow{m} h^\C_X$ be the epi-mono factorization in $\textnormal{Fun}(\C^{\textnormal{opp}},\mS)$ of $\coprod_{i\in I}h^\C_{U_i}\rightarrow h^\C_X$. Since $L(\coprod_{i\in I}h^\C_{U_i})\rightarrow L(h^\C_X)$ is an effective epimorphism in $\textnormal{Shv}_\tau(\C)$, $m$ is an $L$-equivalence, so the maps $U_i\rightarrow X$ generate a $\tau$-covering sieve $\C^0_{/X}\subseteq \C_{/X}$. In particular, we may choose a $\kappa$-small set $K$ and a $\tau$-covering $\{W_k\rightarrow X\}_{k\in K}$ such that, for each $k\in K$, we may choose $i_k\in I$ and a map $W_k\rightarrow U_{i_k}$ where the edge $W_k\rightarrow X$ factors as $W_k\rightarrow U_{i_k}\rightarrow X$. This yields a morphism $\coprod_{k\in K}h^\C_{W_k}\rightarrow \coprod_{i\in I}h^\C_{U_i}$, which in turn yields a map $\coprod_{k\in K}h^\C_{W_k}\rightarrow F\times_Gh^\C_X$. Since $\tau$ is finer than the $\kappa$-topology and $h^\C_X$ is a $\kappa$-sheaf, we have that $F\times_Gh^\C_X$ is also a $\kappa$-sheaf: thus, the map $\coprod_{k\in K}h^\C_{W_k}\rightarrow F\times_Gh^\C_X$ factors as $\coprod_{k\in K}h^\C_{W_k}\rightarrow h^\C_{\coprod_{k\in K}W_k}\rightarrow F\times_Gh^\C_X$. The existence of the induced square

   $$
            \begin{tikzpicture}[node distance=1.8cm, auto]
                \node (A) {$ h^\C_{\coprod_{k\in K}W_k} $};
                % \node (A1) [right of=A] {$ $};
                \node (B) [right of=A] {$ h^\C_X $};
                \node (C) [below of=A] {$ F $};
                \node (D) [below of=B] {$ G $};
                \draw[->] (A) to node {$  $} (B);
                \draw[->] (A) to node [swap] {$  $} (C);
                \draw[->] (C) to node [swap] {$ \underline{\eta} $} (D);
                \draw[->] (B) to node {$ u $} (D);
            \end{tikzpicture}
        $$
   
   \noindent in $\textnormal{Fun}(\C^{\textnormal{opp}},\mS)$ completes our construction of $U_{\leq 0,\bullet}$.

   Next, let $n\in \ZZ_{\geq 0}$, and suppose that we have constructed each of the relevant maps $U_{\leq m,\bullet}:(\Delta^{\textnormal{opp}}_{\textnormal{inj},\leq m})^\triangleright\rightarrow \textnormal{Fun}(\Delta^1,\textnormal{Fun}(\C^{\textnormal{opp}},\mS))$ for all $m\in \ZZ_{\geq 0,\leq n}$. Write $U^0_{\leq n,\bullet}:(\Delta^{\textnormal{opp}}_{\textnormal{inj},\leq n})^\triangleright\rightarrow \textnormal{Fun}(\C^{\textnormal{opp}},\mS)$ for the composition $(\Delta^{\textnormal{opp}}_{\textnormal{inj},\leq n})^\triangleright\xrightarrow{U_{\leq n,\bullet}}\textnormal{Fun}(\Delta^1,\textnormal{Fun}(\C^{\textnormal{opp}},\mS))\xrightarrow{\textnormal{ev}_0}\textnormal{Fun}(\C^{\textnormal{opp}},\mS)$, and similarly define $U^1_{\leq n,\bullet}:(\Delta^{\textnormal{opp}}_{\textnormal{inj},\leq n})^\triangleright\rightarrow \textnormal{Fun}(\C^{\textnormal{opp}},\mS)$. Next, right Kan extend $U_{\leq n,\bullet}$ along $(\Delta^{\textnormal{opp}}_{\textnormal{inj},\leq n})^\triangleright\hookrightarrow (\Delta^{\textnormal{opp}}_{\textnormal{inj},\leq n+1})^\triangleright$ to yield an arrow $U^{n,R}_{\leq n+1,\bullet}:(\Delta^{\textnormal{opp}}_{\textnormal{inj},\leq n+1})^\triangleright\rightarrow \textnormal{Fun}(\Delta^1,\textnormal{Fun}(\C^{\textnormal{opp}},\mS))$. Write $U^{n,R,0}_{\leq n+1,\bullet}:(\Delta^{\textnormal{opp}}_{\textnormal{inj},\leq n+1})^\triangleright\rightarrow \textnormal{Fun}(\C^{\textnormal{opp}},\mS)$ for the composition $(\Delta^{\textnormal{opp}}_{\textnormal{inj},\leq n+1})^\triangleright\xrightarrow{U^{n,R}_{\leq n+1,\bullet}}\textnormal{Fun}(\Delta^1,\textnormal{Fun}(\C^{\textnormal{opp}},\mS))\xrightarrow{\textnormal{ev}_0}\textnormal{Fun}(\C^{\textnormal{opp}},\mS)$, and similarly define $U^{n,R,1}_{\leq n+1,\bullet}:(\Delta^{\textnormal{opp}}_{\textnormal{inj},\leq n+1})^\triangleright\rightarrow \textnormal{Fun}(\C^{\textnormal{opp}},\mS)$. Choosing some $X_{\leq n,\bullet}:(\Delta^{\textnormal{opp}}_{\textnormal{inj},\leq n})^\triangleright\rightarrow \C$ such that $h^\C_\bullet\circ X_{\leq n,\bullet}\cong U^0_{\leq n,\bullet}$, we see that $h^\C_\bullet\circ X^{n,R}_{\leq n+1,\bullet}\cong U^{n,R,0}_{\leq n+1,\bullet}$ (where $X^{n,R}_{\leq n+1,\bullet}$ is the right Kan extension of $X_{\leq n,\bullet}$ along $(\Delta^{\textnormal{opp}}_{\textnormal{inj},\leq n})^\triangleright\hookrightarrow (\Delta^{\textnormal{opp}}_{\textnormal{inj},\leq n+1})^\triangleright$), so we have the identification $U^{n,R,0}_{\leq n+1,\bullet}\cong h^\C_{M_{n+1}(X)}$. Identifying $U^{n,R,1}_{\leq n+1,\bullet}$ as an arrow $U^{n,R,1,\prime}_{\leq n+1,\bullet}:\Delta^{\textnormal{opp}}_{\textnormal{inj},\leq n+1}:\Delta^{\textnormal{opp}}_{\textnormal{inj},\leq n+1}\rightarrow \textnormal{Shv}_\tau(\C)_{/G}$, we have that $U^{n,R,1,\prime}_{\leq n+1,\bullet}$ may be identified as the right Kan extension of the constant diagram $\underline{\eta}:\Delta^{\textnormal{opp}}_{\textnormal{inj},\leq n}\rightarrow \textnormal{Shv}_\tau(\C)_{/G}$ along $\Delta^{\textnormal{opp}}_{\textnormal{inj},\leq n}\hookrightarrow \Delta^{\textnormal{opp}}_{\textnormal{inj},\leq n+1}$. Noting that $\eta$ is $\infty$-connective, it follows from \citeq{highertopostheory} 6.5.3.5 that the canonical map $\gamma^\prime:\underline{\eta}\rightarrow U^{n,R,1,\prime}_{\leq n+1,\bullet}$ of diagrams $\Delta^{\textnormal{opp}}_{\textnormal{inj},\leq n+1}\rightarrow \textnormal{Shv}_\tau(\C)_{/G}$ exhibiting $U^{n,R,1,\prime}_{\leq n+1,\bullet}$ as the $n$-coskeleton of $\underline{\eta}$ satsifies the property that the edge $\gamma^\prime_{n+1}:\eta\rightarrow U^{n,R,1,\prime}_{\leq n+1,n+1}$ is an effective epimorphism in $\textnormal{Shv}_\tau(\C)_{/G}$. 
   
   Applying \ref{Hypercoverings Definition}, we deduce that extending $U_{\leq n,\bullet}:(\Delta^{\textnormal{opp}}_{\textnormal{inj},\leq n})^\triangleright\rightarrow \textnormal{Fun}(\Delta^1,\textnormal{Fun}(\C^{\textnormal{opp}},\mS))$ to a map $(\Delta^{\textnormal{opp}}_{\textnormal{inj},\leq n+1})^\triangleright\rightarrow \textnormal{Fun}(\Delta^1,\textnormal{Fun}(\C^{\textnormal{opp}},\mS))$ is equivalent to giving a commutative square 

   $$
            \begin{tikzpicture}[node distance=1.5cm, auto]
                \node (A) {$ A $};
                % \node (A1) [right of=A] {$ $};
                \node (B) [right of=A] {$ h^\C_{M_{n+1}(X)} $};
                \node (C) [below of=A] {$ B $};
                \node (D) [below of=B] {$ F^\prime_{n+1} $};
                \draw[->] (A) to node {$  $} (B);
                \draw[->] (A) to node [swap] {$  $} (C);
                \draw[->] (C) to node [swap] {$  $} (D);
                \draw[->] (B) to node {$  $} (D);
            \end{tikzpicture}
        $$
        
    \noindent in $\textnormal{Fun}(\C^{\textnormal{opp}},\mS)$, where $F^\prime_{n+1}:=U^{n,R,1}_{\leq n+1,n+1}$, and the right vertical map is induced by the image of $[n]\in (\Delta^{\textnormal{opp}}_{\textnormal{inj},\leq n})^\triangleright$ under $U_{\leq n,\bullet}$. Writing $\gamma_{n+1}:F\rightarrow F^\prime_{n+1}$ for the image of $\gamma^\prime_{n+1}$ under the projection $\textnormal{Shv}_\tau(\C)_{/G}\rightarrow \textnormal{Shv}_\tau(\C)$, we have a cospan in $\textnormal{Fun}(\C^{\textnormal{opp}},\mS)$ of the form $F\xrightarrow{\gamma^\prime_{n+1}}F^\prime_{n+1}\leftarrow h^\C_{M_{n+1}(X)}$, where $\gamma^\prime_{n+1}$ is an effective epimorphism in $\textnormal{Shv}_\tau(\C)$. Arguing as in the base case, there exists some ($\kappa$-small) $\tau$-covering $\{V_j\rightarrow M_{n+1}(X)\}_{j\in J}$ such that the cospan $F\xrightarrow{\gamma^\prime_{n+1}}F^\prime_{n+1}\leftarrow h^\C_{M_{n+1}(X)}$ may be completed to a commutative square

    $$
            \begin{tikzpicture}[node distance=2.0cm, auto]
                \node (A) {$ h^\C_{\coprod_{j\in J}V_j} $};
                % \node (A1) [right of=A] {$ $};
                \node (B) [right of=A] {$ h^\C_{M_{n+1}(X)} $};
                \node (C) [below of=A] {$ F $};
                \node (D) [below of=B] {$ F^\prime_{n+1} $};
                \draw[->] (A) to node {$  $} (B);
                \draw[->] (A) to node [swap] {$  $} (C);
                \draw[->] (C) to node [swap] {$ \gamma_{n+1} $} (D);
                \draw[->] (B) to node {$  $} (D);
            \end{tikzpicture}
        $$
        
    \noindent in $\textnormal{Fun}(\C^{\textnormal{opp}},\mS)$. Choosing $U_{\leq n+1,\bullet}:(\Delta^{\textnormal{opp}}_{\textnormal{inj},\leq n+1})^\triangleright\rightarrow \textnormal{Fun}(\Delta^1,\textnormal{Fun}(\C^{\textnormal{opp}},\mS))$ to be the map induced by said square (in the sense of \ref{Constructing Semisimplicial Objects}) completes our inductive step.

\end{proof}

\section{Sheaves on Topoi}\label{Sheaves on Topoi}

\subsection{Truncated Sheaves on Topoi}\label{Truncated Objects}

\noindent The main result of this section is the following:

\begin{proposition}{}\label{n-truncated preserves small limits: Appendix}

    Let $n\in \ZZ_{\geq 1}\cup\{\infty\}$, and let $\cX$ be an $(n,1)$-topos. Let $m\in \ZZ_{\geq 1}$ be an integer such that $m \leq n$ if $n\in \ZZ_{\geq 1}$, let $\cD$ be an $(m,1)$-category, and let $F:\cX^{\textnormal{opp}}\rightarrow \cD$ be a map. Then: 

    \begin{itemize}
    
        \item [$(1)$] $F$ is a sheaf for the effective epimorphism topology on $\cX$ if and only if $F$ preserves small limits.

        \item [$(2)$] $F$ is a sheaf for the effective epimorphism topology on $\cX$ if and only if $F$ factors as $\cX^{\textnormal{opp}}\xrightarrow{(\tau_{\leq m-1})^{\textnormal{opp}}}(\tau_{\leq m-1}\cX)^{\textnormal{opp}}\xrightarrow{F_{\leq m-1}}\cD$, where $F_{\leq m-1}$ is a sheaf for the effective epimorphism topology on $\tau_{\leq m-1}\cX$. 
        
    \end{itemize}

\end{proposition}

\begin{remark}{}

    In the case that $n\in \ZZ_{\geq 2}$, claim (1) of \ref{n-truncated preserves small limits: Appendix} is mathematical folklore, and the case that $n=1$ is a classical result (\citeq{Johnstone2002-JOHSOA-7} C.2.2.7). In addition, if $\cD$ is locally small, (1) implies (2) (see \ref{localizations lemma}).

\end{remark}

\noindent Our proof of \ref{n-truncated preserves small limits: Appendix} will require several preliminaries.

\begin{definition}{}[HTT 6.4.1.5]\label{n Topoi definition}

    Let $n\in \ZZ_{\geq 1}$ be an integer. We call an $(n,1)$-category $\cX$ an \emph{$(n,1)$-topos} if any of the following equivalent definitions are met:

    \begin{itemize}
        \item [(1)] There exists a small finite complete $(n,1)$-category $\C$ equipped with a Grothendieck topology $J$, and an equivalence $\textnormal{Shv}^{J}_{n-1}(\C)\cong \cX$, where $\textnormal{Shv}^{J}_{n-1}(\C)\subset \textnormal{Shv}_J(\C)$ is the full subcategory consisting of those $J$-sheaves $F:\C^{\textnormal{opp}}\rightarrow \mS$ which take values in $\tau_{\leq n-1}\mS$.

        \item [$(1^\prime)$] There exists a small finite complete $(n,1)$-category $\C$ equipped with a subcanonical Grothendieck topology $J$, and an equivalence $\textnormal{Shv}^{J}_{n-1}(\C)\cong \cX$, where $\textnormal{Shv}^{J}_{n-1}(\C)\subset \textnormal{Shv}_J(\C)$ is the full subcategory consisting of those $J$-sheaves $F:\C^{\textnormal{opp}}\rightarrow \mS$ which take values in $\tau_{\leq n-1}\mS$.

        \item [(2)] There exists an $(\infty,1)$-topos $\cY$ and an equivalence $\cX\cong \tau_{\leq n-1}\cY$. 

        \item [(3)] $\cX$ satisfies the following:

        \begin{itemize}
            \item [(a)] $\cX$ is a presentable $(n,1)$-category
            
            \item [(b)] (Small) colimits in $\cX$ are universal 
            
            \item [(c)] Pairwise coproducts in $\cX$ are disjoint. 

            \item [(d)] For every groupoid object $U:\Delta^{\textnormal{opp}}\rightarrow \cX$ satisfying the property that the canonical map $U(d_0,d_1):U_1\rightarrow U_0\times U_0$ is $(n-2)$-truncated, $U$ is effective. 
            
        \end{itemize}
    \end{itemize}
    
\end{definition}

\noindent We first wish to relate effective epimorphisms in an $(n,1)$-topos $\cX$ for $n\in \ZZ_{\geq 1}$ to effective epimorphisms in an $(\infty,1)$-topos $\cY$ satisfying $\tau_{\leq n-1}\cY\cong \cX$ (see \ref{preservation of effective epimorphisms in truncations}). To do so, we first verify that every $(n,1)$-topos is a semitopos, as alluded to in the introduction of \citeq{highertopostheory} 6.2.3:

\begin{lemma}{}\label{n-topoi are semitopoi}
    Let $\cX$ be an $(n,1)$-topos for $n\in \ZZ_{\geq 1}\cup \{\infty\}$. Then, $\cX$ is a semitopos (in the sense of \citeq{highertopostheory} 6.2.3.1).
\end{lemma}

\begin{proof}
    The case for $n=\infty$ is given by \citeq{highertopostheory} 6.2.3.2, so we instead treat the case that $n\in \ZZ_{\geq 1}$. It suffices to verify that for any map $f:X\rightarrow Y$ in $\cX$, the underlying simplicial object of the \v{C}ech nerve of $f$ is effective. Applying property (d) of characterization (3) of \ref{n Topoi definition}, it suffices to show that for all maps $f:X\rightarrow Y$ in $\cX$, the map $X\times_YX\xrightarrow{\textnormal{pr}_1,\textnormal{pr}_2} X\times X$ is $(n-2)$-truncated. That is, we wish to show that for all $A\in \cX$ the map $\textnormal{Hom}_\cX(A,X)\times_{\textnormal{Hom}_\cX(A,Y)}\textnormal{Hom}_\cX(A,X)\rightarrow \textnormal{Hom}_\cX(A,X)\times \textnormal{Hom}_\cX(A,X)$ is an $(n-2)$-truncated map of spaces. Noting that $\textnormal{Hom}_\cX(A,Y)$ is $(n-1)$-truncated (since $\cX$ is an $(n,1)$-category), it suffices to show that whenever $g:U\rightarrow V$ is a map of spaces with $V$ $(n-1)$-truncated, the map $U\times_VU\rightarrow U\times U$ is $(n-2)$-truncated. Applying \citeq{lurie2024kerodon} 056B, we deduce that this is equivalent to asking that for all pairs $a,b\in U$, the homotopy fibre of $U\times_VU\rightarrow U\times U$ over $(a,b)$ is $(n-2)$-truncated. Our claim now follows from the fact that said homotopy fibre may be identified as the $V$-mapping space $\textnormal{Hom}_V(g(a),g(b))$.

\end{proof}

\begin{corollary}{}\label{epis in n topoi}

     Let $n\in \ZZ_{\geq 1}\cup \infty$, and let $\cX$ be an $(n,1)$-topos. Then, a map $f:X\rightarrow Y$ in $\cX$ is an effective epimorphism if and only if the map $\tau_{\leq 0}f:\tau_{\leq 0}X\rightarrow \tau_{\leq 0}Y$ is an effective epimorphism in the ordinary topos $\tau_{\leq 0}\cX$.
    
\end{corollary}

\begin{proof}

    The case for $n=\infty$ is treated in \citeq{cnossen2026httnotes} 3.12, so it suffices to prove our claim for $n\in \ZZ_{\geq 1}$. Choose an $(\infty,1)$-topos $\cY$, and an equivalence $i:\cX\xrightarrow{\sim}\tau_{\leq n-1}\cY\subset \cY$. Let $f:U\rightarrow V$ be a map in $\cX$ such that $\tau_{\leq 0}f$ is an effective epimorphism in $\cX^\heartsuit$. Since the map $i$ induces an equivalence $\cX^\heartsuit\xrightarrow{\sim}\cY^\heartsuit$, this implies that $i(f)$ is an effective epimorphism in $\cY$. Since $\cX\rightarrow \cY$ preserves \v{C}ech nerves and is fully faithful, this immediately implies that $f$ is an effective epimorphism in $\cX$. 

    To show the converse, we first recall that every $(n,1)$-topos is a semitopos, in the sense of \citeq{highertopostheory} 6.2.3.1 (\ref{n-topoi are semitopoi}). Invoking \citeq{highertopostheory} 6.2.3.10, it now suffices to show that, whenever a map $f:U\rightarrow V$ in $\cX$ is an effective epimorphism in $\cX$, the induced map $\tau_{\leq 0}f^*:\textnormal{Sub}(\tau_{\leq 0}V)\rightarrow \textnormal{Sub}(\tau_{\leq 0}U)$ is injective. Let $F_\bullet:(\Delta^{\textnormal{opp}})^\triangleright\rightarrow \cX$ be the \v{C}ech nerve of $f$, noting that the composition $(\Delta^{\textnormal{opp}})^\triangleright\rightarrow \cX\rightarrow \cX^\heartsuit$ is a colimit diagram. Combining \ref{n-topoi are semitopoi} with \citeq{highertopostheory} 6.2.3.13, we have that $\coprod_{n\in \ZZ_{\geq 0}}\tau_{\leq 0}F_n\rightarrow \tau_{\leq 0}V$ is an effective epimorphism in $\cX^\heartsuit$. In particular, this implies that the map $\textnormal{Sub}(\tau_{\leq 0}V)\rightarrow \textnormal{Sub}(\coprod_{n\in \ZZ_{\geq 0}}\tau_{\leq 0}F_n)\cong \prod_{n\in \ZZ_{\geq 0}}\textnormal{Sub}(\tau_{\leq 0}F_n)$ is injective (see \citeq{highertopostheory} 6.2.3.9). Since each of the induced maps $\textnormal{Sub}(\tau_{\leq 0}V)\rightarrow \textnormal{Sub}(\tau_{\leq 0}F_n)$ factors through $\textnormal{Sub}(\tau_{\leq 0}V)\rightarrow \textnormal{Sub}(\tau_{\leq 0}U)=\textnormal{Sub}(\tau_{\leq 0}F_0)$, it follows that $\textnormal{Sub}(\tau_{\leq 0}V)\rightarrow \textnormal{Sub}(\tau_{\leq 0}U)$ is injective, proving our claim. 

\end{proof}

\begin{corollary}{}\label{preservation of effective epimorphisms in truncations}
    Let $n\in \ZZ_{\geq 1}$, and let $m\in \ZZ_{\geq n}\cup \{\infty\}$. Let $\cX$ be an $(m,1)$-topos. Then, a map $f:X\rightarrow Y$ in $\cX$ is an effective epimorphism if and only if $\tau_{\leq n-1}f$ is an effective epimorphism in the $(n,1)$-topos $\tau_{\leq n-1}\cX$.
    
\end{corollary}

\noindent We may now use \ref{preservation of effective epimorphisms in truncations} to deduce the following component of \ref{n-truncated preserves small limits: Appendix} (2): 

\begin{corollary}{}\label{Partial. Truncated Objects of Shv Can X}

    Let $\cX$ be an $(\infty,1)$-topos, let $n\in \ZZ_{\geq 0}$, and let $F:\cX^{\textnormal{opp}}\rightarrow \tau_{\leq n}\widehat{\mS}$ be a $\C an(\cX)$-sheaf. Then, $F_{n}:=F|_{(\tau_{\leq n}\cX)^{\textnormal{opp}}}$ is a sheaf for the effective epimorphism topology on $\tau_{\leq n}\cX$, and $F$ factors as $\cX^{\textnormal{opp}}\xrightarrow{(\tau_{\leq n})^{\textnormal{opp}}} (\tau_{\leq n}\cX)^{\textnormal{opp}}\xrightarrow{F_{n}} \tau_{\leq n}\widehat{\mS}$.

\end{corollary}

\begin{proof}

    Let $F:\cX^{\textnormal{opp}}\rightarrow \tau_{\leq n}\widehat{\mS}$ be a $\C an(\cX)$-sheaf. It follows from \ref{epis in n topoi} that $\tau_{\leq n}\cX\hookrightarrow \cX$ preserves effective epimorphisms: since this is also left exact, we observe that whenever the composition $(\tau_{\leq n}\cX)^{\textnormal{opp}}\hookrightarrow \cX^{\textnormal{opp}}\xrightarrow{F} \tau_{\leq n}\widehat{\mS}$ preserves small products, said composition is a sheaf for the effective epimorphism topology on $\tau_{\leq n}\cX$ (see \ref{preservation of trunactions 1}). Again applying \ref{preservation of trunactions 1}, we see that the left exact map $\cX\rightarrow \widehat{\textnormal{Shv}}_{\C an(\cX)}(\cX)$ preserves finite limits, effective epimorphisms, and small coproducts. In particular, this is a map of $\infty$-pretopoi, in the sense of \citeq{lurie2018sag} A.6.4.1. It now follows from \citeq{lurie2018sag} A.6.7.1 that for all $X\in \cX$ and $n\in \ZZ_{\geq 1}$, the map $h^\cX_X\rightarrow h^{\cX}_{\tau_{\leq n}X}$ (in $\widehat{\textnormal{Shv}}_{\C an(\cX)}(\cX)$) exhibits $h^{\cX}_{\tau_{\leq n}X}$ as the $n$-truncation of $h^\cX_X$. In particular, the $n$-truncated $\C an(\cX)$-sheaf $F:\cX^{\textnormal{opp}}\rightarrow \tau_{\leq n}\widehat{\mS}$ factors as $\cX^{\textnormal{opp}}\rightarrow (\tau_{\leq n}\cX)^{\textnormal{opp}}\xrightarrow{F_{n}}\tau_{\leq n}\widehat{\mS}$. Since $\cX^{\textnormal{opp}}\rightarrow (\tau_{\leq n}\cX)^{\textnormal{opp}}$ is essentially surjective and preserves small products, it follows that $F_{n}$ also preserves small products, as desired.

\end{proof}

Our proof of the remainder of \ref{n-truncated preserves small limits: Appendix} will rely on establishing the following two facts:

\begin{itemize}

    \item [$(1)$] Let $n\in \ZZ_{\geq 1}$, and let $\C$ be an $(n,1)$-category equipped with a subcanonical Grothendieck topology $J$. Then, for all complete $(n,1)$-categories $\cD$, a map $F:\textnormal{Shv}^J_{n-1}(\C)^{\textnormal{opp}}\rightarrow \cD$ preserves small limits if and only if $F$ is right Kan extended from a $J$-sheaf $\C^{\textnormal{opp}}\rightarrow \cD$ (\ref{univ prperty of n-1 sheaves}).

    \item [$(2)$] Under the assumptions of $(1)$, a map $F:\textnormal{Shv}^J_{n-1}(\C)^{\textnormal{opp}}\rightarrow \cD$ is a sheaf for the effective epimorphism topology on $\textnormal{Shv}^J_{n-1}(\C)$ if and only if $F$ is right Kan extended from a $J$-sheaf $\C^{\textnormal{opp}}\rightarrow \cD$ (this may be deduced from \ref{Basis for Eff Epi n Topoi}).
    
\end{itemize}

\noindent We now prove these claims in order. 

\begin{lemma}{}\label{Kan extensions of Sheaves}

    Let $\C$ be a small $(\infty,1)$-category equipped with a subcanonical topology $J$, and let $\cD$ be a complete $(\infty,1)$-category. Let $F:\C^{\textnormal{opp}}\rightarrow \cD$ be a map, write $F_*:\textnormal{Shv}_J(\C)^{\textnormal{opp}}\rightarrow \cD$ for the right Kan extension of $F$ along $\C^{\textnormal{opp}}\rightarrow \textnormal{Shv}_J(\C)^{\textnormal{opp}}$, and write $F^\prime_*:\textnormal{Fun}(\C^{\textnormal{opp}},\mS)^{\textnormal{opp}}\rightarrow \cD$ for the right Kan extension of $F$ along $\C^{\textnormal{opp}}\rightarrow \textnormal{Fun}(\C^{\textnormal{opp}},\mS)^{\textnormal{opp}}$. Then, the following are equivalent:

    \begin{itemize}
    
        \item [$(1)$] $F:\C^{\textnormal{opp}}\rightarrow \cD$ is a $J$-sheaf, in the sense of \ref{Sheaf Definition}.

        \item [$(2)$] For all $X\in \C$ and $J$-sieves $\C^0_{/X}\subseteq \C_{/X}$, the monomorphism $m:S\rightarrow h^\C_X$ corresponding to $\C^0_{/X}\subseteq \C_{/X}$ under the equivalence of \citeq{highertopostheory} 6.2.2.5 is carried to an isomorphism by $(F_*^\prime)^{\textnormal{opp}}$.

        \item [$(3)$] $F^\prime_*$ factors as $\textnormal{Fun}(\C^{\textnormal{opp}},\mS)^{\textnormal{opp}}\xrightarrow{L^{\textnormal{opp}}}\textnormal{Shv}_J(\C)^{\textnormal{opp}}\xrightarrow{F_*}\cD$, where $L:\textnormal{Fun}(\C^{\textnormal{opp}},\mS)\rightarrow \textnormal{Shv}_J(\C)$ is left adjoint to the inclusion.

        \item [$(4)$] $F_*$ preserves small limits.

    \end{itemize}
    
\end{lemma}

\begin{proof}

Consider the map $F^\prime_*:\textnormal{Fun}(\C^{\textnormal{opp}},\mS)^{\textnormal{opp}}\rightarrow \cD$. Let $X\in \C$, let $\C^0_{/X}\subseteq \C_{/X}$ be a $J$-sieve on $X$, and consider the corresponding monomorphism $m:S\rightarrow h^\C_X$. Noting $S$ is the colimit of the composition $\C^0_{/X}\rightarrow \C\rightarrow \textnormal{Fun}(\C^{\textnormal{opp}},\mS)$, the map $m:S\rightarrow h^\C_X$ is induced by the composition $(\C^0_{/X})^\triangleright\hookrightarrow (\C_{/X})^\triangleright\rightarrow \C\rightarrow  \textnormal{Fun}(\C^{\textnormal{opp}},\mS)$. Thus, the composition $(\C^0_{/X})^\triangleright\hookrightarrow (\C_{/X})^\triangleright\rightarrow \C\rightarrow  \textnormal{Fun}(\C^{\textnormal{opp}},\mS)\xrightarrow{(F_*)^{\textnormal{opp}}}\cD^{\textnormal{opp}}$ is a colimit diagram if and only if $(F_*^\prime)^{\textnormal{opp}}$ carries $m:S\rightarrow h^\C_X$ to an isomorphism. Since $(\C^0_{/X})^\triangleright\hookrightarrow (\C_{/X})^\triangleright\rightarrow \C\rightarrow  \textnormal{Fun}(\C^{\textnormal{opp}},\mS)\xrightarrow{(F_*)^{\textnormal{opp}}}\cD^{\textnormal{opp}}$ agrees with the composition $(\C^0_{/X})^\triangleright\rightarrow (\C_{/X})^\triangleright\rightarrow \C\xrightarrow{F^{\textnormal{opp}}}\cD^{\textnormal{opp}}$, we see that (1) and (2) are equivalent. 

Next, assume that (2) holds, and let $\mu:\Delta^1\times \textnormal{Fun}(\C^{\textnormal{opp}},\mS)\rightarrow \textnormal{Fun}(\C^{\textnormal{opp}},\mS)$ be the unit map exhibiting $L$ as left adjoint to the inclusion. Since $F_*$ preserves small limits, it follows from \ref{Univ property of Presentable Localizations} that $F_*$ carries every $L^{\textnormal{opp}}$-equivalence to an isomorphism in $\cD$. Noting also that $F_*^\prime$ restricts to $F_*$ (up to isomorphism), the composition $\Delta^1\times \textnormal{Fun}(\C^{\textnormal{opp}},\mS)\rightarrow \textnormal{Fun}(\C^{\textnormal{opp}},\mS)\xrightarrow{(F^\prime_*)^{\textnormal{opp}}}\cD^{\textnormal{opp}}$ induces a natural isomorphism from the composition $\textnormal{Fun}(\C^{\textnormal{opp}},\mS)^{\textnormal{opp}}\xrightarrow{L^{\textnormal{opp}}}\textnormal{Shv}_J(\C)^{\textnormal{opp}}\xrightarrow{F_*}\cD$ to $F^\prime_*$, implying (3). Since (3) automatically implies (2), we have that (2) and (3) are equivalent. 

Finally, suppose that (3) is satisfied. Since $F_*$ preserves small limits, and $L^{\textnormal{opp}}$ is essentially surjective and preserves all existing limits in $\textnormal{Fun}(\C^{\textnormal{opp}},\mS)^{\textnormal{opp}}$, (4) immediately follows. If (4) holds, the composition $\textnormal{Fun}(\C^{\textnormal{opp}},\mS)^{\textnormal{opp}}\xrightarrow{L^{\textnormal{opp}}}\textnormal{Shv}_J(\C)^{\textnormal{opp}}\xrightarrow{F_*}\cD$ preserves small limits, and as such, is the right Kan extension of the map $F_0:\C^{\textnormal{opp}}\rightarrow \cD$ given by the composition $\C^{\textnormal{opp}}\rightarrow \textnormal{Fun}(\C^{\textnormal{opp}},\mS)^{\textnormal{opp}}\xrightarrow{L^{\textnormal{opp}}}\textnormal{Shv}_J(\C)^{\textnormal{opp}}\xrightarrow{F_*}\cD$ along $\C^{\textnormal{opp}}\rightarrow \textnormal{Fun}(\C^{\textnormal{opp}},\mS)^{\textnormal{opp}}$. Since $J$ is subcanonical we may identify $F_0$ as $F$, proving (3).

\end{proof}

\begin{lemma}{}\label{localizations lemma}

    Let $\cD$ be a locally small $(\infty,1)$-category, and let $L:\cE\rightarrow \cE^\prime$ be a localization at a (possibly large) set of arrows $W$ in $\cE$. Then, a map $F:\cE\rightarrow \cD$ factors as $\cE\xrightarrow{L}\cE^\prime\rightarrow \cD$ if and only if, for each $X\in \cD$, the composition $\cE\xrightarrow{F}\cD\xrightarrow{\textnormal{Hom}_\cD(X,-)}\mS$ factors as $\cE\xrightarrow{L}\cE^\prime\rightarrow \mS$. 

\end{lemma}

\begin{proof}

    This follows from the fact that the $\textnormal{Hom}_\cD(X,-)$ are jointly conservative.  
    
\end{proof}

\begin{corollary}{}\label{univ prperty of n-1 sheaves}

    Let $n\in \ZZ_{\geq -1}$, let $\C$ be a small $(n,1)$-category equipped with a subcanonical topology $J$, and let $\cD$ be a complete $(n,1)$-category. Then, a map $F:\textnormal{Shv}^J_{n-1}(\C)^{\textnormal{opp}}\rightarrow \cD$ preserves small limits if and only if $F:\textnormal{Shv}^J_{n-1}(\C)^{\textnormal{opp}}\rightarrow \cD$ is the right Kan extension of a $J$-sheaf $\C^{\textnormal{opp}}\rightarrow \cD$ along $\C^{\textnormal{opp}}\rightarrow \textnormal{Shv}^J_{n-1}(\C)^{\textnormal{opp}}$.
    
\end{corollary}

\begin{proof}

Changing universe if necessary, without loss of generality we may assume that $\cD$ admits a Yoneda embedding $h^\cD_\bullet:\cD\rightarrow \textnormal{Fun}(\cD^{\textnormal{opp}},\tau_{\leq n-1}\widehat{\mS})$. Combining \ref{Kan extensions of Sheaves} with \ref{localizations lemma}, it now suffices to prove our claim in the case that $\cD=\tau_{\leq n-1}\widehat{\mS}$.

To begin, let $F:\C^{\textnormal{opp}}\rightarrow \tau_{\leq n-1}\widehat{\mS}$ be a $J$-sheaf, noting that $h_F:\widehat{\textnormal{Shv}}_{J}(\C)^{\textnormal{opp}}\rightarrow \tau_{\leq n-1}\widehat{\mS}$ may be identified as the right Kan extension of $F$ along $\C^{\textnormal{opp}}\rightarrow \widehat{\textnormal{Shv}}_{J}(\C)^{\textnormal{opp}}$ (\ref{Kan extensions of Sheaves}). Noting that $h_F:\widehat{\textnormal{Shv}}_{J}(\C)^{\textnormal{opp}}\rightarrow \tau_{\leq n-1}\widehat{\mS}$ factors as $\widehat{\textnormal{Shv}}_{J}(\C)^{\textnormal{opp}}\rightarrow(\tau_{\leq n-1}\widehat{\textnormal{Shv}}_{J}(\C))^{\textnormal{opp}}\xrightarrow{h_F}\tau_{\leq n-1}\widehat{\mS}$, it follows from \ref{Small Sheaves are stable under small limits} that the composition $\textnormal{Shv}_J(\C)^{\textnormal{opp}}\hookrightarrow \widehat{\textnormal{Shv}}_J(\C)^{\textnormal{opp}}\rightarrow (\tau_{\leq n-1}\widehat{\textnormal{Shv}}_{J}(\C))^{\textnormal{opp}}\xrightarrow{h_F}\tau_{\leq n-1}\widehat{\mS}$ preserves small limits. It now follows from \ref{HyperSheafification Ind of Choice of Universe} (2) implies that the composition $\textnormal{Shv}_J(\C)^{\textnormal{opp}}\rightarrow (\tau_{\leq n-1}\textnormal{Shv}_J(\C))^{\textnormal{opp}}\hookrightarrow \tau_{\leq n-1}\widehat{\textnormal{Shv}_J(\C)}^{\textnormal{opp}}\xrightarrow{h_F}\tau_{\leq n-1}\widehat{\mS}$ preserves small limits, which in turn implies that the restriction $h_F|_{(\tau_{\leq n-1}\textnormal{Shv}_J(\C))^{\textnormal{opp}}}$ preserves small limits. This direction of our claim now follows from the fact that $h_F|_{(\tau_{\leq n-1}\textnormal{Shv}_J(\C))^{\textnormal{opp}}}$ may be identified as the right Kan extension of $F:\C^{\textnormal{opp}}\rightarrow \widehat{\mS}$ along $\C^{\textnormal{opp}}\rightarrow (\tau_{\leq n-1}\textnormal{Shv}_J(\C))^{\textnormal{opp}}$.
  
 For the converse, let $G:\C^{\textnormal{opp}}\rightarrow \tau_{\leq n-1}\widehat{\mS}$ be a map, which we right Kan extend along $\C^{\textnormal{opp}}\rightarrow (\tau_{\leq n-1}\textnormal{Shv}_J(\C))^{\textnormal{opp}}$ to an arrow $G^\prime:(\tau_{\leq n-1}\textnormal{Shv}_J(\C))^{\textnormal{opp}}\rightarrow \tau_{\leq n-1}\widehat{\mS}$. It follows that the composition $\textnormal{Fun}(\C^{\textnormal{opp}},\mS)^{\textnormal{opp}}\rightarrow (\tau_{\leq n-1}\textnormal{Shv}_J(\C))^{\textnormal{opp}}\xrightarrow{G^\prime} \tau_{\leq n-1}\widehat{\mS}$, so the fact that $G$ is a $J$-sheaf now follows from \ref{Kan extensions of Sheaves}. 
 
\end{proof}

\noindent We will not need the following, though it is an immediate consequence of \ref{univ prperty of n-1 sheaves}:

\begin{corollary}{}\label{Univ prop of n-1 presheaves}
    Let $n\in \ZZ_{\geq -1}$, and let $\cD$ be a small $(n,1)$-category. Then, $\textnormal{Fun}(\cD^{\textnormal{opp}},\tau_{\leq n-1}\mS)$ satisfies the following universal property: for all cocomplete $(n,1)$-categories $\cE$, pulling-back along $\cD\rightarrow \textnormal{Fun}(\cD^{\textnormal{opp}},\tau_{\leq n-1}\mS)$ induces an equivalence of $(n,1)$-categories $\textnormal{Fun}^{co \ell im}(\textnormal{Fun}(\cD^{\textnormal{opp}},\tau_{\leq n-1}\mS),\cE)\rightarrow \textnormal{Fun}(\cD,\cE)$. Moreover, a map $G:\textnormal{Fun}(\cD^{\textnormal{opp}},\tau_{\leq n-1}\mS)\rightarrow \cE$ preserves small colimits if and only if $G$ is left Kan extended from $\cD$.
\end{corollary}

\noindent A key part of our proof of \ref{n-truncated preserves small limits: Appendix} relies on the following theorem (\citeq{Aoki_2023} A.6):

\begin{theorem}{}[\textbf{[Aoki23]} A.6]\label{Hypersheaf Criteria}
    Let $\C$ be a small $(\infty,1)$-category equipped with a topology $J$, and let $F:\cB\rightarrow \C$ be a fully faithful $(\infty,1)$-functor. Suppose that, for each $X\in \C$, there exists a (small) set of morphisms $\{F(B_i)\rightarrow X\}_{i\in I}$ in $\C$ generating a $J$-sieve on $X$, where each $B_i$ is in $\cB$. Equip $\cB$ with the topology $J_\cB$, where for each $U\in \cB$, the collection of $J_\cB$-sieves on $U$ is defined as those sieves $\cB^0_{/U}\hookrightarrow \cB_{/U}$ on $U$ where the fully faithful map $\cB^0_{/U}\rightarrow \C_{/F(U)}$ generates a $J$-sieve on $U\in \C$. Then, a map $H:\C^{\textnormal{opp}}\rightarrow \mS$ is a $J$-hypersheaf if and only if $H$ is the right Kan extension of a $J_\cB$-hypersheaf on $\cB$ along $F$. 

\end{theorem}

\begin{definition}{}\label{Basis Definition}
    Under the notation of \ref{Hypersheaf Criteria}, we will say that the pair $(\cB,J_\cB)$ is a \emph{basis} for the pair $(\C,J)$.
\end{definition}

\begin{proposition}{}\label{Basis for Eff Epi n Topoi}

    Let $n\in \ZZ_{\geq 1}$, and let $\C$ be an $(n,1)$-category equipped with a subcanonical topology $J$. Then, $(\C,J)$ is a basis for $(\textnormal{Shv}^J_{n-1}(\C),\C an(\textnormal{Shv}^J_{n-1}(\C)))$.
    
\end{proposition}

\begin{proof}

    Since for all $F\in \textnormal{Shv}^J_{n-1}(\C)$ we may choose a small set $I$ of objects $\{X_i\}_{i\in I}$ in $\C$ such that there exists an effective epimorphism $\coprod_{i\in I}h^\C_{X_i}\rightarrow F$, it only remains to show that the topology on $\C$ induced by the canonical topology on $\textnormal{Shv}^J_{n-1}(\C)$ (in the sense of \ref{Hypersheaf Criteria}) is equal to $J$.  

    Let $X\in \C$, and let $\{f_i:U_i\rightarrow X\}_{i\in I}$ be a (small) set of arrows in $\C$. Write $\coprod_{i\in I}h^\C_{U_i}\xrightarrow{e} S\xrightarrow{m} h^\C_X$ for the epi-mono factorization in $\textnormal{Fun}(\C^{\textnormal{opp}},\tau_{\leq n-1}\mS)$ of the induced map $\coprod_{i\in I}h^\C_{U_i}\rightarrow h^\C_X$, and denote with $L:\textnormal{Fun}(\C^{\textnormal{opp}},\tau_{\leq n-1}\mS)\rightarrow \textnormal{Shv}_{n-1}^J(\C)$ the $J$-sheafification functor. Since $L$ preserves both effective epimorphisms and monomorphisms, we deduce there exists an epi-mono factorization $\coprod_{i\in I}h^\C_{U_i}\xrightarrow{e^\prime} L(S)\xrightarrow{m^\prime} h^\C_X$ of $\coprod_{i\in I}h^\C_{U_i} \rightarrow h^\C_X$ in $\textnormal{Shv}_{n-1}^J(\C)$, where $e^\prime\cong L(e)$ and $m^\prime\cong L(m)$ in the arrow category of $\textnormal{Shv}_{n-1}^J(\C)$. This implies that $\coprod_{i\in I}h^\C_{U_i} \rightarrow h^\C_X$ is an effective epimorphism in $\textnormal{Shv}_{n-1}^J(\C)$ precisely when $m:S\rightarrow h^\C_X$ is an $L$-equivalence. It now follows from \citeq{highertopostheory} 6.2.2.16 that the map $\coprod_{i\in I}h^\C_{U_i}\rightarrow  h^\C_X$ is an effective epimorphism in $\textnormal{Shv}_{n-1}^J(\C)$ if and only if the set of arrows $\{f_i:U_i\rightarrow X\}_{i\in I}$ generates a $J$-covering sieve on $X\in \C$, completing the proof of our claim.

\end{proof}

\begin{remark}{}\label{Generalization of Basis for Eff Epi n Topoi}

    The argument given in \ref{Basis for Eff Epi n Topoi} can be easily adapted to show the following:

    \begin{itemize}
        \item [\textnormal{(i)}] For all $m\in \ZZ_{\geq n}$, $(\C,J)$ is a basis for $(\textnormal{Shv}_{m-1}^J(\C),\C an(\textnormal{Shv}_{m-1}^J(\C)))$. 

        \item [\textnormal{(ii)}] For any small $(\infty,1)$-category $\cD$ equipped with a topology $K$, the pair $(\cD,K)$ is a basis for $(\textnormal{Shv}_K(\cD),\C an(\textnormal{Shv}_K(\cD)))$.

    \end{itemize}
    
\end{remark}

\noindent We may now prove \ref{n-truncated preserves small limits: Appendix} (1):

\begin{corollary}{}\label{n topoi sheaves preserve small limits}

    Let $n\in \ZZ_{\geq 1}$, and let $\cX$ be an $(n,1)$-topos. Let $\cD$ be an $(n,1)$-category. Then, a map $F:\cX^{\textnormal{opp}}\rightarrow \cD$ preserves small limits if and only if $F$ is a sheaf for the effective epimorphism topology on $\cX$. 
    
\end{corollary}

\begin{proof}

    Just as in the proof of \ref{univ prperty of n-1 sheaves}, it now suffices to prove that a map $F:\cX^{\textnormal{opp}}\rightarrow \tau_{\leq n-1}\widehat{\mS}$ is a sheaf for the effective epimorphism topology on $\cX$ precisely when $F$ preserves small limits. Invoking \ref{preservation of trunactions 1}, we deduce that whenever $F:\cX^{\textnormal{opp}}\rightarrow \tau_{\leq n-1}\widehat{\mS}$ preserves small limits, $F$ is a $\C an(\cX)$-sheaf, proving one direction of our claim. 
    
    We now prove the converse. Choose some small $(n,1)$-category $\C$, a Grothendieck topology $J$ on $\C$, and an equivalence $\textnormal{Shv}^J_{n-1}(\C)\xrightarrow{\sim}\cX$. Since every sheaf $F:\textnormal{Shv}^J_{n-1}(\C)^{\textnormal{opp}}\rightarrow \tau_{\leq n-1}\widehat{\mS}$ for the effective epimorphism topology on $\textnormal{Shv}^J_{n-1}(\C)$ is an $(n-1)$-truncated object of $\widehat{\textnormal{Shv}}_{\C an (\textnormal{Shv}^J_{n-1}(\C))}(\textnormal{Shv}^J_{n-1}(\C))$, each is hypercomplete (see \citeq{highertopostheory} 6.5.2.9). Combining \ref{Hypersheaf Criteria} with \ref{Basis for Eff Epi n Topoi}, it follows that $F$ is the right Kan extension of a $J$-sheaf along the map $\C^{\textnormal{opp}}\rightarrow \textnormal{Shv}^J_{n-1}(\C)^{\textnormal{opp}}$, so our result now follows from \ref{univ prperty of n-1 sheaves}.

\end{proof}

\noindent We conclude this section by proving \ref{n-truncated preserves small limits: Appendix} (2):

\begin{corollary}{}\label{Factorization of m sheaves for the eff epi topology}

    Let $n\in \ZZ_{\geq 1}\cup\{\infty\}$, and let $\cX$ be an $(n,1)$-topos. Let $m\in \ZZ_{\geq 0}$ be an integer such that $m \leq n$ if $n\in \ZZ_{\geq 1}$. Then, for all $(m,1)$-categories $\cD$, a map $F:\cX^{\textnormal{opp}}\rightarrow \cD$ is a sheaf for the effective epimorphism topology on $\cX$ if and only if $F$ factors as 

    $$\cX^{\textnormal{opp}}\xrightarrow{(\tau_{\leq m-1})^{\textnormal{opp}}}\tau_{\leq m-1}\cX^{\textnormal{opp}}\xrightarrow{F_{m-1}}\cD,$$

    \noindent where $F_{m-1}$ is a sheaf for the effective epimorphism topology on $\tau_{\leq m-1}\cX$.
    
\end{corollary}

\begin{proof} We first prove our claim in the case that $n=\infty$. To begin, let $F:\cX^{\textnormal{opp}}\rightarrow \cD$ be a sheaf for the effective epimorphism topology on $\cX$. Combining \ref{Partial. Truncated Objects of Shv Can X} with \ref{localizations lemma}, we deduce that $F$ factors as $\cX^{\textnormal{opp}}\xrightarrow{(\tau_{\leq m-1})^{\textnormal{opp}}}\tau_{\leq m-1}\cX^{\textnormal{opp}}\xrightarrow{F_{m-1}}\cD,$ where $F_{m-1}$ is a sheaf for the effective epimorphism topology on $\cX$. The converse follows immediately from combining \ref{n topoi sheaves preserve small limits} with \ref{Cech Descent Theorem}. 

We now prove our claim in the case that $n\in \ZZ_{\geq 1}$. Suppose again that $F:\cX^{\textnormal{opp}}\rightarrow \cD$ is a sheaf for the effective epimorphism topology on $\cX$. Choose some $(\infty,1)$-topos $\cY$, and an equivalence $\cX\cong \tau_{\leq n-1}\cY$. Since the composition $\cY^{\textnormal{opp}}\rightarrow \cX^{\textnormal{opp}}\xrightarrow{F}\cD$ is a sheaf for the effective epimorphism topology on $\cY$, it follows from the above that the composition $\cY^{\textnormal{opp}}\rightarrow \cX^{\textnormal{opp}}\xrightarrow{F}\cD$ factors as $\cY^{\textnormal{opp}}\rightarrow (\tau_{\leq m-1}\cY)^{\textnormal{opp}}\xrightarrow{F^\prime}\cD$ where $F^\prime$ is a sheaf for the effective epimorphism topology on $\tau_{\leq m-1}\cY$. It now follows that $F$ factors as $\cX^{\textnormal{opp}}\rightarrow (\tau_{\leq m-1}\cX)^{\textnormal{opp}}\xrightarrow{F^{\prime\prime}}\cD$, where $F^{\prime\prime}$ is identified with $F^\prime$ under the equivalence $\tau_{\leq m-1}\cY\cong \tau_{\leq m-1}\cX$. The converse follows immediately from \ref{n topoi sheaves preserve small limits}.

\end{proof}

\subsection{Sheaves on Topoi}\label{Sheaves on Topoi Subsection}

The main result of this section is the following theorem:

\begin{theorem}{}\label{Sheaves on Topoi Result}

     Let $\cX$ be an $(\infty,1)$-topos, and write $\C an(\cX)$ for the effective epimorphism topology on $\cX$. Then:

     \begin{itemize}

         \item [$(1)$] The inclusion $\textnormal{Shv}_{\widehat{\mS}}(\cX)\hookrightarrow \textnormal{Fun}(\cX^{\textnormal{opp}},\widehat{\mS})$ factors through $\widehat{\textnormal{Shv}}_{\C an(\cX)}(\cX)\subset \textnormal{Fun}(\cX^{\textnormal{opp}},\widehat{\mS})$. Moreover, the left adjoint $C:\widehat{\textnormal{Shv}}_{\C an(\cX)}(\cX)\rightarrow \textnormal{Shv}_{\widehat{\mS}}(\cX)$ induced by HTT 6.3.5.28 is a cotopological localization (in our larger universe).

         \item [$(2)$] For all $n\in \ZZ_{\geq 0}$, pulling-back along the localization $\cX\rightarrow \tau_{\leq n}\cX$ yields a categorical equivalence $\widehat{\textnormal{Shv}}_{\tau_{\leq n}\widehat{\mS}}(\tau_{\leq n}\cX)\xrightarrow{\sim} \tau_{\leq n-1}\widehat{\textnormal{Shv}}_{\C an(\cX)}(\cX)$, where $\widehat{\textnormal{Shv}}_{\tau_{\leq n}\widehat{\mS}}(\tau_{\leq n}\cX)$ is the full subcategory of $\textnormal{Fun}(\tau_{\leq n}\cX^{\textnormal{opp}},\tau_{\leq n}\widehat{\mS})$ spanned by those maps $F:\tau_{\leq n}\cX^{\textnormal{opp}}\rightarrow \tau_{\leq n}\widehat{\mS}$ which preserve small limits.

         \item [$(3)$] Pulling-back along $\cX\rightarrow \cX^{\textnormal{hyp}}$ yields a categorical equivalence $\textnormal{Shv}_{\widehat{\mS}}(\cX^{\textnormal{hyp}})\xrightarrow{\sim} \widehat{\textnormal{Shv}}_{\C an(\cX)}(\cX)^{\textnormal{hyp}}$.
         
     \end{itemize}
    
\end{theorem}

\begin{remark}{}

 In light of $(1)$, the respective classifications of the truncated and hypercomplete objects of $\widehat{\textnormal{Shv}}_{\C an(\cX)}(\cX)$ given by $(2)$ and $(3)$ also classify the analogous data for $\textnormal{Shv}_{\widehat{\mS}}(\cX)$.
    
\end{remark}

\noindent Before proving \ref{Sheaves on Topoi Result}, we discuss several consequences: 

\begin{corollary}{}

    Let $\cX$ be a hypercomplete $(\infty,1)$-topos. Then, a map $F:\cX^{\textnormal{opp}}\rightarrow \mS$ is a hypersheaf with respect to the effective epimorphism topology on $\cX$ if and only if $F$ is representable. 
    
\end{corollary}

\noindent This is the higher analogue of the following, which may be deduced from \ref{n-truncated preserves small limits: Appendix}:

\begin{corollary}{}\label{Sheaves for n topoi are representable}

    Let $n\in \ZZ_{\geq 1}$, and let $\cX$ be an $(n,1)$-topos. Then, a map $F:\cX^{\textnormal{opp}}\rightarrow \tau_{\leq n-1}\mS$ is a sheaf for the effective epimorphism topology on $\cX$ if and only if $F$ is representable. 
    
\end{corollary}

\noindent We again note that the case of \ref{Sheaves for n topoi are representable} for $n=1$ is a classical result (see \citeq{Johnstone2002-JOHSOA-7} C.2.2.7). In the case that $\cX$ is not hypercomplete, we have the following connection to \ref{Sheaves for n topoi are representable}:

\begin{corollary}{}\label{The Canonical Topology on a Topos 1}

    Let $n\in \ZZ_{\geq 1}\cup \{\infty\}$, and let $\cX$ be an $(n,1)$-topos. Let $Y\in \cX$, and let $\cX^0_{/Y}\subseteq \cX_{/Y}$ be a sieve on $Y$ such that every representable functor in $\textnormal{Fun}(\cX^{\textnormal{opp}},\widehat{\mS})$ is local with respect to the induced monomorphism $m:S\rightarrow h^\cX_Y$. Then, $\cX^0_{/Y}\subseteq \cX_{/Y}$ is a sieve for the effective epimorphism topology on $\cX$. 
    
\end{corollary}

\begin{proof}

    We first show this in the case that $n=\infty$. Write $L:\textnormal{Fun}(\cX^{\textnormal{opp}},\widehat{\mS})\rightarrow \textnormal{Shv}_{\widehat{\mS}}(\cX)$ be the map left adjoint to the inclusion. Applying \citeq{highertopostheory} 6.3.5.28, we deduce that whenever $F:\cX^{\textnormal{opp}}\rightarrow \widehat{\mS}$ factors as $\cX^{\textnormal{opp}}\rightarrow \mS\hookrightarrow \widehat{\mS}$, so too does $L(F):\cX^{\textnormal{opp}}\rightarrow \widehat{\mS}$. Next, let $m:S\rightarrow h^\cX_Y$ be a monomorphism in $\textnormal{Fun}(\cX^{\textnormal{opp}},\widehat{\mS})$ such that every representable functor is $m$-local. Since $m$ is a monomorphism we may identify $m$ as an edge $m:S\rightarrow h^\cX_Y$ in $\textnormal{Fun}(\cX^{\textnormal{opp}},\mS)$, and it immediately follows that $m$ is an $L^\prime$-equivalence. Since $L^\prime$ is the restriction of $L$ (in the appropriate sense), it follows that $m$ is an $L^\prime$-equivalence. Writing $T:\textnormal{Fun}(\cX^{\textnormal{opp}},\widehat{\mS})\rightarrow \widehat{\textnormal{Shv}}_{\C an(\cX)}(\cX)$ for the sheafification functor, combining \citeq{highertopostheory} 6.2.2.16, \citeq{highertopostheory} 6.5.2.16 and \ref{Sheaves on Topoi Result} implies that $m$ is also a $T$-equivalence, proving this part of our claim.

    We now prove our claim in the case that $n\in \ZZ_{\geq 1}$. Choose some $(\infty,1)$-topos $\cY$, and an equivalence $\tau_{\leq n-1}\cY\cong \cX$. Let $G:\textnormal{Fun}(\cY^{\textnormal{opp}},\widehat{\mS})\rightarrow \textnormal{Shv}_{\widehat{\mS}}(\cY)$ be left adjoint to the inclusion. Noting now that the arrow $\textnormal{Fun}(\cX^{\textnormal{opp}},\widehat{\mS})\rightarrow \textnormal{Fun}(\cY^{\textnormal{opp}},\widehat{\mS})$ is fully faithful and write $G^\prime:\textnormal{Fun}(\cX^{\textnormal{opp}},\widehat{\mS})\rightarrow \textnormal{Shv}_{\widehat{\mS}}(\cY)$ for the map induced by the restriction of $G$. Since $G$ is left exact, $G^\prime$ restricts to an arrow $G^{\prime\prime}:\textnormal{Fun}(\cX^{\textnormal{opp}},\tau_{\leq n-1}\widehat{\mS})\rightarrow \tau_{\leq n-1}\textnormal{Shv}_{\widehat{\mS}}(\cY)$. It follows from \ref{Sheaves on Topoi Result} (2) that $G^{\prime\prime}$ is isomorphic to the map $F:\textnormal{Fun}(\cX^{\textnormal{opp}},\tau_{\leq n-1}\widehat{\mS})\rightarrow \widehat{\textnormal{Shv}}\vphantom{}^{\C an(\cX)}_{n-1}(\cX)$ left adjoint to the inclusion, and it now follows from \citeq{highertopostheory} 6.3.5.28 that $F$ restricts (in the appropriate sense) to a map $\textnormal{Fun}(\cX^{\textnormal{opp}},\tau_{\leq n-1}\mS)\rightarrow \widehat{\textnormal{Shv}}\vphantom{}_{n-1}^{\C an(\cX)}(\cX)\cong \cX$ left adjoint to the Yoneda embedding. Our claim now follows from a similar logic to that used in the $n=\infty$ case above. 
    
\end{proof}

\noindent Corollary \ref{The Canonical Topology on a Topos 1} immediately implies the following:

\begin{corollary}{}\label{The Canonical Topology on a Topos}
    Let $n\in \ZZ_{\geq 1}\cup \{\infty\}$, and let $\cX$ be an $(n,1)$-topos. Then, there exists a unique finest Grothendieck topology $J$ on $\cX$ satisfying the property that every representable functor $\cX^{\textnormal{opp}}\rightarrow \mS$ is a $J$-sheaf. Moreover, this is the effective epimorphism topology on $\cX$. 
\end{corollary}

\begin{corollary}{}\label{Sheafification For n Topoi}

Choose some integer $n\in \ZZ_{\geq 1}$, and let $\cX$ be an $(n,1)$-topos. Then, the sheafification functor $\textnormal{Fun}(\cX^{\textnormal{opp}},\tau_{\leq n-1}\widehat{\mS})\rightarrow \widehat{\textnormal{Shv}}\vphantom{}^{\C an(\cX)}_{n-1}(\cX)$ restricts to a left exact map $\textnormal{Fun}(\cX^{\textnormal{opp}},\tau_{\leq n-1}\mS)\rightarrow \widehat{\textnormal{Shv}}\vphantom{}_{n-1}^{\C an(\cX)}(\cX)\cong \cX$ left adjoint to the Yoneda embedding. 
  
\end{corollary}

\begin{proof}

    See the proof of \ref{The Canonical Topology on a Topos 1}.
    
\end{proof}

\noindent We also have the following the characterization of geometric morphisms of $(n,1)$-topoi:

\begin{corollary}{}

    Let $n\in \ZZ_{\geq 1}\cup \{\infty\}$, and let $m\in \ZZ_{\geq 0}$ such that $m\leq n$ if $n\in \ZZ_{\geq 1}$. Let $\cX$ be an $(n,1)$-topos, and let $\cD$ be an $(m,1)$-category. Let $F:\cX\rightarrow \cD$ be an $(\infty,1)$-functor which preserves pullbacks and small coproducts. Then, $F$ preserves small colimits if and only if $F$ preserves effective epimorphisms.

\end{corollary}

\noindent We now prove \ref{Sheaves on Topoi Result}. Noting that the factorization of $\widehat{\textnormal{Shv}}_{\C an(\cX)}(\cX)\hookrightarrow \textnormal{Fun}(\cX^{\textnormal{opp}},\widehat{\mS})$ described in (1) is implied by \ref{Cech Descent Theorem}, we may deduce the remainder of (1) from \ref{n-truncated preserves small limits: Appendix}:

\begin{proof} (\ref{Sheaves on Topoi Result} (1).)  Write $i:\cX\rightarrow \widehat{\textnormal{Shv}}_{\C an(\cX)}(\cX)$ for the Yoneda embedding. Let $K$ be a small simplicial set, let $F:K\rightarrow \cX$ be a diagram, and let $\overline{F}:K^\triangleright\rightarrow \cX$ be a diagram which exhibits some $X\in \cX$ as the colimit of $F$. Writing $\overline{i\circ F}:K^\triangleright\rightarrow\widehat{\textnormal{Shv}}_{\C an(\cX)}(\cX)$ for a map which exhibits some $G\in \widehat{\textnormal{Shv}}_{\C an(\cX)}(\cX)$ as the colimit of $i\circ F$, we have a canonical morphism $j:G\rightarrow i(X)$. In order to prove our claim, it suffices to show that $j$ is $\infty$-connective. Applying \citeq{highertopostheory} 6.5.1.14, it suffices to show that for all $n\in \ZZ_{\geq 0}$ and $U\in \tau_{\leq n}(\widehat{\textnormal{Shv}}_{\C an(\cX)}(\cX)_{/i(X)})\subseteq \widehat{\textnormal{Shv}}_{\C an(\cX)}(\cX)_{/i(X)}$, $U$ is local with respect to the map $j^\prime:j\rightarrow \textnormal{id}_{i(X)}$, where $j^\prime$ denotes the right degenerate $2$-simplex of $\widehat{\textnormal{Shv}}_{\C an(\cX)}(\cX)$ corresponding to $j$. Fix such an $n\in \ZZ_{\geq 0}$ and such a $U\in \tau_{\leq n}(\widehat{\textnormal{Shv}}_{\C an(\cX)}(\cX)_{/i(X)})$. Next, identify $j:G\rightarrow i(X)$ as an edge $K\star \Delta^1\rightarrow \widehat{\textnormal{Shv}}_{\C an(\cX)}(\cX)$ of $\widehat{\textnormal{Shv}}_{\C an(\cX)}(\cX)_{(i\circ F)/}$. Noting that $\Delta^1$ is weakly contractible, \citeq{lurie2024kerodon} 0198 implies that $\Delta^1\hookrightarrow K\star \Delta^1$ is right anodyne. In particular, there exists an extension $J:K\star \Delta^1\rightarrow \widehat{\textnormal{Shv}}_{\C an(\cX)}(\cX)_{/i(X)}$ of $j^\prime:\Delta^1\rightarrow \widehat{\textnormal{Shv}}_{\C an(\cX)}(\cX)_{/i(\cX)}$ along $\Delta^1\hookrightarrow K\star \Delta^1$ such that the composition $K\star \Delta^1\xrightarrow{J}\widehat{\textnormal{Shv}}_{\C an(\cX)}(\cX)_{/i(\cX)}\rightarrow \widehat{\textnormal{Shv}}_{\C an(\cX)}(\cX)$ agrees with the edge of $\widehat{\textnormal{Shv}}_{\C an(\cX)}(\cX)_{(i\circ F)/}$ induced by $j:G\rightarrow i(X)$. In order to prove our claim, it now suffices to show that the composition 

$$\Delta^1\hookrightarrow \Delta^1\star (K^\textnormal{opp})=(\Delta^1)^\textnormal{opp}\star (K^\textnormal{opp})=(K\star \Delta^1)^{\textnormal{opp}}\xrightarrow{J^{\textnormal{opp}}}(\widehat{\textnormal{Shv}}_{\C an(\cX)}(\cX)_{/i(\cX)})^{\textnormal{opp}}\xrightarrow{h_U}\tau_{\leq n}\widehat{\mS}$$

\noindent is an isomorphism in $\tau_{\leq n}\widehat{\mS}$. Since the composition $K^\triangleright=K\star \{0\}\hookrightarrow K\star \Delta^1\xrightarrow{J}\widehat{\textnormal{Shv}}_{\C an(\cX)}(\cX)_{/i(\cX)}\rightarrow \widehat{\textnormal{Shv}}_{\C an(\cX)}(\cX)$ is a colimit diagram, it suffices to prove that the composition 

$$(K^{\textnormal{opp}})^\triangleleft=(K\star \{1\})^{\textnormal{opp}}\hookrightarrow (K\star \Delta^1)^{\textnormal{opp}}\xrightarrow{J^{\textnormal{opp}}}(\widehat{\textnormal{Shv}}_{\C an(\cX)}(\cX)_{/i(\cX)})^{\textnormal{opp}}\xrightarrow{h_U}\tau_{\leq n}\widehat{\mS},$$

\noindent which we will denote by $F^\prime:(K^{\textnormal{opp}})^\triangleleft\rightarrow \widehat{\mS}$, is a limit diagram. Writing $\phi:\widehat{\textnormal{Shv}}_{\C an(\cX)}(\cX)_{/i(\cX)}\xrightarrow{\sim}\widehat{\textnormal{Shv}}_{\C an(\cX_{/X})}(\cX_{/X})$ for the categorical equivalence of \ref{Slices of sheaves of eff. epis.}, it follows from \ref{Equivalence of slice topoi explicit Remark} that $F^\prime$ may be identified as the composition $(K^{\textnormal{opp}})^\triangleleft=(K^\triangleright)^{\textnormal{opp}}\xrightarrow{F_X^{\textnormal{opp}}}(\cX_{/X})^\textnormal{opp}\xrightarrow{\phi(U)}\tau_{\leq n}\widehat{\mS}$, where $F_X$ is a lifting of $F:K^\triangleright\rightarrow \cX$ along the cone point. Our claim now follows from \ref{n-truncated preserves small limits: Appendix}.

\end{proof}

\noindent Noting that \ref{Sheaves on Topoi Result} (2) is an immediate corollary of \ref{n-truncated preserves small limits: Appendix}, it only remains to prove \ref{Sheaves on Topoi Result} (3):

\begin{proof}

    Let $\cX$ be an $(\infty,1)$-topos, and let $F:\cX^{\textnormal{opp}}\rightarrow \widehat{\mS}$ be a map which preserves small limits. Invoking \ref{Cech Descent Theorem}, it suffices to prove the following: $F$ carries every $\infty$-connective morphism in $\cX$ to an isomorphism in $\widehat{\mS}$ if and only if $F$, for all $\C an(\cX)$-hypercoverings $X_\bullet:(\Delta^{\textnormal{opp}}_{\textnormal{inj}})^\triangleright\rightarrow \cX$, the composition $\Delta_{\textnormal{inj}}^\triangleleft\xrightarrow{X_\bullet^{\textnormal{opp}}}\cX^{\textnormal{opp}}\xrightarrow{F}\widehat{\mS}$ is a limit diagram. 

    To begin, suppose that $F$ carries every $\infty$-connective morphism in $\cX$ to an isomorphism in $\widehat{\mS}$, and let $X_\bullet:(\Delta^{\textnormal{opp}}_{\textnormal{inj}})^\triangleright\rightarrow \cX$ be a hypercovering of $X\in \cX$ for the effective epimorphism topology on $\cX$. Identify $X_\bullet$ as a map $X^\prime_\bullet:\Delta^{\textnormal{opp}}_{\textnormal{inj}}\rightarrow \cX_{/X}$, and extend $X^\prime_\bullet$ to a colimit diagram $\overline{X^\prime_\bullet}:(\Delta^{\textnormal{opp}}_{\textnormal{inj}})^\triangleright\rightarrow \cX_{/X}$, noting the cone point of of $\overline{X^\prime_\bullet}$ is $\infty$-connective (\citeq{lurie2018sag} A.5.6.3). In particular, $\overline{X^\prime_\bullet}$ is $\infty$-connective as an edge of $\cX$ (\citeq{highertopostheory} 6.5.1.19), so it immediately follows that the composition $\Delta_{\textnormal{inj}}^\triangleleft\xrightarrow{X_\bullet^{\textnormal{opp}}}\cX^{\textnormal{opp}}\xrightarrow{F}\widehat{\mS}$ is a limit diagram.

    We now prove the converse. Let $f:X\rightarrow Y$ be an $\infty$-connective morphism in $\cX$, which we identify as a vertex of $\cX_{/Y}$. Let $\underline{f}:\Delta^{\textnormal{opp}}_{\textnormal{inj}}\rightarrow \cX_{/Y}$ be the constant diagram with value $f$, which we identify as an arrow $\underline{f}^\prime:(\Delta^{\textnormal{opp}}_{\textnormal{inj}})^\triangleright\rightarrow \cX$ with conepoint $Y$. Applying \citeq{highertopostheory} 6.5.3.5 we deduce that $\underline{f}^\prime$ is a $\C an(\cX)$-hypercovering of $Y$, implying that the composition $\Delta_{\textnormal{inj}}^\triangleleft\xrightarrow{(\underline{f}^\prime)^{\textnormal{opp}}}\cX^{\textnormal{opp}}\xrightarrow{F}\widehat{\mS}$ is a limit diagram. Since every edge of the underlying map $\Delta_{\textnormal{inj}}\rightarrow \widehat{\mS}$ is an isomorphism and $\Delta_{\textnormal{inj}}$ is weakly contractible, it follows that the composition $\Delta_{\textnormal{inj}}^\triangleleft\xrightarrow{(\underline{f}^\prime)^{\textnormal{opp}}}\cX^{\textnormal{opp}}\xrightarrow{F}\widehat{\mS}$ carries every edge of $\Delta_{\textnormal{inj}}^\triangleleft$ to an isomorphism. Our claim now follows from the fact that for each $n\in \ZZ_{\geq 0}$, the arrow $\underline{f}^\prime:(\Delta^{\textnormal{opp}}_{\textnormal{inj}})^\triangleright\rightarrow \cX$ carries the $[n]\rightarrow [-1]$ edge of $(\Delta^{\textnormal{opp}}_{\textnormal{inj}})^\triangleright$ to the map $f:X\rightarrow Y$.

\end{proof}

\subsection{Global Sheaves on Topoi}\label{Global Sheaves Section}

As discussed in \ref{Geom Sites Examples}, the \'{e}tale topology on the $(\infty,1)$-category of $(\infty,1)$-topoi, $\infty\cT op$, is generated by a geometric site. Just as the effective epimorphism topology on an $(\infty,1)$-topos $\cX$ is a good starting point when working to identify the appropriate notion of descent on $\cX$, the \'{e}tale topology on $\infty\cT op$ is a good starting point when working to identify the appropriate notion of descent on $\infty\cT op$. However, just as the appropriate notion of a sheaf on an $(\infty,1)$-topos $\cX$ does not in general correspond to the effective epimorphism topology on $\cX$ (though \ref{Sheaves on Topoi Result} establishes the fact that these notions are intimately related), the additional structure of higher category theory allows one to instead ask for a stronger notion of descent on $\infty\cT op$ than any Grothendieck topology can provide:

\begin{definition}{}[HTT 6.3.5.19]\label{Global Sheaves on Topoi Defn}

    Let $\cD$ be an $(\infty,1)$-category, and let $F:\infty\cT op^{\textnormal{opp}}\rightarrow \cD$ be a map. We call $F$ a \emph{sheaf} on $\infty\cT op$ if $F$ satisfies either of the following equivalent conditions:

    \begin{itemize}
    
        \item [$(1)$] The restriction of $F$ to $\infty\cT op^{\textnormal{\'{e}t},\textnormal{opp}}\subset \infty\cT op^{\textnormal{opp}}$ preserves small limits.

        \item [$(2)$] For all $(\infty,1)$-topoi $\cX$, the composition $\cX^{\textnormal{opp}}\xrightarrow{\sim}(\infty\cT op^{\textnormal{\'{e}t}}_{/\cX})^{\textnormal{opp}}\rightarrow \infty\cT op^{\textnormal{\'{e}t},\textnormal{opp}}\rightarrow \infty\cT op^{\textnormal{opp}}\xrightarrow{F}\cD$ is a sheaf on $\cX$, where $\cX^{\textnormal{opp}}\xrightarrow{\sim}(\infty\cT op^{\textnormal{\'{e}t}}_{/\cX})^{\textnormal{opp}}$ is the canonical equivalence of \citeq{highertopostheory} 6.3.5.10. 
        
    \end{itemize}
    
\end{definition}

\noindent We will use the following relative variant of \ref{Global Sheaves on Topoi Defn}:

\begin{definition}{}\label{Relative Global Sheaves on Topoi Defn}

    Let $F:\infty\cT op^{\textnormal{opp}}\rightarrow \widehat{\mS}$ be a sheaf on $\infty\cT op$, and let $\cD$ be an $(\infty,1)$-category. Noting that $\infty\cT op^{\textnormal{\'{e}t}}_{/F}:=\infty\cT op^{\textnormal{\'{e}t}}\times_{\infty\cT op}\infty\cT op_{/F}$ admits small colimits which are preserved by $\infty\cT op^{\textnormal{\'{e}t}}_{/F}\hookrightarrow \infty\cT op_{/F}$ (see \ref{Colimits over Global Sheaves on Topoi}), we will call a map $G:(\infty\cT op_{/F})^{\textnormal{opp}}\rightarrow \cD$ a \emph{sheaf} on $\infty\cT op_{/F}$ if $G$ satisfies either of the following equivalent conditions:

    \begin{itemize}
    
        \item [$(1)$] The restriction of $G$ to $(\infty\cT op^{\textnormal{\'{e}t}}_{/F})^\textnormal{opp}$ preserves small limits.

        \item [$(2)$] For all $(\infty,1)$-topoi $\cX$ and maps $\eta:h_\cX\rightarrow F$, the composition $\cX^{\textnormal{opp}}\xrightarrow{\sim}((\infty\cT op^{\textnormal{\'{e}t}}_{/F})_{/(\cX,\eta)})^{\textnormal{opp}}\rightarrow (\infty\cT op_{/F})^{\textnormal{opp}}\xrightarrow{G}\cD$ is a sheaf on $\cX$.
        
    \end{itemize}
    
\end{definition}

\begin{remark}{}\label{Notation Conflict For Global Sheaves on Topoi}

    In an unfortunate instance of notational conflict, it is \emph{not} true that for an $(\infty,1)$-topos $\cX$ the two $(\infty,1)$-categories $\infty\cT op^{\textnormal{\'{e}t}}_{/h_\cX}$ and $\infty\cT op^{\textnormal{\'{e}t}}_{/\cX}$ coincide (unless $\cX$ is the inital $(\infty,1)$-topos). Instead, $\infty\cT op^{\textnormal{\'{e}t}}_{/h_\cX}$ may be identified as the subcategory of $\infty\cT op_{/\cX}$ spanned by the morphisms $f_*:\cU\rightarrow \mathcal{V}$ over $\cX$ which are themselves \'{e}tale, and $\infty\cT op^{\textnormal{\'{e}t}}_{/\cX}$ is the full subcategory of $\infty\cT op_{/\cX}$ spanned by the \'{e}tale morphisms $\cU\rightarrow \cX$ (noting $\infty\cT op^{\textnormal{\'{e}t}}_{/\cX}$ may then also be identified as a full subcategory of $\infty\cT op^{\textnormal{\'{e}t}}_{/h_\cX}$). 
    
\end{remark}

\begin{proposition}{}\label{Colimits over Global Sheaves on Topoi}

 We work under the notation of \ref{Relative Global Sheaves on Topoi Defn}. Let $K$ a small simplicial set, and let $\phi:K\rightarrow \infty\cT op^{\textnormal{\'{e}t}}_{/F}$ be a diagram. Then, $\phi$ admits a colimit in $\infty\cT op^{\textnormal{\'{e}t}}_{/F}$. Moreover, for any extension $\overline{\phi}:K^\triangleright\rightarrow \infty\cT op^{\textnormal{\'{e}t}}_{/F}$ of $\phi$, the following are equivalent:

 \begin{itemize}
     \item [$(1)$] $\overline{\phi}:K^\triangleright\rightarrow \infty\cT op^{\textnormal{\'{e}t}}_{/F}$ is a colimit diagram. 

     \item [$(2)$] The composition $K^\triangleright\xrightarrow{\overline{\phi}} \infty\cT op^{\textnormal{\'{e}t}}_{/F}\hookrightarrow \infty\cT op_{/F}$ is a colimit diagram. 

     \item [$(3)$]  The composition $K^\triangleright\xrightarrow{\overline{\phi}} \infty\cT op^{\textnormal{\'{e}t}}_{/F}\rightarrow\infty\cT op^{\textnormal{\'{e}t}}$ is a colimit diagram. 

     \item [$(4)$]  The composition $K^\triangleright\xrightarrow{\overline{\phi}} \infty\cT op^{\textnormal{\'{e}t}}_{/F}\rightarrow\infty\cT op$ is a colimit diagram. 
 \end{itemize}
    
\end{proposition}

\begin{proof}

    Combine \citeq{highertopostheory} 6.3.5.13 with \ref{Colimits in Right Fibrations}. 
    
\end{proof}

\noindent The main result of this section is the following:

\begin{proposition}{}\label{Global Sheaves on Topoi Prop}

   Let $F:\infty\cT op^{\textnormal{opp}}\rightarrow \widehat{\mS}$ be a sheaf, and write $\widehat{\textnormal{Shv}}(\infty\cT op_{/F})$ for the full subcategory of $\textnormal{Fun}((\infty\cT op_{/F})^{\textnormal{opp}},\widehat{\mS})$ spanned by the sheaves $\infty\cT op^{\textnormal{opp}}\rightarrow \widehat{\mS}$. Writing $\widehat{\textnormal{Shv}}_{\textnormal{\'{e}t}_{/F}}(\infty\cT op_{/F})$ for the full subcategory spanned by the sheaves for the \'{e}tale topology on $\infty\cT op_{/F}$ (in the sense of \ref{Etale Slices of Sheaf Topoi} (1)), we then have the following:

    \begin{itemize}
    
        \item [$(1)$] The inclusion $\widehat{\textnormal{Shv}}(\infty\cT op_{/F})\hookrightarrow \textnormal{Fun}((\infty\cT op_{/F})^{\textnormal{opp}},\widehat{\mS})$ factors as the composition $\widehat{\textnormal{Shv}}(\infty\cT op_{/F})\hookrightarrow \widehat{\textnormal{Shv}}_{\textnormal{\'{e}t}_{/F}}(\infty\cT op_{/F})\hookrightarrow \textnormal{Fun}((\infty\cT op_{/F})^{\textnormal{opp}},\widehat{\mS})$. 

        \item [$(2)$] The inclusion $\widehat{\textnormal{Shv}}(\infty\cT op_{/F})\hookrightarrow \widehat{\textnormal{Shv}}_{\textnormal{\'{e}t}_{/F}}(\infty\cT op_{/F})$ induced by $(1)$ admits a cotopologoical (in our larger universe) left adjoint $C:\widehat{\textnormal{Shv}}_{\textnormal{\'{e}t}_{/F}}(\infty\cT op_{/F})\rightarrow \widehat{\textnormal{Shv}}(\infty\cT op_{/F})$. 
        
    \end{itemize}

    \noindent Moreover, the inclusion $\widehat{\textnormal{Shv}}(\infty\cT op_{/F})\hookrightarrow \textnormal{Fun}((\infty\cT op_{/F})^{\textnormal{opp}},\widehat{\mS})$ identifies $\widehat{\textnormal{Shv}}(\infty\cT op_{/F})^{\textnormal{hyp}}$ with the full subcategory of $\textnormal{Fun}((\infty\cT op_{/F})^{\textnormal{opp}},\widehat{\mS})$ spanned by those maps $G:(\infty\cT op_{/F})^{\textnormal{opp}}\rightarrow \widehat{\mS}$ which satisfy the following two properties:

    \begin{itemize}
        \item [$(\textnormal{i})$] $G$ preserves small products. 

        \item [$(\textnormal{ii})$] For all augmented semisimplicial objects $X_\bullet:(\Delta_{\textnormal{inj}}^{\textnormal{opp}})^\triangleright\rightarrow \infty\cT op_{/F}$ satisfying the property that for each $n\in \ZZ_{\geq 0}$ the $n$th covering map $X_n\rightarrow M_n(X)$ is both \'{e}tale and an effective epimorphism of $(\infty,1)$-topoi, the composition $\Delta_{\textnormal{inj}}^\triangleleft\xrightarrow{X^{\textnormal{opp}}_\bullet} (\infty\cT op_{/F})^{\textnormal{opp}}\xrightarrow{G} \widehat{\mS}$ is a limit diagram.

    \end{itemize}
    
\end{proposition}

\begin{proof}

Claim $(1)$ follows from \ref{Cech Descent Theorem}. Once we prove $(2)$,  the component of our claim identifying the hypercomplete objects of $\widehat{\textnormal{Shv}}(\infty\cT op)$ will also follow from \ref{Cech Descent Theorem} (see also \ref{Geom Sites Examples} (6)), so it only remains to $(2)$. To begin, we wish to show that the equivalence $\textnormal{Fun}((\infty\cT op)_{/F})^{\textnormal{opp}},\widehat{\mS})\xrightarrow{\sim} \textnormal{Fun}(\infty\cT op^{\textnormal{opp}},\widehat{\mS})_{/F}$ restricts to an equivalence $\widehat{\textnormal{Shv}}(\infty\cT op_{/F})\xrightarrow{\sim}\widehat{\textnormal{Shv}}(\infty\cT op)_{/F}$, where $\widehat{\textnormal{Shv}}(\infty\cT op)$ is the full subcategory of $\textnormal{Fun}(\infty\cT op^{\textnormal{opp}},\widehat{\mS})$ spanned by the sheaves. Applying \ref{Left Kan extensions of slice sheaves are sheaves}, it suffices to show that $E:((\infty\cT op)_{/F})^{\textnormal{opp}}\rightarrow \widehat{\mS}$ is a sheaf if and only if the right fibration given by the composition $(\infty\cT op_{/F})_{/E}\rightarrow \infty\cT op_{/F}\rightarrow \infty\cT op$ is classified by a sheaf $E^\prime:\infty\cT op^\textnormal{opp}\rightarrow \widehat{\mS}$. This fact may be deduced from combining \ref{Colimits over Global Sheaves on Topoi} with \ref{Colimits in Right Fibrations}.

In order to now prove our claim, it suffices to show that the map $L_F:\widehat{\textnormal{Shv}}_{\textnormal{\'{e}t}}(\infty\cT op)_{/F}\rightarrow \widehat{\textnormal{Shv}}(\infty\cT op)_{/F}$ left adjoint to the inclusion is a cotopological localization (in our larger universe). Noting that $F\in \widehat{\textnormal{Shv}}_{\textnormal{\'{e}t}}(\infty\cT op)$ is local with respect to the localization functor $L:\widehat{\textnormal{Shv}}_{\textnormal{\'{e}t}}(\infty\cT op)\rightarrow \widehat{\textnormal{Shv}}(\infty\cT op)$ induced by $(1)$, we may identify $L_F:\widehat{\textnormal{Shv}}_{\textnormal{\'{e}t}}(\infty\cT op)_{/F}\rightarrow \widehat{\textnormal{Shv}}(\infty\cT op)_{/F}$ as the map on slices induced by $L$. Note now the following:

\begin{itemize}

    \item [$(\textnormal{a})$] $L_F$ is left exact if $L:\widehat{\textnormal{Shv}}_{\textnormal{\'{e}t}}(\infty\cT op)\rightarrow \widehat{\textnormal{Shv}}(\infty\cT op)$ preserves pullbacks.

    \item [$(\textnormal{b})$] If $L$ is a localization at a (possibly large) set of arrows $S$ in $\widehat{\textnormal{Shv}}_{\textnormal{\'{e}t}}(\infty\cT op)$, then $L_F$ is a localization at $p^{-1}(S)$, where $p:\widehat{\textnormal{Shv}}_{\textnormal{\'{e}t}}(\infty\cT op)_{/F}\rightarrow \widehat{\textnormal{Shv}}_{\textnormal{\'{e}t}}(\infty\cT op)$ is the projection map (\ref{Reflective Localizations and Slices}).

    \item [$(\textnormal{c})$] An edge $f:X\rightarrow Y$ of $\widehat{\textnormal{Shv}}_{\textnormal{\'{e}t}}(\infty\cT op)_{/F}$ is $\infty$-connective if and only if $p(f):p(X)\rightarrow p(Y)$ is $\infty$-connective (\citeq{highertopostheory} 6.5.1.19).
    
\end{itemize}

\noindent It follows that $L_F$ is a cotopological localization whenever $L$ is a cotopological localization (both in our larger universe), so we are reduced to proving that $L$ is cotopological. The fact that $L$ is left exact follows from combining (the proof of) \citeq{highertopostheory} 6.2.3.20 with \citeq{highertopostheory} 6.3.5.21, so it only remains to show that $L$ restricts to a reflective localization at a (possibly large) set of $\infty$-connective morphisms in $\widehat{\textnormal{Shv}}_{\textnormal{\'{e}t}}(\infty\cT op)$. We will achieve this using a technique very similar to that used in the proof of \ref{Sheaves on Topoi Result}. 

Let $K$ be a small simplicial set, let $\psi:K\rightarrow \infty\cT op^{\textnormal{\'{e}t}}\subset \infty\cT op$ be a map, and let $\overline{\psi}:K^\triangleright\rightarrow \infty\cT op$ be a diagram which exhibits $\cX\in \infty\cT op$ as the colimit of $\psi$. Let $i:\infty\cT op\rightarrow \widehat{\textnormal{Shv}}_{\textnormal{\'{e}t}}(\infty\cT op)$ denote the Yoneda embedding, and write $\overline{i\circ \psi}:K^\triangleright\rightarrow \widehat{\textnormal{Shv}}_{\textnormal{\'{e}t}}(\infty\cT op)$ for a map which exhibits some $G\in \widehat{\textnormal{Shv}}_{\textnormal{\'{e}t}}(\infty\cT op)$ as the colimit of $i\circ \psi$, yielding a canonical morphism $j:G\rightarrow i(\cX)$. In order to prove our claim, it suffices to show that $j$ is $\infty$-connective. Writing $j^\prime:\Delta^1\rightarrow \widehat{\textnormal{Shv}}_{\textnormal{\'{e}t}}(\infty\cT op)_{/i(\cX)}$ for the map given by the right degenerate $2$-simplex $\Delta^2\rightarrow \widehat{\textnormal{Shv}}_{\textnormal{\'{e}t}}(\infty\cT op)$ corresponding to $j$, it suffices to prove that for all $n\in \ZZ_{\geq 0}$ and $n$-truncated $U\in \widehat{\textnormal{Shv}}_{\textnormal{\'{e}t}}(\infty\cT op)_{/i(X)}$, $U$ is local with respect to $j^\prime$ (\citeq{highertopostheory} 6.5.1.14). Fix such an $n\in \ZZ_{\geq 0}$ and $U\in \widehat{\textnormal{Shv}}_{\textnormal{\'{e}t}}(\infty\cT op)_{/i(X)}$. Identifying $j$ as an edge $K\star \Delta^1\rightarrow \widehat{\textnormal{Shv}}_{\textnormal{\'{e}t}}(\infty\cT op)$ of $\widehat{\textnormal{Shv}}_{\textnormal{\'{e}t}}(\infty\cT op)_{i\circ \psi/}$, and observing that the inclusion $\Delta^1\hookrightarrow K\star \Delta^1$ is right anodyne (\citeq{lurie2024kerodon} 0198), the map $K\star \Delta^1\rightarrow \widehat{\textnormal{Shv}}_{\textnormal{\'{e}t}}(\infty\cT op)$ admits a lift to an arrow $J:K\star \Delta^1\rightarrow \widehat{\textnormal{Shv}}_{\textnormal{\'{e}t}}(\infty\cT op)_{/i(X)}$ such that the composition $\Delta^1\hookrightarrow K\star \Delta^1\rightarrow \widehat{\textnormal{Shv}}_{\textnormal{\'{e}t}}(\infty\cT op)_{/i(X)}$ is precisely $j^\prime$. Noting that the composition $K^\triangleright\xrightarrow{\sim}K\star \{0\}\hookrightarrow K\star \Delta^1\xrightarrow{J}\widehat{\textnormal{Shv}}_{\textnormal{\'{e}t}}(\infty\cT op)_{/i(X)}$ is a colimit diagram, in order to prove our claim it suffices to prove that the composition 

    $$(K^{\textnormal{opp}})^\triangleleft\xrightarrow{\sim}(K\star \{1\})^\textnormal{opp}\hookrightarrow (K\star \Delta^1)^\textnormal{opp}\xrightarrow{J^\textnormal{opp}}(\widehat{\textnormal{Shv}}_{\textnormal{\'{e}t}}(\infty\cT op)_{/i(X)})^\textnormal{opp}\xrightarrow{h_U}\tau_{\leq n}\widehat{\mS},$$

\noindent which we will denote by $\psi^\prime$, is a limit diagram. Write $\phi:\widehat{\textnormal{Shv}}_{\textnormal{\'{e}t}}(\infty\cT op)_{/i(\cX)}\xrightarrow{\sim}\widehat{\textnormal{Shv}}_{\textnormal{\'{e}t}_{/\cX}}(\infty\cT op_{/\cX})$ for the categorical equivalence induced by \ref{Etale Slices of Sheaf Topoi}, and let $\psi_\cX:K^\triangleright\rightarrow (\infty\cT op^{\textnormal{\'{e}t}})_{/\cX}$ be a lifting of $\phi$ along the cone point of $K^\triangleright$. It now follows from \ref{Equivalence of slice topoi explicit Remark} that $\psi^\prime$ may be identified as the composition $(K^{\textnormal{opp}})^\triangleleft=(K^\triangleright)^{\textnormal{opp}}\xrightarrow{\psi_X^{\textnormal{opp}}}(\infty\cT op^{\textnormal{\'{e}t}}_{/\cX})^\textnormal{opp}\hookrightarrow (\infty\cT op_{/\cX})^\textnormal{opp}\xrightarrow{\phi(U)}\tau_{\leq n}\widehat{\mS}$. Since the composition $(\infty\cT op^{\textnormal{\'{e}t}}_{/\cX})^\textnormal{opp}\hookrightarrow (\infty\cT op_{/\cX})^\textnormal{opp}\xrightarrow{\phi(U)}\tau_{\leq n}\widehat{\mS}$ is an $n$-truncated sheaf for the effective epimorphism topology on $\infty\cT op^{\textnormal{\'{e}t}}_{/\cX}\cong \cX$ (\ref{preservation of trunactions 1}), it now follows from \ref{Sheaves on Topoi Result} that $\psi^\prime$ is a limit diagram, as desired.

\end{proof}

\newpage

\section{Appendix}\label{Appendix}

\subsection{Subcategories of Modules}\label{Subcategories of Modules}

\noindent The goal of this section is to prove the following, which we used to prove \ref{Internal Vs External Group Objects Theorem}:

\begin{proposition}{}\label{Localization of Modules}

    Let $\C^\otimes$ be a monoidal $(\infty,1)$-category satisfying the property that $\C$ admits geometric realizations and the tensor product $\otimes:\C\times \C\rightarrow \C$ preserves geometric realizations separately in each variable. Let $\cD\subseteq \C$ be a reflective subcategory such that the localization functor $L:\C\rightarrow \cD$ is compatible with the monoidal structure on $\C$, and let $A\in \textnormal{Alg}(\C)$. Then:

    \begin{itemize}
    
        \item [$(1)$] The map $L^\prime:\textnormal{LMod}_{A}(\C)\rightarrow \textnormal{LMod}_{L(A)}(\cD)$ induced by post-composition with $L^\otimes$ is a reflective localization.

        \item [$(2)$] An object $M\in \textnormal{LMod}_{A}(\C)$ is $L^\prime$-local if and only if the underlying $\C$-object of $M$ is contained in the full subcategory $\cD\subseteq \C$.
        
    \end{itemize}
    
\end{proposition}

\noindent Our proof of \ref{Localization of Modules} will require some preliminaries.

\begin{lemma}{}\label{Tensor product commutes}

    Let $\C^\otimes$ and $\cD^\otimes$ be monoidal $(\infty,1)$-categories whose monoidal structures are each compatible with $\Delta^{\textnormal{opp}}$-indexed colimits. Let $F^\otimes:\C^\otimes\rightarrow \cD^\otimes$ be a monoidal $(\infty,1)$-functor such that $F:\C\rightarrow \cD$ preserves geometric realizations. Then, the square 

    $$
            \begin{tikzpicture}[node distance=2.9cm, auto]
                \node (A) {$ \textnormal{BMod}(\C)\times_{\textnormal{Alg}(\C)} \textnormal{BMod}(\C)$};
                \node (A1) [right of=A] {$ $};
                \node (B) [right of=A1] {$ \textnormal{BMod}(\cD)\times_{\textnormal{Alg}(\cD)} \textnormal{BMod}(\cD)  $};
                \node (C) [below of=A] {$ \textnormal{BMod}(\C) $};
                \node (D) [below of=B] {$ \textnormal{BMod}(\cD)   $};
                \draw[->] (A) to node {$ F^\otimes\circ $} (B);
                \draw[->] (A) to node [swap] {$  $} (C);
                \draw[->] (C) to node [swap] {$ F^\otimes \circ  $} (D);
                \draw[->] (B) to node {$  $} (D);
            \end{tikzpicture}
        $$
        
    \noindent commutes up to homotopy, where the vertical maps are given by the corresponding relative tensor product functors.

\end{lemma}

\begin{proof}

    Combine \citeq{lurie2017higheralgebra} 4.4.2.8, \citeq{lurie2017higheralgebra} 4.4.3.2, and \citeq{lurie2017higheralgebra} 4.4.3.9. 
    
\end{proof}

\begin{corollary}{}\label{Pres. of cocart arrows for LMod cor}

    Under the same assumptions as \ref{Tensor product commutes}, write $p:\textnormal{LMod}(\C)\rightarrow \textnormal{Alg}(\C)$ and $q:\textnormal{LMod}(\cD)\rightarrow \textnormal{Alg}(\cD)$ for the forgetful functors. Then, the map $\textnormal{LMod}_F(\cdot):\textnormal{LMod}(\C)\rightarrow \textnormal{LMod}(\cD)$ given by post-composition with $F^\otimes$ satisfies the following two properties:

    \begin{itemize}
    
        \item [$(1)$] $\textnormal{LMod}_F(\cdot)$ carries $p$-cartesian arrows to $q$-cartesian arrows. 

        \item [$(2)$] $\textnormal{LMod}_F(\cdot)$ carries $p$-cocartesian arrows to $q$-cocartesian arrows. 
        
    \end{itemize}
    
\end{corollary}

\begin{proof}

 Claim $(1)$ is true in general without any assumptions about the existence of certain geometric realizations, and may be deduced from combining \citeq{lurie2017higheralgebra} 4.2.3.2 with the fact that the compositions $\textnormal{LMod}(\C)\xrightarrow{\textnormal{LMod}_F(\cdot)}\textnormal{LMod}(\cD)\rightarrow \cD$ and $\textnormal{LMod}(\C)\rightarrow \C \xrightarrow{F}\cD$ agree. 

 Noting that a similar proof to $(1)$ shows that $\textnormal{BMod}(\C)\rightarrow \textnormal{BMod}(\cD)$ carries arrows cartesian over $\textnormal{Alg}(\C)\times \textnormal{Alg}(\C)$ to arrows cartesian over $\textnormal{Alg}(\cD)\times \textnormal{Alg}(\cD)$, claim $(2)$ follows from combining \ref{Tensor product commutes} with the respective contructions of $p$-cocartesian arrows and $q$-cocartesian arrows given in the proof of \citeq{lurie2017higheralgebra} 4.6.2.17.

\end{proof}

\begin{proof}(\ref{Localization of Modules}.) Acknowledging the slight abuse of notation, we will prove the analogous claim for the map $\textnormal{LMod}^{\mathbb{A}_\infty}_A(\C)\rightarrow \textnormal{LMod}^{\mathbb{A}_\infty}_{\textnormal{Alg}_{L}(A)}(\cD)$, where we identify $\C^\otimes$ and $\cD^\otimes$ with cocartesian fibrations $P:\C^\otimes\rightarrow \Delta^{\textnormal{opp}}$ and $Q:\cD^\otimes\rightarrow \Delta^{\textnormal{opp}}$, respectively. Consider now the counit map $\mu:\Delta^1\times \cD^\otimes\rightarrow \cD^\otimes$, which is a natural isomorphism from $L^\otimes\circ i$ to $\textnormal{id}_{\cD^\otimes}$, where $i:\cD^\otimes\hookrightarrow \C^\otimes$ is the inclusion. In particular, since the only isomorphisms in $\Delta^{\textnormal{opp}}$ are the identity morphisms, the map $\textnormal{Fun}(\cD^\otimes,\cD^\otimes)\rightarrow \textnormal{Fun}(\cD^\otimes,\Delta^{\textnormal{opp}})$ carries $\mu$ to the identity map $\textnormal{id}_Q$. Consider now the edge of $\textnormal{Fun}(\C^\otimes,\textnormal{Assoc}^\otimes)\times_{\textnormal{Fun}(\cD^\otimes,\textnormal{Assoc}^\otimes)}\textnormal{Fun}(\cD^\otimes,\cD^\otimes)$ given by $(\textnormal{id}_P,\mu)$, which is an isomorphism from $(P,L^\otimes|_{\cD^\otimes})$ to $(P,\textnormal{id}_{\cD^\otimes})$. Next, combine \citeq{lurie2024kerodon} 01TB with (the opposite of) \citeq{lurie2024kerodon} 039R to deduce that $L^\otimes$ is $Q$-right Kan extended from $\cD^\otimes\subseteq \C^\otimes$. Since the map $\textnormal{Fun}(\C^\otimes,\cD^\otimes)\rightarrow \textnormal{Fun}(\C^\otimes,\textnormal{Assoc}^\otimes)\times_{\textnormal{Fun}(\cD^\otimes,\textnormal{Assoc}^\otimes)}\textnormal{Fun}(\cD^\otimes,\cD^\otimes)$ carries $L^\otimes$ to $(P,L^\otimes|_{\cD^\otimes})$, it follows from \citeq{lurie2024kerodon} 030R that there exists an arrow $L^{\prime,\otimes}:\C^\otimes\rightarrow \cD^\otimes$ isomorphic to $L^\otimes$, such that $\textnormal{Fun}(\C^\otimes,\cD^\otimes)\rightarrow \textnormal{Fun}(\C^\otimes,\textnormal{Assoc}^\otimes)\times_{\textnormal{Fun}(\cD^\otimes,\textnormal{Assoc}^\otimes)}\textnormal{Fun}(\cD^\otimes,\cD^\otimes)$ carries $L^{\prime,\otimes}$ to $(P,\textnormal{id}_{\cD^\otimes})$. In particular, $L^{\prime,\otimes}$ is a monoidal $(\infty,1)$-functor, and satisfies the property that $L^{\prime,\otimes}|_{\cD^\otimes}=\textnormal{id}_{\cD^\otimes}$ (as a map of simplicial sets). For the remainder of this proof, we will suppose that $L^\otimes:\C^\otimes\rightarrow \cD^\otimes$ is equal (as a map of simplicial sets) to $L^{\prime,\otimes}$.

Consider now the reflective localization $L^{\prime\prime}:\textnormal{Alg}_{/\Delta^{\textnormal{opp}}}(\C)\rightarrow \textnormal{Alg}_{/\Delta^{\textnormal{opp}}}(\cD)$ given by post-composition with $L^{\otimes}$. Noting that the map $\textnormal{Hom}(A,L(A))\rightarrow \textnormal{Hom}(L(A),L(A))$ is a homotopy equivalence, we may choose some $f:A\rightarrow L(A)$ which is carried to $\textnormal{id}_{L(A)}$ (up to equivalence). Next, combining \ref{Pres. of cocart arrows for LMod cor} with \citeq{lurie2024kerodon} 05J5 implies that the compositions $\textnormal{LMod}^{\AA_\infty}_A(\C)\xrightarrow{L(A)\otimes_A(\cdot)}\textnormal{LMod}^{\AA_\infty}_{L(A)}(\C)\rightarrow \textnormal{LMod}^{\AA_\infty}_{L(A)}(\cD)$ and $\textnormal{LMod}^{\AA_\infty}_A(\C)\rightarrow \textnormal{LMod}^{\AA_\infty}_{L(A)}(\cD)\xrightarrow{L(A)\otimes_{L(A)}(\cdot)}\textnormal{LMod}^{\AA_\infty}_{L(A)}(\cD)$ agree (up to natural isomorphism). Since $f:A\rightarrow L(A)$ was carried to $\textnormal{id}_{L(A)}$ (up to equivalence), it follows that the composition $\textnormal{LMod}^{\AA_\infty}_A(\C)\xrightarrow{L(A)\otimes_A(\cdot)}\textnormal{LMod}^{\AA_\infty}_{L(A)}(\C)\rightarrow \textnormal{LMod}^{\AA_\infty}_{L(A)}(\cD)$ is naturally isomorphic to $L^\prime:\textnormal{LMod}^{\AA_\infty}_{A}(\C)\rightarrow \textnormal{LMod}^{\AA_\infty}_{L(A)}(\cD)$. In particular, $L^\prime$ admits a right adjoint $R^\prime:\textnormal{LMod}^{\AA_\infty}_{L(A)}(\cD)\rightarrow \textnormal{LMod}^{\AA_\infty}_{A}(\C)$ given by the composition $\textnormal{LMod}^{\AA_\infty}_{L(A)}(\cD)\hookrightarrow \textnormal{LMod}^{\AA_\infty}_{L(A)}(\C)\xrightarrow{(\cdot)|_{A}}\textnormal{LMod}^{\AA_\infty}_{L(A)}(\C)$. In order to prove our claim, it now suffices to show the following:

\begin{itemize}

    \item [$(i)$] The map $R^\prime:\textnormal{LMod}^{\AA_\infty}_{L(A)}(\cD)\rightarrow \textnormal{LMod}^{\AA_\infty}_{A}(\C)$ is fully faithful. 

    \item [$(ii)$] Let $\eta:\Delta^1\times \textnormal{LMod}^{\AA_\infty}_{A}(\C)\rightarrow \textnormal{LMod}^{\AA_\infty}_{A}(\C)$ be the unit map which exhibits $L^\prime$ as left adjoint to $R^\prime$. Then, the underlying $\C$-object of some $M\in \textnormal{LMod}^{\AA_\infty}_{A}(\C)$ is in $\cD$ precisely when $\eta_M:\Delta^1\rightarrow \textnormal{LMod}^{\AA_\infty}_{A}(\C)$ is an isomorphism. 
    
\end{itemize}

\noindent We now prove these in order. Let $\gamma^0:\Delta^1\times \textnormal{LMod}^{\AA_\infty}_{L(A)}(\C)\rightarrow \textnormal{LMod}^{\AA_\infty}_{L(A)}(\C)$ be the counit which exhibits $L(A)\otimes_A(\cdot)$ as left adjoint to $(\cdot)|_{A}$, and let $\gamma^1:\Delta^1\times \textnormal{LMod}^{\AA_\infty}_{L(A)}(\cD)\rightarrow \textnormal{LMod}^{\AA_\infty}_{L(A)}(\cD)$ be the counit which exhibits the inclusion $i^\prime:\textnormal{LMod}^{\AA_\infty}_{L(A)}(\cD)\hookrightarrow \textnormal{LMod}^{\AA_\infty}_{L(A)}(\C)$ as right adjoint to $\textnormal{LMod}^{\AA_\infty}_{L(A)}(\C)\rightarrow \textnormal{LMod}^{\AA_\infty}_{L(A)}(\cD)$. Let $\gamma^{0,\prime}$ be the composition 

$$\Delta^1\times \textnormal{LMod}^{\AA_\infty}_{L(A)}(\cD)\xrightarrow{\textnormal{id}\times i^\prime}\Delta^1\times \textnormal{LMod}^{\AA_\infty}_{L(A)}(\C)\xrightarrow{\gamma^0}\textnormal{LMod}^{\AA_\infty}_{L(A)}(\C)\rightarrow \textnormal{LMod}_{L(A)}(\cD),$$

\noindent and let $\gamma:\Delta^1\times \textnormal{LMod}^{\AA_\infty}_{L(A)}(\cD)\rightarrow \textnormal{LMod}^{\AA_\infty}_{L(A)}(\cD)$ be a composition $\gamma^1\circ \gamma^{0,\prime}$. Noting that $\gamma:\Delta^1\times \textnormal{LMod}^{\AA_\infty}_{L(A)}(\cD)\rightarrow \textnormal{LMod}^{\AA_\infty}_{L(A)}(\cD)$ may be identified as a counit exhibiting $R^\prime$ as right adjoint to $L^\prime$, in order to prove $(i)$ it suffices to show that $\gamma$ is a natural isomorphism (\citeq{lurie2024kerodon} 02FF). Noting further that $\gamma^1$ is a natural isomrphism, it suffices to show that $\gamma^{0,\prime}$ is a natural isomorphism. In particular, it suffices to show that for each $N\in \textnormal{LMod}^{\AA_\infty}_{L(A)}(\cD)$, the map $\gamma^0_N:L(A)\otimes_A(i^\prime(N)|_A)\rightarrow N$ in $\textnormal{LMod}^{\AA_\infty}_{L(A)}(\C)$ is carried to an isomorphism by $\textnormal{LMod}^{\AA_\infty}_{L(A)}(\C)\rightarrow \textnormal{LMod}^{\AA_\infty}_{L(A)}(\cD)$. Since \ref{Pres. of cocart arrows for LMod cor} implies that the map $i^\prime(N)|_A\rightarrow L(A)\otimes^\C_A(i^\prime(N)|_A)$ over $f:A\rightarrow L(A)$ is taken to an isomorphism in $\textnormal{LMod}^{\AA_\infty}(\cD)$ by $\textnormal{LMod}^{\AA_\infty}(\C)\rightarrow \textnormal{LMod}^{\AA_\infty}(\cD)$, the proof of \citeq{highertopostheory} 5.2.2.8 implies that it suffices to prove that the cartesian arrow $i^\prime(N)|_A\rightarrow i^\prime(N)$ over $f:A\rightarrow L(A)$ is taken to an isomorphism by $\textnormal{LMod}^{\AA_\infty}(\C)\rightarrow \textnormal{LMod}^{\AA_\infty}(\cD)$, which follows immediately from \ref{Pres. of cocart arrows for LMod cor}. We now show $(ii)$. Since the composition $\textnormal{LMod}^{\AA_\infty}_{L(A)}(\cD)\xrightarrow{R^\prime}\textnormal{LMod}^{\AA_\infty}_{A}(\C)\rightarrow \C$ is canonically equivalent to the composition $\textnormal{LMod}^{\AA_\infty}_{L(A)}(\cD)\rightarrow \cD\hookrightarrow \C$, we have that if the map $\eta_M:\Delta^1\rightarrow \textnormal{LMod}^{\AA_\infty}_{A}(\C)$ of $(ii)$ is an isomorphism then the underlying $\C$-object of $M$ is in $\cD$. Supposing now that the underlying $\C$-object of $M$ is in $\cD$, it suffices to show that the image of $\gamma_M$ in $\textnormal{LMod}^{\AA_\infty}_{L(A)}(\cD)$ is an isomorphism (since $\textnormal{LMod}^{\AA_\infty}_{A}(\C)\rightarrow \C$ is conservative). This follows immediately from the fact that $L^\prime:\textnormal{LMod}^{\AA_\infty}_{A}(\C)\rightarrow \textnormal{LMod}^{\AA_\infty}_{L(A)}(\cD)$ is a reflective localization (by $(i)$), completing the proof of our claim.

\end{proof}

\subsection{(Hyper)sheafification and Change of Universe}\label{(Hyper)sheafification and Change of Universe}

The main goal of this section is to show the following:

\begin{proposition}{}\label{HyperSheafification Ind of Choice of Universe}

Let $\C$ be a small $(\infty,1)$-category equipped with a Grothendieck topology $J$. Write $i:\textnormal{Fun}(\C^{\textnormal{opp}},\mS)\hookrightarrow \textnormal{Fun}(\C^{\textnormal{opp}},\widehat{\mS})$ and  $i^\prime:\textnormal{Shv}_J(\C)\hookrightarrow \widehat{\textnormal{Shv}}_J(\C)$ for the corresponding inclusions, let $L:\textnormal{Fun}(\C^{\textnormal{opp}},\mS)\rightarrow \textnormal{Shv}_J(\C)$ denote the sheafification functor, which we identify as an arrow $L:\textnormal{Fun}(\C^{\textnormal{opp}},\mS)\rightarrow \textnormal{Fun}(\C^{\textnormal{opp}},\mS)$, and write $L^{\textnormal{hyp}}:\textnormal{Shv}_J(\C)\rightarrow \textnormal{Shv}_J(\C)^{\textnormal{hyp}}$ for the map left adjoint to the inclusion, which we identify as an arrow $L^{\textnormal{hyp}}:\textnormal{Shv}_J(\C)\rightarrow \textnormal{Shv}_J(\C)$. For each $n\in \ZZ$ let $\tau_{\leq n}:\textnormal{Shv}_J(\C)\rightarrow\tau_{\leq n}\textnormal{Shv}_J(\C)$ be the corresponding truncation functor, which we identify as an arrow $\tau_{\leq n}:\textnormal{Shv}_J(\C)\rightarrow \textnormal{Shv}_J(\C)$. Similarly define the maps $L^\prime:\textnormal{Fun}(\C^{\textnormal{opp}},\widehat{\mS})\rightarrow \widehat{\textnormal{Shv}}_J(\C)\subseteq \textnormal{Fun}(\C^{\textnormal{opp}},\widehat{\mS})$, $L^{\prime,\textnormal{hyp}}:\widehat{\textnormal{Shv}}_J(\C)\rightarrow\widehat{\textnormal{Shv}}_J(\C)^\textnormal{hyp}\subseteq \widehat{\textnormal{Shv}}_J(\C)$ and $\tau^\prime_{\leq n}:\widehat{\textnormal{Shv}}_J(\C)\rightarrow \tau_{\leq n}\widehat{\textnormal{Shv}}_J(\C)\subseteq \widehat{\textnormal{Shv}}_J(\C)$. Then: 

    \begin{itemize}
    
        \item [$(1)$] $i^\prime \circ L\cong L^\prime\circ i$ as maps $\textnormal{Fun}(\C^{\textnormal{opp}},\mS)\rightarrow \textnormal{Fun}(\C^{\textnormal{opp}},\widehat{\mS})$.

        \item [$(2)$] $i^\prime\circ \tau_{\leq n} \cong \tau^\prime_{\leq n}\circ i^\prime$ as maps $\textnormal{Shv}_J(\C)\rightarrow \widehat{\textnormal{Shv}}_J(\C)$. 

        \item [$(3)$] $i^\prime\circ L^{\textnormal{hyp}}\cong L^{\prime,\textnormal{hyp}}\circ i^\prime$ as maps $\textnormal{Shv}_J(\C)\rightarrow\widehat{\textnormal{Shv}}_J(\C)$.
        
    \end{itemize}

\end{proposition}

\begin{remark}{}

   A proof of \ref{HyperSheafification Ind of Choice of Universe} (1) is given in Lurie's \citeq{lurie2009derivedalgebraicgeometryv} (see \citeq{lurie2009derivedalgebraicgeometryv} 2.4.10). At the time of writing this paper, the results \ref{HyperSheafification Ind of Choice of Universe} (2) and \ref{HyperSheafification Ind of Choice of Universe} (3) do not appear to be present in the literature. 
    
\end{remark}

\noindent Our proof will require some preliminaries.

\begin{notation}{}

    Let $\Omega$ be a strongly inaccessible cardinal, and let $\kappa$ be a regular cardinal which is $\Omega$-small. We will use the following notation:

    \begin{itemize}
        \item [$\bullet$] An $(\infty,1)$-category $\cD$ will be called \emph{$(\Omega,\kappa)$-accessible} if $\cD$ is $\kappa$-accessible with respect to the universe defined by $\Omega$.

        \item [$\bullet$] An $(\Omega,\kappa)$-accessible $(\infty,1)$-category $\cD$ will be called an \emph{$(\Omega,\kappa)$-presentable} if $\cD$ admits $\Omega$-small colimits.

        \item [$\bullet$] For any essentially $\Omega$-small $(\infty,1)$-category $\C$, we will write $\C\rightarrow \textnormal{Ind}_\kappa^{\Omega}(\C)$ for the $\textnormal{Ind}_\kappa$-completion of $\C$ with respect to the universe defined by $\Omega$. More generally, we will call any $(\infty,1)$-functor with source $\C$ isomorphic to $\C\rightarrow \textnormal{Ind}_\kappa^{\Omega}(\C)$ an \emph{$(\Omega,\kappa)$-Ind completion} of $\C$. 

        \item [$\bullet$] Let $\cD$ be an $(\infty,1)$-category which admits $\Omega$-small $\kappa$-filtered colimits. Then, we will call an object $X\in \cD$ \emph{$(\Omega,\kappa)$-compact} if $\textnormal{Hom}_\cD(X,-):\cD\rightarrow \mS_{<\Omega}$ preserves $\Omega$-small $\kappa$-filtered colimits. 
        
    \end{itemize}

\end{notation}

\begin{lemma}{}\label{Change of universe and Compact Generation of Sheaf Topoi}

    Let $\Omega$ and $\Omega^\prime$ be strongly inaccessible cardinals satisfying $\Omega\leq \Omega^\prime$, and let $\C$ be an $\Omega$-small $(\infty,1)$-category equipped with a Grothendieck topology $J$. Choose an $\Omega$-small uncountable regular cardinal $\kappa$ such that $\textnormal{Shv}^J_{<\Omega}(\C)$ is $(\Omega,\kappa)$-accessible, and the inclusion $\textnormal{Shv}^J_{<\Omega}(\C)\hookrightarrow \textnormal{Fun}(\C^{\textnormal{opp}},\mS_{<\Omega})$ preserves $\Omega$-small $\kappa$-filtered colimits. Write $\textnormal{Shv}^J_{<\Omega}(\C)_{<\kappa}$ for the full subcategory of $\textnormal{Shv}^J_{<\Omega}(\C)$ spanned by the $(\Omega,\kappa)$-compact objects, and similarly define $\textnormal{Shv}^J_{<\Omega^\prime}(\C)_{<\kappa}\subseteq \textnormal{Shv}^J_{<\Omega^\prime}(\C)$. Then:

    \begin{itemize}
        \item [$(1)$] The map $\textnormal{Shv}^J_{<\Omega}(\C)_{<\kappa}\hookrightarrow \textnormal{Shv}^J_{<\Omega^\prime}(\C)$ exhibits  $\textnormal{Shv}^J_{<\Omega^\prime}(\C)$ as the $(\Omega^\prime,\kappa)$-Ind completion of $\textnormal{Shv}^J_{<\Omega}(\C)_{<\kappa}$. 
        
        \item [$(2)$] The inclusion $\textnormal{Shv}^J_{<\Omega}(\C)_{<\kappa}\hookrightarrow \textnormal{Shv}^J_{<\Omega^\prime}(\C)$ yields a categorical equivalence $\textnormal{Shv}^J_{<\Omega}(\C)_{<\kappa}\xrightarrow{\sim} \textnormal{Shv}^J_{<\Omega^\prime}(\C)_{<\kappa}$, where $\textnormal{Shv}^J_{<\Omega^\prime}(\C)_{<\kappa}\subseteq \textnormal{Shv}^J_{<\Omega^\prime}(\C)$ denotes the full subcategory of $\textnormal{Shv}^J_{<\Omega^\prime}(\C)$ spanned by the $(\Omega^\prime,\kappa)$-compact objects.

    \end{itemize}
    
\end{lemma}

\begin{proof}

    We first prove (1). To begin, we observe that since $\textnormal{Shv}^J_{<\Omega}(\C)\hookrightarrow \textnormal{Fun}(\C^{\textnormal{opp}},\mS_{<\Omega})$ preserves $\Omega$-small $\kappa$-filtered colimits, the left adjoint $L:\textnormal{Fun}(\C^{\textnormal{opp}},\mS)\rightarrow \textnormal{Shv}^J_{<\Omega}(\C)$ preserves $(\Omega,\kappa)$-compact objects (\citeq{highertopostheory} 5.5.7.2). It follows that the canonical map $\C\rightarrow \textnormal{Shv}^J_{<\Omega}(\C)$ factors as $\C\rightarrow \textnormal{Shv}_{<\Omega}^J(\C)_{<\kappa}\hookrightarrow \textnormal{Shv}_{<\Omega}^J(\C)$. In particular, combining \citeq{lurie2024kerodon} 06NW with \ref{Small error in 1.3.1.7} implies that for all $(\infty,1)$-categories $\cE$ which admit $\Omega$-small limits, the map $\textnormal{Fun}^{\ell im_{<\kappa}}((\textnormal{Shv}_{<\Omega}^J(\C)_{<\kappa})^{\textnormal{opp}},\cE)\rightarrow \textnormal{Fun}(\C^{\textnormal{opp}},\cE)$ is fully faithful, with essential image $\textnormal{Shv}_{\cE}^J(\C)\subseteq \textnormal{Fun}(\C^{\textnormal{opp}},\cE)$. Let $\cD$ be an $(\infty,1)$-category which admits $\Omega^\prime$-small limits. Applying  \ref{Small error in 1.3.1.7}, we deduce that the arrow $\textnormal{Fun}^{\ell im_{<\Omega^\prime}}(\textnormal{Shv}^J_{<\Omega^\prime}(\C)^{\textnormal{opp}},\cD)\rightarrow \textnormal{Fun}(\C^{\textnormal{opp}},\cD)$ is fully faithful, with essential image $\textnormal{Shv}_{\cD}^J(\C)\subseteq \textnormal{Fun}(\C^{\textnormal{opp}},\cD)$. Applying \ref{Small Sheaves are stable under small limits}, we observe that $\textnormal{Shv}_{<\Omega}^J(\C)_{<\kappa}\hookrightarrow \textnormal{Shv}^J_{<\Omega^\prime}(\C)$ preserves $\kappa$-small colimits. In particular, the map $\textnormal{Fun}^{\ell im_{<\Omega^\prime}}(\textnormal{Shv}^J_{<\Omega^\prime}(\C)^{\textnormal{opp}},\cD)\rightarrow \textnormal{Fun}(\C^{\textnormal{opp}},\cD)$ factors as 

    $$\textnormal{Fun}^{\ell im_{<\Omega^\prime}}(\textnormal{Shv}^J_{<\Omega^\prime}(\C)^{\textnormal{opp}},\cD)\rightarrow \textnormal{Fun}^{\ell im_{<\kappa}}((\textnormal{Shv}_{<\Omega}^J(\C)_{<\kappa})^{\textnormal{opp}},\cD)\rightarrow \textnormal{Fun}(\C^{\textnormal{opp}},\cD).$$

    \noindent It follows that $\textnormal{Fun}^{\ell im_{<\Omega^\prime}}(\textnormal{Shv}^J_{<\Omega^\prime}(\C)^{\textnormal{opp}},\cD)\rightarrow \textnormal{Fun}^{\ell im_{<\kappa}}((\textnormal{Shv}_{<\Omega}^J(\C)_{<\kappa})^{\textnormal{opp}},\cD)$ is a categorical equivalence. Since $\textnormal{Fun}^{\ell im_{<\Omega^\prime}}(\textnormal{Ind}_\kappa^{\Omega^\prime}(\textnormal{Shv}^J_{<\Omega}(\C)_{<\kappa})^{\textnormal{opp}},\cD)\rightarrow \textnormal{Fun}^{\ell im_{<\kappa}}((\textnormal{Shv}_{<\Omega}^J(\C)_{<\kappa})^{\textnormal{opp}},\cD)$ is also an equivalence, the canonical map $\textnormal{Ind}_\kappa^{\Omega^\prime}(\textnormal{Shv}^J_{<\Omega}(\C)_{<\kappa})^{\textnormal{opp}}\rightarrow \textnormal{Shv}^J_{<\Omega^\prime}(\C)^{\textnormal{opp}}$ right Kan extended from $(\textnormal{Shv}^J_{<\Omega}(\C)_{<\kappa})^{\textnormal{opp}}$ is an equivalence, proving (1). Claim (2) now follows from combining (1) with the fact that $\textnormal{Shv}^J_{<\Omega^\prime}(\C)_{<\kappa}$ is idempotent complete (\citeq{highertopostheory} 5.4.2.4).

   ,

\end{proof}

\begin{lemma}{}\label{Small Sheaves are stable under small limits}

    Let $\Omega$ be a strongly inaccessible cardinal, and let $\Omega^\prime$ be a strongly inaccessible cardinal satisfying $\Omega^\prime\geq \Omega$. Let $\C$ be an $\Omega$-small $(\infty,1)$-category equipped with a Grothendieck topology $J$. Then, the inclusion $i:\textnormal{Shv}^J_{<\Omega}(\C)\hookrightarrow \textnormal{Shv}^J_{<\Omega^\prime}(\C)$ preserves all $\Omega$-small limits and all $\Omega$-small colimits.
    
\end{lemma}

\begin{proof}

    The piece of our claim concerning $\Omega$-small limits is immediate. The case of $\Omega$-small colimits may be deduced from \citeq{lurie2009derivedalgebraicgeometryv} 2.4.10, though to avoid circular reasoning in our proof of \ref{HyperSheafification Ind of Choice of Universe} $(1)$ we include here an alternative argument. \citeq{highertopostheory} 6.3.5.17 establishes the fact that the $(\infty,1)$-functor $\textnormal{Fun}^{\ell im_{<\Omega}}(\textnormal{Shv}^J_{<\Omega}(\C)^{\textnormal{opp}},\mS_{<\Omega})\hookrightarrow \textnormal{Fun}^{\ell im_{<\Omega}}(\textnormal{Shv}^J_{<\Omega}(\C)^{\textnormal{opp}},\mS_{<\Omega^\prime})$ preserves $\Omega$-small colimits, and \ref{Small error in 1.3.1.7} implies that pulling-back along $\C^{\textnormal{opp}}\rightarrow \textnormal{Shv}^J_{<\Omega}(\C)^{\textnormal{opp}}$ yields equivalences $\textnormal{Fun}^{\ell im_{<\Omega}}(\textnormal{Shv}^J_{<\Omega}(\C)^{\textnormal{opp}},\mS_{<\Omega})\xrightarrow{\sim}\textnormal{Shv}^J_{<\Omega}(\C)$ and $\textnormal{Fun}^{\ell im_{<\Omega}}(\textnormal{Shv}^J_{<\Omega}(\C)^{\textnormal{opp}},\mS_{<\Omega^\prime})\xrightarrow{\sim}\textnormal{Shv}^J_{<\Omega^\prime}(\C)$. Our claim now follows from the fact that under these equivalences the inclusion $\textnormal{Fun}^{\ell im_{<\Omega}}(\textnormal{Shv}^J_{<\Omega}(\C)^{\textnormal{opp}},\mS_{<\Omega})\hookrightarrow \textnormal{Fun}^{\ell im_{<\Omega}}(\textnormal{Shv}^J_{<\Omega}(\C)^{\textnormal{opp}},\mS_{<\Omega^\prime})$ may be identified as the inclusion $i:\textnormal{Shv}^J_{<\Omega}(\C)\hookrightarrow \textnormal{Shv}^J_{<\Omega^\prime}(\C)$.

\end{proof}

\begin{lemma}{}\label{kappa accessible functors}

    Fix a small regular cardinal $\kappa$, let $\C$ be an $(\infty,1)$-category which admits small $\kappa$-filtered colimits, and let $\cD$ be a $\kappa$-accessible $(\infty,1)$-category which admits $\kappa$-small limits. Then, the full subcategory of $\textnormal{Fun}(\C,\cD)_\kappa$ of $\textnormal{Fun}(\C,\cD)$ spanned by maps $F:\C\rightarrow \cD$ which preserve small $\kappa$-filtered colimits is closed under $\kappa$-small limits. 
    
\end{lemma}

\begin{proof}

    It suffices to prove that, for all small $\kappa$-filtered $(\infty,1)$-categories $K$, the full subcategory of $\textnormal{Fun}(K^\triangleright,\cD)$ spanned by the colimit diagrams $F:K^\triangleright\rightarrow \cD$ is closed under $\kappa$-small limits. Write $\cD_{<\kappa}$ for the full subcategory of $\cD$ spanned by the $\kappa$-compact objects, and let $K$ be a small $\kappa$-filtered $(\infty,1)$-category. Since the fully faithful map $\cD\rightarrow \textnormal{Fun}((\cD_{<\kappa})^{\textnormal{opp}},\mS)$ preserves small $\kappa$-filtered colimits and all existing limits, it suffices to prove that the full subcategory of $\textnormal{Fun}(K^\triangleright,\textnormal{Fun}((\cD_{<\kappa})^{\textnormal{opp}},\mS))$ spanned by the colimit diagrams $F^\prime:K^\triangleright\rightarrow \textnormal{Fun}((\cD_{<\kappa})^{\textnormal{opp}},\mS)$ is closed under $\kappa$-small limits. This follows immediately from \citeq{lurie2024kerodon} 05Y0.
    
\end{proof}

\begin{lemma}{}[]\label{kappa small limits in PrL,kappa}

    Fix a small uncountable regular cardinal $\kappa$, and write $\textnormal{Pr}^{L,\kappa}$ for the subcategory of $\textnormal{Pr}^{L}$ spanned by those $\kappa$-presentable $(\infty,1)$-categories $\C$ and morphisms $F:\C\rightarrow \cD$ which preserve $\kappa$-compact objects. Write $\textnormal{Pr}^{R,\kappa}$ for the subcategory of $\textnormal{Pr}^{R}$ spanned by the $\kappa$-presentable $(\infty,1)$-categories and $\kappa$-accessible functors, and write $\widehat{\C at}_\infty$ for the $(\infty,1)$-category of (possibly large) $(\infty,1)$-categories. Then, $\textnormal{Pr}^{L,\kappa}\subset \widehat{\C at}_\infty$ is stable under $\kappa$-small limits and $\textnormal{Pr}^{R,\kappa}\subset \widehat{\C at}_\infty$ is stable under small limits. 
    
\end{lemma}

\begin{proof}

    The fact that $\textnormal{Pr}^{L,\kappa}\subset \widehat{\C at}_\infty$ is stable under $\kappa$-small limits follows from combining \citeq{ramzi2024dualizablepresentableinftycategories} 2.31 with \citeq{highertopostheory} 5.5.3.13, and the fact that $\textnormal{Pr}^{R,\kappa}\subset \widehat{\C at}_\infty$ is stable under $\kappa$-small limits is precisely \citeq{highertopostheory} 5.5.7.6. 
    
\end{proof}

\noindent The idea for the following proof is thanks to Jiacheng Liang:

\begin{corollary}{}\label{small hypersheaves are large hypersheaves}

    Under the notation of \ref{Change of universe and Compact Generation of Sheaf Topoi}, the inclusion $\textnormal{Shv}^J_{<\Omega}(\C)\hookrightarrow \textnormal{Shv}^J_{<\Omega^\prime}(\C)$ restricts to a map $\textnormal{Shv}^J_{<\Omega}(\C)^{\textnormal{hyp}}\hookrightarrow \textnormal{Shv}^J_{<\Omega^\prime}(\C)^{\textnormal{hyp}}$. 
    
\end{corollary}

\begin{proof}

     Combining \ref{Change of universe and Compact Generation of Sheaf Topoi} with \citeq{lurie2024kerodon} 06JY, we have that $\textnormal{Fun}(\Delta^1,\textnormal{Shv}^J_{<\Omega^\prime}(\C))$ is $(\Omega^{\prime},\kappa)$-presentable, and the full subcategory of $\textnormal{Fun}(\Delta^1,\textnormal{Shv}^J_{<\Omega^\prime}(\C))_{<\kappa}$ spanned by the $(\Omega^{\prime},\kappa)$-compact objects may be identified as the essential image of $\textnormal{Fun}(\Delta^1,\textnormal{Shv}^J_{<\Omega}(\C)_{<\kappa})\hookrightarrow \textnormal{Fun}(\Delta^1,\textnormal{Shv}^J_{<\Omega^\prime}(\C))$. Writing $\textnormal{Shv}^J_{<\Omega^\prime}(\C)(\infty)$ for the full subcategory of $\textnormal{Fun}(\Delta^1,\textnormal{Shv}^J_{<\Omega^\prime}(\C))$ spanned by the $\infty$-connective morphisms, we now wish to show the following:

      \begin{itemize}
     
         \item [(a)] $\textnormal{Shv}^J_{<\Omega^\prime}(\C)(\infty)$ is $(\Omega^\prime,\kappa)$-accessible, and $\textnormal{Shv}^J_{<\Omega^\prime}(\C)(\infty)\hookrightarrow \textnormal{Fun}(\Delta^1,\textnormal{Shv}^J_{<\Omega^\prime}(\C))$ preserves $(\Omega^\prime,\kappa)$-compact objects. 
         
     \end{itemize}

     \noindent Since a map $f:X\rightarrow Y$ is $\infty$-connective in $\textnormal{Shv}^J_{<\Omega}(\C)$ if and only if the image of $f$ is $\infty$-connective in $\textnormal{Shv}^J_{<\Omega^\prime}(\C)$ (combine \ref{Small Sheaves are stable under small limits} with \citeq{highertopostheory} 6.5.1.18), (a) would immediately imply that for all $\infty$-connective morphisms $f^\prime:X^\prime\rightarrow Y^\prime$ in $\textnormal{Shv}^J_{<\Omega^\prime}(\C)$ there exists an $\Omega^{\prime}$-small $\kappa$-filtered $(\infty,1)$-category $K$ and a diagram $F:K\rightarrow \textnormal{Fun}(\Delta^1,\textnormal{Shv}^J_{<\Omega}(\C)_{<\kappa})\subseteq \textnormal{Fun}(\Delta^1,\textnormal{Shv}^J_{<\Omega^\prime}(\C)_{<\kappa})$ such that for each $k\in K$ we have that $F(k)\in \textnormal{Shv}^J_{<\Omega^\prime}(\C)$, and that there exists a colimit diagram $\overline{F}:K^\triangleright\rightarrow \textnormal{Fun}(\Delta^1,\textnormal{Shv}^J_{<\Omega^\prime}(\C))$ exhibiting $f^\prime$ as the colimit of the composition $K\xrightarrow{F}\textnormal{Fun}(\Delta^1,\textnormal{Shv}^J_{<\Omega}(\C)_{<\kappa})\hookrightarrow \textnormal{Fun}(\Delta^1,\textnormal{Shv}^J_{<\Omega^\prime}(\C))$. Thus, in order to prove our claim it suffices to show (a). We will deduce (a) from the following:

     \begin{itemize}
     
         \item [(b)] Let $n\in \ZZ_{\geq 0}$, and write $\textnormal{Shv}^J_{<\Omega^\prime}(\C)(n)$ for the full subcategory of $\textnormal{Fun}(\Delta^1,\textnormal{Shv}^J_{<\Omega^\prime}(\C))$ spanned by the $n$-connective morphisms. Then, $\textnormal{Shv}^J_{<\Omega^\prime}(\C)(n)$ is $(\Omega^\prime,\kappa)$-presentable, and the inclusion $\textnormal{Shv}^J_{<\Omega^\prime}(\C)(n)\hookrightarrow \textnormal{Fun}(\Delta^1,\textnormal{Shv}^J_{<\Omega^\prime}(\C))$ preserves $\Omega^\prime$-small colimits and $(\Omega^\prime,\kappa)$-compact objects. 
         
     \end{itemize}

     \noindent Temporarily assuming (b), write $U:\ZZ_{\geq 0}^{\textnormal{opp}}\rightarrow \textnormal{Pr}^{L,\Omega^\prime}_\kappa$ for the diagram 
     
     $$\cdots \hookrightarrow \textnormal{Shv}^J_{<\Omega^\prime}(\C)(2)\hookrightarrow\textnormal{Shv}^J_{<\Omega^\prime}(\C)(1)\hookrightarrow \textnormal{Shv}^J_{<\Omega^\prime}(\C)(0),$$
     
     \noindent where $\textnormal{Pr}^{L,\Omega^\prime}_\kappa$ denotes the $(\infty,1)$-category of $(\Omega^\prime,\kappa)$-presentable $(\infty,1)$-categories. Combining the fact that each transition map $\textnormal{Shv}^J_{<\Omega^\prime}(\C)(n+1)\hookrightarrow\textnormal{Shv}^J_{<\Omega^\prime}(\C)(n)$ is an isofibration with \citeq{lurie2024kerodon} 034E, \citeq{lurie2024kerodon} 03BP, and \citeq{lurie2024kerodon} 03BU, we deduce that the limit of $U$ (taken in the $(\infty,1)$-category of $\Omega^\prime$-small $(\infty,1)$-categories) is precisely $\textnormal{Shv}^J_{<\Omega^\prime}(\C)(\infty)$. We may now combine (b) and \ref{kappa small limits in PrL,kappa} to deduce that $\textnormal{Shv}^J_{<\Omega^\prime}(\C)(\infty)$ is $(\Omega^\prime,\kappa)$-accessible, and that the inclusion $\textnormal{Shv}^J_{<\Omega^\prime}(\C)(\infty)\hookrightarrow \textnormal{Fun}(\Delta^1,\textnormal{Shv}^J_{<\Omega^\prime}(\C))$ preserves $(\Omega^\prime,\kappa)$-compact objects, proving (a). We now show (b). Let $n\in \ZZ_{\geq 0}$, and write $\textnormal{Shv}^J_{<\Omega^\prime}(\C)(\tau,n)$ for the full subcategory of $\textnormal{Fun}(\Delta^1,\textnormal{Shv}^J_{<\Omega^\prime}(\C))$ spanned by the $n$-truncated morphisms. Write $L_n:\textnormal{Fun}(\Delta^1,\textnormal{Shv}^J_{<\Omega^\prime}(\C))\rightarrow \textnormal{Shv}^J_{<\Omega^\prime}(\C)(\tau,n)$ for the map left adjoint to the inclusion (\citeq{highertopostheory} 6.5.2.6), and observe that $\textnormal{Shv}^J_{<\Omega^\prime}(\C)(n)\hookrightarrow \textnormal{Fun}(\Delta^1,\textnormal{Shv}^J_{<\Omega^\prime}(\C))$ may be identified as the homotopy pullback of the constant functor $\textnormal{Shv}^J_{<\Omega^\prime}(\C)\rightarrow \textnormal{Shv}^J_{<\Omega^\prime}(\C)(\tau,n)$ along $L_n:\textnormal{Fun}(\Delta^1,\textnormal{Shv}^J_{<\Omega^\prime}(\C))\rightarrow \textnormal{Shv}^J_{<\Omega^\prime}(\C)(\tau,n)$ (see the proof of \citeq{highertopostheory} 6.5.2.7). Since the constant functor $\textnormal{Shv}^J_{<\Omega^\prime}(\C)\rightarrow \textnormal{Shv}^J_{<\Omega^\prime}(\C)(\tau,n)$ preserves $\Omega^\prime$-small colimits, a second application of \ref{kappa small limits in PrL,kappa} implies that in order to prove (b), it suffices to prove the following claim:

     \begin{itemize}
     
         \item [(c)] For each $n\in \ZZ_{\geq 0}$, the $(\infty,1)$-category $\textnormal{Shv}^J_{<\Omega^\prime}(\C)(\tau,n)$ is $(\Omega^\prime,\kappa)$-presentable. Moreover, both the diagonal map $\textnormal{Shv}^J_{<\Omega^\prime}(\C)\rightarrow \textnormal{Shv}^J_{<\Omega^\prime}(\C)(\tau,n)$ and the map $L_n:\textnormal{Fun}(\Delta^1,\textnormal{Shv}^J_{<\Omega^\prime}(\C))\rightarrow \textnormal{Shv}^J_{<\Omega^\prime}(\C)(\tau,n)$ preserve $(\Omega^\prime,\kappa)$-compact objects. 
         
     \end{itemize}

     \noindent Noting that $\textnormal{Shv}^J_{<\Omega^\prime}(\C)(\tau,n)$ and the map $L_n:\textnormal{Fun}(\Delta^1,\textnormal{Shv}^J_{<\Omega^\prime}(\C))\rightarrow \textnormal{Shv}^J_{<\Omega^\prime}(\C)(\tau,n)$ are well-defined for $n\in \ZZ_{\geq -2}$ (\citeq{highertopostheory} 6.5.2.6), we will prove a stronger version of (c) for $n\in \ZZ_{\geq -2}$. For each natural transformation $\alpha:\Delta^1\rightarrow \textnormal{Fun}(\textnormal{Fun}(\Delta^1,\textnormal{Shv}^J_{<\Omega^\prime}(\C)),\textnormal{Shv}^J_{<\Omega^\prime}(\C))$, write $F_\alpha:\textnormal{Fun}(\Delta^1,\textnormal{Shv}^J_{<\Omega^\prime}(\C))\rightarrow \textnormal{Fun}(\Delta^1,\textnormal{Shv}^J_{<\Omega^\prime}(\C))$ for the map adjoint to $\alpha$. For each such $\alpha$ write $F_\alpha^\delta:\textnormal{Fun}(\Delta^1,\textnormal{Shv}^J_{<\Omega^\prime}(\C))\rightarrow \textnormal{Fun}(\Delta^1,\textnormal{Shv}^J_{<\Omega^\prime}(\C))$ for $F_{\delta_\alpha}$, where $\delta_\alpha$ is the diagonal of $\alpha$. Observe now that whenever the map $F_\alpha:\textnormal{Fun}(\Delta^1,\textnormal{Shv}^J_{<\Omega^\prime}(\C))\rightarrow \textnormal{Fun}(\Delta^1,\textnormal{Shv}^J_{<\Omega^\prime}(\C))$ preserves $\Omega^\prime$-small limits and is $(\Omega^\prime,\kappa)$-accessible, we also have that the map $F_{\alpha}^\delta:\textnormal{Fun}(\Delta^1,\textnormal{Shv}^J_{<\Omega^\prime}(\C))\rightarrow \textnormal{Fun}(\Delta^1,\textnormal{Shv}^J_{<\Omega^\prime}(\C))$ preserves $\Omega^\prime$-small limits and is $(\Omega^\prime,\kappa)$-accessible (\ref{kappa accessible functors}). Noting that the diagonal map $\textnormal{Shv}^J_{<\Omega^\prime}(\C)\rightarrow \textnormal{Fun}(\Delta^1,\textnormal{Shv}^J_{<\Omega^\prime}(\C))$ preserves $\Omega^\prime$-small limits and is $(\Omega^\prime,\kappa)$-accessible, whenever $F_\alpha:\textnormal{Fun}(\Delta^1,\textnormal{Shv}^J_{<\Omega^\prime}(\C))\rightarrow \textnormal{Fun}(\Delta^1,\textnormal{Shv}^J_{<\Omega^\prime}(\C))$ preserves $\Omega^\prime$-small limits and is $(\Omega^\prime,\kappa)$-accessible, the homotopy pullback $\cE\rightarrow \textnormal{Fun}(\Delta^1,\textnormal{Shv}^J_{<\Omega^\prime}(\C))$ of $\textnormal{Shv}^J_{<\Omega^\prime}(\C)\rightarrow \textnormal{Fun}(\Delta^1,\textnormal{Shv}^J_{<\Omega^\prime}(\C))$ along $F^\delta_\alpha:\textnormal{Fun}(\Delta^1,\textnormal{Shv}^J_{<\Omega^\prime}(\C))\rightarrow \textnormal{Fun}(\Delta^1,\textnormal{Shv}^J_{<\Omega^\prime}(\C))$ also preserves $\Omega^\prime$-small limits and is $(\Omega^\prime,\kappa)$-accessible (\ref{kappa small limits in PrL,kappa}). Next, set $\alpha_{-2}$ to be the edge $\Delta^1\rightarrow \textnormal{Fun}(\textnormal{Fun}(\Delta^1,\textnormal{Shv}^J_{<\Omega^\prime}(\C)),\textnormal{Shv}^J_{<\Omega^\prime}(\C))$ where $F_{\alpha_{-2}}:\textnormal{Fun}(\textnormal{Fun}(\Delta^1,\textnormal{Shv}^J_{<\Omega^\prime}(\C))\rightarrow \textnormal{Fun}(\textnormal{Fun}(\Delta^1,\textnormal{Shv}^J_{<\Omega^\prime}(\C))$ is the identity map. For each $n\in \ZZ_{\geq -2}$, inductively define $\alpha_{n+1}$ to be the diagonal of $\alpha_{n}$. Since the homotopy pullback $\cE_n\rightarrow \textnormal{Fun}(\textnormal{Fun}(\Delta^1,\textnormal{Shv}^J_{<\Omega^\prime}(\C))$ of $F_{\alpha_{n}}:\textnormal{Fun}(\textnormal{Fun}(\Delta^1,\textnormal{Shv}^J_{<\Omega^\prime}(\C))\rightarrow \textnormal{Fun}(\textnormal{Fun}(\Delta^1,\textnormal{Shv}^J_{<\Omega^\prime}(\C))$ along the diagonal map $\textnormal{Shv}^J_{<\Omega^\prime}(\C)\rightarrow \textnormal{Fun}(\Delta^1,\textnormal{Shv}^J_{<\Omega^\prime}(\C))$ is fully faithful (\citeq{lurie2024kerodon} 046R), we may identify $\cE_n\rightarrow \textnormal{Fun}(\textnormal{Fun}(\Delta^1,\textnormal{Shv}^J_{<\Omega^\prime}(\C))$ as the inclusion of the full subcategory of $\textnormal{Fun}(\textnormal{Fun}(\Delta^1,\textnormal{Shv}^J_{<\Omega^\prime}(\C))$ spanned by those arrows $f:X\rightarrow Y$ in $\textnormal{Shv}^J_{<\Omega^\prime}(\C)$ such that the $(n+2)$ iteration of the diagonal of $f$ is an isomorphism, which is precisely the inclusion $\textnormal{Shv}^J_{<\Omega^\prime}(\C)(\tau,n)\hookrightarrow \textnormal{Fun}(\textnormal{Fun}(\Delta^1,\textnormal{Shv}^J_{<\Omega^\prime}(\C))$ (\citeq{highertopostheory} 5.5.6.15). In particular, the inclusion $\textnormal{Shv}^J_{<\Omega^\prime}(\C)(\tau,n)\hookrightarrow \textnormal{Fun}(\textnormal{Fun}(\Delta^1,\textnormal{Shv}^J_{<\Omega^\prime}(\C))$ preserves $\Omega^\prime$-small limits and $(\Omega^\prime,\kappa)$-compact objects, so \citeq{highertopostheory} 5.5.7.2 implies that $L_n:\textnormal{Fun}(\textnormal{Fun}(\Delta^1,\textnormal{Shv}^J_{<\Omega^\prime}(\C))\rightarrow \textnormal{Shv}^J_{<\Omega^\prime}(\C)(\tau,n)$ preserves $(\Omega^\prime,\kappa)$-compact objects, which in turn also implies that $\textnormal{Shv}^J_{<\Omega^\prime}(\C)(\tau,n)$ is $(\Omega^\prime,\kappa)$-presentable (\citeq{highertopostheory} 5.5.7.3). Finally, since the diagonal map $\textnormal{Shv}^J_{<\Omega^\prime}(\C)\rightarrow \textnormal{Shv}^J_{<\Omega^\prime}(\C)(\tau,n)$ factors as $\textnormal{Shv}^J_{<\Omega^\prime}(\C)\rightarrow  \textnormal{Fun}(\Delta^1,\textnormal{Shv}^J_{<\Omega^\prime}(\C))\xrightarrow{L_n}\textnormal{Shv}^J_{<\Omega^\prime}(\C)(\tau,n)$ where $\textnormal{Shv}^J_{<\Omega^\prime}(\C)\rightarrow \textnormal{Fun}(\Delta^1,\textnormal{Shv}^J_{<\Omega^\prime}(\C))$ is (also) the diagonal map, it follows that said diagonal map $\textnormal{Shv}^J_{<\Omega^\prime}(\C)\rightarrow \textnormal{Shv}^J_{<\Omega^\prime}(\C)(\tau,n)$ preserves $(\Omega^\prime,\kappa)$-compact objects, completing the proof of our claim. 
     
\end{proof}

\begin{proof} (\ref{HyperSheafification Ind of Choice of Universe}.) We begin by proving $(1)$. First, write $T$ for the isomorphism class of monomorphisms of the form $m:S\rightarrow h^\C_X$ in $\textnormal{Fun}(\C^{\textnormal{opp}},\mS_{<\Omega})$, where $X\in \C$ and $m:S\rightarrow h^\C_X$ corresponds to a $J$-sieve on $X$ under the bijection of \citeq{highertopostheory} 6.2.2.5. Noting that the composition $\textnormal{Fun}(\C^{\textnormal{opp}},\mS_{<\Omega})\hookrightarrow \textnormal{Fun}(\C^{\textnormal{opp}},\mS_{<\Omega^\prime})\xrightarrow{L^\prime}\textnormal{Shv}^J_{<\Omega^\prime}(\C)$ preserve small colimits and carries each arrow in $T$ to an isomorphism, it follows from \ref{Univ property of Presentable Localizations} that $\textnormal{Fun}(\C^{\textnormal{opp}},\mS_{<\Omega})\rightarrow \textnormal{Shv}^J_{<\Omega^\prime}(\C)$ carries each $L$-equivalence to an isomorphism. Thus, $i:\textnormal{Fun}(\C^{\textnormal{opp}},\mS_{<\Omega})\hookrightarrow \textnormal{Fun}(\C^{\textnormal{opp}},\mS_{<\Omega^\prime})$ carries $L$-equivalences to $L^\prime$-equivalences. Next, write $\textnormal{Fun}(\C^{\textnormal{opp}},\mS_{<\Omega})^0$ for the essential image of $\textnormal{Fun}(\C^{\textnormal{opp}},\mS_{<\Omega})$ in $\textnormal{Fun}(\C^{\textnormal{opp}},\mS_{<\Omega^\prime})$. Since $i:\textnormal{Fun}(\C^{\textnormal{opp}},\mS_{<\Omega})\hookrightarrow \textnormal{Fun}(\C^{\textnormal{opp}},\mS_{<\Omega^\prime})$ restricts to the map $i^\prime:\textnormal{Shv}^J_{<\Omega}(\C)\rightarrow\textnormal{Shv}^J_{<\Omega^\prime}(\C)$, it immediately follows that $L^\prime:\textnormal{Fun}(\C^{\textnormal{opp}},\mS_{<\Omega^\prime})\rightarrow \textnormal{Fun}(\C^{\textnormal{opp}},\mS_{<\Omega^\prime})$ restricts to a map $\textnormal{Fun}(\C^{\textnormal{opp}},\mS_{<\Omega^\prime})^0\rightarrow \textnormal{Fun}(\C^{\textnormal{opp}},\mS_{<\Omega^\prime})^0$, proving (1).

Claim (2) may be deduced from (1), though here we present an alternative argument. Let $n\in \ZZ_{\geq -2}$. To begin, observe that $i^\prime:\textnormal{Shv}^J_{<\Omega}(\C)\hookrightarrow \textnormal{Shv}^J_{<\Omega^\prime}(\C)$ carries $n$-truncated objects to $n$-truncated objects. So, just as in the proof of (1), it suffices to prove that $i^\prime$ carries $\tau_{\leq n}$-equivalences to $\tau^\prime_{\leq n}$-equivalences. However, it follows from \ref{Small Sheaves are stable under small limits} that $i^\prime$ is a map of $\infty$-pretopoi in the sense of \citeq{lurie2018sag} A.6.4.1, so our claim follows from \citeq{lurie2018sag} A.6.7.1.

We now show (3). Combining \ref{Small Sheaves are stable under small limits} with \citeq{highertopostheory} 6.5.1.18, we have that $i^\prime$ carries $L^{\textnormal{hyp}}$-equivalences to $L^{\prime,\textnormal{hyp}}$-equivalences. Thus, just as in the proof of (1) (and the proof of (2)), it suffices to show that $i^\prime$ carries $\textnormal{Shv}^J_{<\Omega}(\C)^{\textnormal{hyp}}\subseteq \textnormal{Shv}^J_{<\Omega}(\C)$ to $\textnormal{Shv}^J_{<\Omega^\prime}(\C)^{\textnormal{hyp}}\subseteq \textnormal{Shv}^J_{<\Omega^\prime}(\C)$, which we have shown in \ref{small hypersheaves are large hypersheaves}.

\end{proof}

\newpage

\section*{References}
\printbibliography[heading=none]

\addcontentsline{toc}{section}{References}

\vspace{20pt}

\noindent Department of Mathematics, University of Minnesota, Minneapolis, MN, 55455, \emph{email:} bass0161@umn.edu

\end{document}